\documentclass[a4paper,11pt]{article}
\usepackage{mthesis_paper}
\usepackage[english]{babel}
\usepackage{tcolorbox}
\usepackage{authblk}

\title{Optimization of an eigenvalue arising in optimal insulation with a lower bound}
\author[*]{S\"oren Bartels}
\author[**]{Giuseppe Buttazzo}
\author[*]{Hedwig Keller}
\affil[*]{Department for Applied Mathematics, University of Freiburg (Germany)}
\affil[**]{Department of Mathematics, University of Pisa (Italy)}
\begin{document}
\maketitle

\section*{Abstract}

An eigenvalue problem arising in optimal insulation  related to the minimization of the heat decay rate of an insulated body is adapted to enforce a positive lower bound imposed on the distribution of insulating material. We prove the existence of optimal domains among a class of convex shapes and propose a numerical scheme to approximate the eigenvalue. The stability of the shape optimization among convex, bounded domains in $\mathbb{R}^3$ is proven for an approximation with polyhedral domains under a non-conformal convexity constraint. We prove that on the ball, symmetry breaking of the optimal insulation can be expected in general. To observe how the lower bound affects the breaking of symmetry in the optimal insulation and the shape optimization, the eigenvalue and optimal domains are approximated for several values of mass $m$ and lower bounds $\ell_{\min}\ge0$. The numerical experiments suggest, that in general symmetry breaking still arises, unless $m$ is close to a critical value $m_0$, and $\ell_{\min}$ large enough such that almost all of the mass $m$ is fixed through the lower bound. For $\ell_{\min}=0$, the numerical results are consistent with previous numerical experiments on shape optimization restricted to rotationally symmetric, convex domains.\\

{\it 2020 Mathematics Subject Classification:} Primary 49R05, 49Q10, 65N12 secondary: 35J25, 65N25, 52-08.\\
{\it Keywords:} Optimal insulation, symmetry breaking, numerical scheme, shape optimization.\\

\textit{Acknowledgments:} This work is supported by DFG grants BA2268/4-2 within the Priority Program SPP 1962 (Non-smooth and Complementarity-based Distributed Parameter Systems: Simulation and Hierarchical Optimization). The work of GB is part of the project 2017TEXA3H {\it``Gradient flows, Optimal Transport and Metric Measure Structures''} funded by the Italian Ministry of Research and University.

\section{Introduction}

In previous research \cite{BBN17,keller,BB19} an eigenvalue in optimal insulation was considered, which translates to minimizing the heat decay rate of a heat conducting body $\Omega \subset \mathbb{R}^d$ surrounded by an insulating material $\ell:\partial\Omega\to\mathbb{R}_{+}$ of mass $m$, so that $\int_{\partial\Omega}\ell\,\mathrm{d}s=m$.
The eigenvalue
\begin{equation} \label{intro:eigval_ell}
\lambda_\ell(\Omega) = \inf_{u \in H^1(\Omega)} \left\lbrace\int_\Omega \vert \nabla u \vert ^2\, \mathrm{d}x + \int_{\partial \Omega} \frac{u^2}{\ell}\, \mathrm{d}s \ : \ \int_\Omega u^2\, \mathrm{d}x = 1 \right\rbrace
\end{equation}
results from a model reduction for the thickness of the insulating layer (see \cite{acerbi1986reinforcement}). For the eigenvalue a breaking of symmetry was observed, if there was not enough insulating material, both for the optimal distribution of insulating material as well as the optimal domains in a shape optimization problem under a convexity constraint \cite{BBN17,keller,BB19}.

A numerical scheme for this eigenvalue was proposed in \cite{BB19} and the shape optimization among rotationally symmetric convex domains was considered in \cite{keller}. The numerical experiments confirm the breaking of symmetry both for the domain $B_1(0)$ and the approximated optimal domains of the shape optimization problem.

This breaking of symmetry was further investigated in \cite{huang,huang2022stability} and it was shown that for domains which are bounded with a $C^1$ regular boundary or convex, concentration breaking occurs \cite[Theorem 1.7]{huang}, such that a part of the boundary is left uncovered, if $m$ is too small. Analytically not much is known about the shape of optimal domains, except that for low values of $m$ the ball cannot be optimal \cite{BBN17}. Numerical experiments for the shape optimization can be found in \cite{BB19,keller}. In \cite{BB19} a two-dimensional shape optimization is conducted as well as a three-dimensional shape optimization among assembled half ellipsoids. The numerical approximation of the eigenvalue from \cite{BB19} confirms the expected asymmetry. In \cite{keller} rotationally symmetry is assumed, and the shape optimization reduced to a two-dimensional setting. This allows to give more conclusive results to the optimal domains concerning the eigenvalue in optimal insulation. For this optimization problem, even with the assumption of rotational symmetry, the approximated eigenfunctions and optimal domains were consistent with the expectations from prior analytical results, in particular the breaking of symmetry \cite{BBN17}. However, the restriction to rotationally symmetric domains is only justified by simple numerical experiments \cite{BB19}, and even for a rotationally symmetric domain we cannot infer symmetry of the corresponding eigenfunction.

Due to the symmetry breaking, a part of the boundary is left uncovered for small values $m$ of insulating material \cite{BBN17, huang}. Since the insulating material might have additional properties, it is of interest to avoid a concentration breaking and to ensure that the whole boundary is covered by the given material. To guarantee this, we place a lower bound on the distribution of insulating material and restrict the set of admissible distributions to
\begin{equation*}
\mathcal{H}_{\widehat{m},\ell_{\min}}(\partial \Omega):=\left\lbrace\widehat{\ell}:\partial\Omega\to\mathbb{R}\ :\ \widehat{\ell}\ge\ell_{\min},\ \int_{\partial \Omega} \widehat{\ell}\,\mathrm{d} s = \widehat{m}\right\rbrace 
\end{equation*}
for a positive lower bound $\ell_{\min}>0$. 
This corresponds to an adaptation of the results of \cite{BBN17,acerbi1986reinforcement}, but with Robin boundary conditions in the model reduction. This was previously considered for an energy problem in \cite{della2021optimization} and the results can be used to adapt the eigenvalue problem in optimal insulation to an eigenvalue problem with a lower bound placed on the insulating material. For large enough values of mass $\widehat{m}$ for which the optimal distribution is expected to be constant, such a lower bound has no effect. For low values of mass $\widehat{m}$ however, it is not clear, how a lower bound on the optimal insulation affects the breaking of symmetry both in the optimal insulation as well as the shape optimization. We prove that on the ball, symmetry breaking arises in the optimal insulation for $\ell_{\min}>0$ and $\widehat{m} <m_0$ unless all the mass is fixed through $ \ell_{\min}$.

The numerical experiments in Section \ref{sec:num_exp} suggest, that in the shape optimization in general symmetry breaking still arises, unless $\widehat{m}$ is close to a critical value $m_0$, and $\ell_{\min}$ is large enough such that almost all of the mass $\widehat{m}$ is fixed through the lower bound, but that the optimal domains for fixed $\widehat{m}$ transform to the ball as $\ell_{\min}$ increases, and kinks in the boundary of the domains smoothen.

The work is structured as followed. We first summarize the results for the eigenvalue arising in optimal insulation as introduced in \cite{BBN17} and the results of \cite{della2021optimization}, from which we can derive an eigenvalue problem in optimal insulation with a positive lower bound imposed on the distribution of insulating material. For the eigenvalue on the ball, we show that symmetry breaking occurs even for positive $\ell_{\min}$.
We prove the existence of an optimal domain among a class of convex bounded domains with fixed volume in $\mathbb{R}^d, d = 2,3$, and show that the problem is not well-posed in the absence of a convexity constraint.
A numerical scheme for the approximation of the eigenvalue is proposed and the stability of the shape optimization is proven in the framework of \cite{keller2_preprint} for an approximation with a non-conformal polyhedral convexity in $\mathbb{R}^3$, under suitable assumptions on the admissible discrete domains. Lastly, we report numerical experiments and approximated domains for both eigenvalues.

\section{An eigenvalue problem arising in optimal insulation}\label{sec:oi}

Assume that a heat conducting body $\Omega$ is surrounded by a thin layer of insulating material, arranged in such a way that the heat is retained as well as possible. The optimization criterion is the minimization of the decay rate of the temperature, and $m$ is the total amount of the insulating material. The goal is to find the optimal thickness $\ell(\sigma)$ of the insulating layer, for every $\sigma\in\partial\Omega$. As shown in \cite{BBN17} this corresponds to the minimization of the eigenvalue
\begin{equation} \label{eq:rayleigh_oi}
\lambda_{m}(\Omega) = \inf_{u \in H^1(\Omega)}\left\lbrace \int_{\Omega} \vert \nabla u \vert^2 \, \mathrm{d}x + \frac{1}{m} \|u \|_{L^1(\partial \Omega)}^2: \,\Vert u \Vert_{L^2(\Omega)} = 1 \right\rbrace.
\end{equation}

The boundary term corresponds to Robin-type boundary conditions which result from the model reduction for the thickness of the insulating layer $\ell:\partial\Omega\to\mathbb{R}_{+}$, a nonnegative measurable function with total mass $m$. The optimal distribution of insulating material can be reconstructed from a minimal $u$ as shown below in \eqref{eq:oi_optimaldis}.

Let $\Omega_\varepsilon=\Omega\cup\Sigma_\varepsilon$, with
\begin{equation}\label{sigmaeps}
\Sigma_\varepsilon=\{x\ :\ x=\sigma+\varepsilon s\ell(\sigma)\nu(\sigma),\ \sigma\in\partial\Omega,\ s\in(0,1)\},
\end{equation}
with $\varepsilon > 0$, $\nu(\sigma)$ the outer unit normal such that $\Sigma_\varepsilon$ corresponds to the region occupied by insulating material in which the conductivity $\varepsilon$ is significantly smaller than in the domain $\Omega$. \
The energy problem with a given heat source $f$ is described by the functional
\begin{equation}\label{G_eps}
G_\varepsilon (u)= \frac{1}{2} \int_\Omega \vert \nabla u \vert^2 \, \mathrm{d}x + \frac{\varepsilon}{2} \int_{\Sigma_\varepsilon}\vert \nabla u \vert^2 \, \mathrm{d}x - \int_{\Omega} fu \, \mathrm{d}x,
\end{equation}
defined on the Sobolev space $H^1_0(\Omega_\varepsilon)$, which corresponds to the boundary value problem
\begin{equation}\label{G_eps_pde}
\begin{cases}
-\Delta u_\varepsilon=f\quad\text{in }\Omega,\\
-\Delta u_\varepsilon=0\quad\text{in }\Sigma_\varepsilon,\\
u_\varepsilon=0\quad\text{on }\partial\Omega_\varepsilon,\\
\displaystyle\frac{\partial u_\varepsilon^-}{\partial\nu}=\varepsilon\frac{\partial u_\varepsilon^+}{\partial \nu}\quad\text{on }\partial \Omega.
\end{cases}
\end{equation}

It has been shown e.g. in \cite{acerbi1986reinforcement} that the sequence of functionals $G_\varepsilon$ $\Gamma$-converges in the strong $L^2$-topology to
\begin{equation*}
G(u,\ell) = \frac{1}{2}\int_{\Omega} \vert \nabla u \vert^2 \, \mathrm{d}x -\int_{\Omega} fu\, \mathrm{d}x + \frac{1}{2} \int_{\partial \Omega} \frac{u^2}{\ell}\,\mathrm{d}s
\end{equation*}
as $\varepsilon\to0$. Rather than considering the energy problem, we optimize the heat decay rate which is determined by the following eigenvalue problem
\begin{equation} \label{rein:eigval_ell}
\lambda_\ell(\Omega) = \inf_{u \in H^1(\Omega)} \left\lbrace\int_\Omega \vert \nabla u \vert ^2\, \mathrm{d}x + \int_{\partial \Omega} \frac{u^2}{\ell}\, \mathrm{d}s \ : \ \int_\Omega u^2\, \mathrm{d}x = 1 \right\rbrace
\end{equation} 
for an admissible distribution of insulating material
\begin{equation*}
\ell\in\mathcal{H}_m(\partial\Omega):=\left\lbrace\ell\in L^1(\partial\Omega)\ :\ \ell\ge0\text{ and } \int_{\partial\Omega}\ell\,\mathrm{d}s=m\right\rbrace
\end{equation*}
of a given mass $m$. A simple computation yields that, if $u$ is the solution of \eqref{eq:rayleigh_oi}, then the optimal thickness $\ell_{\text{opt}}\in\mathcal{H}_m(\partial \Omega)$ is given by
\begin{equation}\label{eq:oi_optimaldis}
\ell_{\text{opt}}(z) = \frac{m\,|u(z)|}{\int_{\partial\Omega}|u|\,\mathrm{d}s}.
\end{equation}

In \cite{BBN17} a surprising breaking of symmetry for the eigenvalue is observed. 

\begin{theorem}[{\cite[Theorem 3.1]{BBN17}}] \label{thm:sym_breaking}
Let $\Omega$ be a ball. Then there exists $m_0>0$ such that solutions to \eqref{eq:rayleigh_oi} are radial if $m > m_0$, while solutions are not radial for $0 < m <m_0$. As a consequence, the optimal density $\ell_{\mathrm{opt}}$ is not constant if $m <m_0$.
\end{theorem}

In \cite{BBN17} it is further noted that this threshold value $m_0$ is given by the unique positive $m$ for which $\lambda_m = \mu_2$, the first non-zero eigenvalue of the Neumann problem. A numerical estimation of the critical value gives $m_0 \approx 1.8534$ for $B_1(0) \subset \mathbb{R}^2$ and $m_0\approx 5.7963$ for $B_1(0) \subset \mathbb{R}^3$.

A corresponding shape optimization problem for a fixed mass $m>0$ is defined as follows:
\begin{equation}\label{problem_m}
\min\big\{\lambda_m( \Omega)\ :\ \Omega\in\mathcal{C}_V(Q)\big\}\tag{$\mathbf{P_m}$}
\end{equation}
where
$$\mathcal{C}_V(Q)=\left\lbrace\Omega\subset Q\ :\ \Omega\text{ is convex and open in } \mathbb{R}^d \text{ and }|\Omega|=V\right\rbrace$$
being $V>0$ and $Q\subset\mathbb{R}^d$ a convex and compact domain for $d = 2,3$.

\begin{proposition}[{\cite[Proposition 1]{keller}}]\label{thm:oi_ex}
There exists an optimal domain $\Omega$ for the shape optimization problem \eqref{problem_m}.
\end{proposition}

Existence of a solution is proven with compactness results for special functions of bounded variation \cite{BG,BG10} and the compactness results for convex domains \cite{BG97}. In particular, the uniform bound on the trace inequality is important which implies the strong $L^2$-convergence of the eigenfunctions of a minimizing sequence \cite{keller}. Without the convexity constraint, non-existence of an optimal domain can be established by taking a disjoint union of $n$ balls of radius $r_n\to0$, giving that the infimum in problem \eqref{problem_m} vanishes, see \cite{BB19}.

\section{Imposing a lower bound on the optimal insulation}

For the previously defined eigenvalue of optimal insulation a breaking of symmetry can be observed for small values of insulating material, leading to concentration breaking of the protective layer, such that some part of the boundary is not insulated. 
Since the film can have additional purposes other then insulation, such as for stability or against contamination, it appears interesting to look at insulation problems, for which a positive lower bound is imposed on the insulation such that no part of the boundary is left uncovered. Then the admissible distributions of insulating material are restricted to
\begin{equation}\label{eq:dis_lb}
\mathcal{H}_{\widehat{m},\ell_{\min}}(\partial \Omega): = \left\lbrace \widehat{\ell}:\partial\Omega\to\mathbb{R}\ :\ \widehat{\ell} \ge \ell_{\min}, \int_{\partial \Omega} \widehat{\ell}\, \mathrm{d} s = \widehat{m}\right\rbrace 
\end{equation}
for a positive lower bound $\ell_{\min}>0$. 

Instead of seeking the optimal distribution $\widehat{\ell}\in \mathcal{H}_{\widehat{m},\ell_{\min}}(\partial \Omega)$, we partition the insulating layer into a fixed, constant part of thickness $\ell_{\min}$ and a variable part $\ell \in \mathcal{H}_{m}(\partial \Omega)$ such that $ \widehat{\ell} = \ell_{\min} + \ell$, with
$$m:=\widehat{m}-|\partial\Omega|\,\ell_{\min}.$$
A constant distribution of insulating material corresponds to Robin boundary condition. Therefore for a constant, positive lower bound $\ell_{\min} > 0$, an equivalent formulation of this problem can be derived from a similar model reduction for \eqref{G_eps_pde} as in \cite{BBN17}, but with Robin boundary conditions instead of Dirichlet boundary conditions on $\partial\Omega_\varepsilon$, see \eqref{F_eps_pde}. 

This problem and the existence of minimizers was considered in \cite{della2021optimization}, and the results imply that these two approaches are indeed equivalent.

Let $\Omega_\varepsilon=\Omega\cup\Sigma_\varepsilon$, with $\Sigma_\varepsilon$ as in \eqref{sigmaeps}, with $\ell : \partial \Omega \mapsto (0,+\infty)$ a bounded Lipschitz function with a positive lower bound. In \cite{della2021optimization} an energy problem is considered, which seeks a function $u \in H^1(\Omega)$ minimizing for a given heat source $ f \in L^2(\Omega)$ the functional
\begin{equation*}
F_\varepsilon(u) = \frac{1}{2} \int_\Omega \vert \nabla u \vert^2 \, \mathrm{d}x + \frac{\varepsilon}{2} \int_{\Sigma_\varepsilon}\vert \nabla u \vert^2 \, \mathrm{d}x + \frac{\beta}{2} \int_{\partial \Omega_\varepsilon} u^2\, \mathrm{d} \mathcal{H}^{n-1} - \int_{\Omega} fu \, \mathrm{d}x 
\end{equation*}
for a fixed parameter $\beta > 0$. Assuming sufficient regularity, the minimizer of $F_\varepsilon$ satisfies
\begin{equation}\label{F_eps_pde}
\begin{cases}
-\Delta u_\varepsilon=f\quad\text{in }\Omega,\\
-\Delta u_\varepsilon=0\quad\text{in }\Sigma_\varepsilon,\\
\displaystyle\frac{\partial u_\varepsilon}{\partial\nu}+\beta u_\varepsilon=0\quad\text{on }\partial\Omega_\varepsilon,\\
\displaystyle\frac{\partial u_\varepsilon^-}{\partial\nu}=\varepsilon\frac{\partial u_\varepsilon^+}{\partial \nu}\quad\text{on }\partial\Omega.
\end{cases}
\end{equation}
It is proven in \cite{della2021optimization} that as $\varepsilon\to0 $ the sequence of functionals $F_\varepsilon$ $\Gamma$-converges in the strong $L^2$-topology to 
\begin{equation*}
F_\beta(u,\ell) = \frac{1}{2}\int_{\Omega}|\nabla u|^2\,\mathrm{d}x-\int_{\Omega}fu\,\mathrm{d}x + \frac{\beta}{2} \int_{\partial\Omega}\frac{u^2}{1+\beta \ell}\,\mathrm{d}s.
\end{equation*}
This corresponds to having two layers of insulating material, one being described by $\ell$, and one of constant thickness $\beta^{-1}=\ell_{\min}$. The minimizer of the functional $F_\beta(\cdot, \ell) $ satisfies
$$\begin{cases}
-\Delta u=f&\text{in }\Omega\\
\displaystyle(1+\beta \ell)\frac{\partial u}{\partial \nu}+\beta u = 0&\text{on }\partial\Omega.
\end{cases}$$
Similar as to the problem in optimal insulation without a lower bound we seek a solution for the following minimization problem 
\begin{equation}\label{P}
\min\Big\{F_\beta(v,\ell)\ :\ (v,\ell)\in H^1(\Omega)\times\mathcal{H}_{m}(\partial\Omega)\Big\}
\end{equation}
where $\mathcal{H}_{m}(\partial\Omega)$ is the set of admissible distributions
\begin{equation*}
\mathcal{H}_{m}(\partial\Omega)=\left\lbrace\ell\in L^1(\partial\Omega),\ \ell\ge0,\ \int_{\partial\Omega}\ell\,\mathrm{d}s=m\right\rbrace.
\end{equation*}
The following proposition shows that for a given $v\in H^1(\Omega)$ there exists a corresponding optimal distribution.

\begin{proposition}[{\cite[Proposition 4.1]{della2021optimization}}] \label{lemma:lower_bound_optimal_distribution}
Let $\beta>0$, $m>0$ be fixed, let $v \in L^2(\partial \Omega)$, and $h_v \in L^2(\partial \Omega)$ be the function defined by
\begin{equation}\label{rein:optimal_h_u}
h_v(s):=\begin{cases} (c_v\beta)^{-1}\vert v (s)\vert - \beta^{-1} & \text{ if } \vert v(s) \vert \ge c_v, \\ 0 & \text{ otherwise,} \end{cases}
\end{equation}
where $c_v$ is the unique positive constant satisfying
\begin{equation} \label{rein:eq:c_u}
c_v=\big(\big|\{|v|\ge c_v\}\big|+m\beta\big)^{-1} \int_{\{|v|\ge c_v\}}|v(s)|\,\mathrm{d}s.
\end{equation}
In particular $c_v = 0$ if and only if $v= 0$ $\mathcal{H}^{d-1}$-a.e. on $\partial\Omega$. Then $h_v$ is the solution of the minimization problem
\begin{equation*}
\min\left\{\int_{\partial \Omega} \frac{v^2}{\beta^{-1}+\ell}\,\mathrm{d}s\ :\ \ell\in\mathcal{H}_{m}(\partial \Omega)\right\}.
\end{equation*}
\end{proposition}

Knowing the optimal distribution $\ell$ for a given function $v \in H^1(\Omega)$ leads to a proof of the existence of a pair of minimizers.

\begin{theorem}[{\cite[Theorem 4.1]{della2021optimization}}]\label{thm:lower_bound_existence}
Given any $\beta,m > 0$, there exits a couple $(u,h_u) \in H^1(\Omega) \times L^2(\partial \Omega)$, with $h_u \in \mathcal{H}_{m}(\partial \Omega)$, which minimizes \eqref{P}. Moreover,
\begin{equation*}
h_u(s):=\begin{cases}
(c_u\beta)^{-1}|u (s)| - \beta^{-1}&\text{if } |u(s)|\ge c_u,\\
0&\text{otherwise,}
\end{cases}
\end{equation*}
where $c_u$ is the unique positive constant satisfying
\begin{equation}\label{eq:c_u}
c_u=\big(\big|\{|u|\ge c_u\}\big|+m\beta\big)^{-1} \int_{\{|u|\ge c_u\}}|u(s)|\,\mathrm{d}s.
\end{equation}
Furthermore, the couple $(u,h_u)$ is a solution of
$$\begin{cases}
-\Delta u=f&\text{in }\Omega\\
(1+\beta h_u)\partial_\nu u+\beta u=0&\text{on }\partial\Omega,
\end{cases}$$
and is unique if the domain $\Omega$ is connected.
\end{theorem}

The uniqueness is a consequence of a convexity property of the functionals $F_{\beta}$, cf. \cite[Proposition 4.3]{della2021optimization}.

\begin{remark}\rm \label{rein:remark:optimal_dis_lb}
For further analysis and error estimates, the following observations are useful.
\begin{enumerate}
\item The optimal distribution is scaling invariant, in the sense that for $w = av$, it holds that $h_v = h_w$, with $ v \in H^1(\Omega)$ such that $v \ge 0, a > 0$. 
\item Occasionally, it is convenient to write the boundary term as
\begin{equation*}
\int_{\partial \Omega} \frac{u^2}{\beta^{-1}+h_u} \,\mathrm{d}s = \int_{\partial \Omega} G_{c_u}(u) \,\mathrm{d}s
\end{equation*}
with the Lipschitz continuous function
\begin{equation*}
G_{c}(x) = \chi_{\{\vert x \vert < c\}} \beta x^2 + \chi_{\{\vert x \vert \ge c\}} \beta c \vert x\vert
\end{equation*}
where the dependence of $c$ on $u$ is neglected.
\item To approximate the optimal distribution for a given $u \in H^1(\Omega)$ the constant $c_u$ has to be approximated.
We use the observation, that
$$h_u=\ell_{\min}f(u,c_u)\qquad\text{with }f(v,a) = \max(v/a-1,0),$$
which is a bounded, continuous function for $a>0$ and $v \ge 0.$
For $\vert c_1- c_2 \vert \le \delta$, without loss of generality with $0<c_1 \le c_2$, and $u \ge 0$ we compute
\begin{align*}
|h_{u,c_1}-h_{u,c_2}|&=\ell_{\min} \big|\max(u/c_1- 1,0) - \max(u/c_2-1,0)\big|\\
&=\ell_{\min} \big|\chi_{\{u \ge c_2\}}( u/c_1-u/c_2) + \chi_{\{u \in [c_1,c2)\}}(u/c_1-1)\big|\\
&\le\frac{\ell_{\min}\delta}{c_1}\left( \frac{u}{c_1}+1\right)
\end{align*}
which yields an estimate for an error caused by an approximation of the constant $c_u$. Further, for fixed $c >0$ and $u_n\to u$ in $L^2(\partial \Omega)$ we have that
\begin{equation*}
\|h_{u_n,c}-h_{u,c}\|_{L^2(\partial \Omega)} \le (\ell_{\min}/c) \Vert u_n - u \Vert_{L^2(\partial \Omega)}\to0.
\end{equation*}
Since $c_{u} =0$ if and only if $u = 0$ almost everywhere on $\partial \Omega$, for the strongly convergent sequence $u_n\to0$ in $L^1(\partial\Omega)$ the formula \eqref{rein:eq:c_u} implies $0 \le c_{u_n} \le (\ell_{\min} / m) \Vert u_n \Vert_{L^1(\partial\Omega)}\to0$.
\end{enumerate}
\end{remark}

\section{The eigenvalue problem}\label{sec:lb_ev}

In \cite{della2021optimization} an energy problem is considered. Here, we are interested in the corresponding eigenvalue problem.
To minimize the heat decay rate, we consider the functional
\begin{equation} \label{eq:reinf_func_beta}
J_{\ell_{\min}} (u,\ell) = \int_{\Omega} \vert \nabla u \vert^2 \, \mathrm{d}x + \int_{\partial \Omega} \frac{u^2}{\ell_{\min} + \ell} \, \,\mathrm{d}s 
\end{equation}
for $\ell \in \mathcal{H}_{m}(\partial \Omega)$, and look for a minimizer $(u,\ell) \in H^1(\Omega)\times \mathcal{H}_{m}(\partial \Omega)$ of
\begin{equation}\label{rein:prop:eigval}
\lambda_{m,\ell_{\min}}(\Omega)=\inf_{(u,\ell)\in H^1(\Omega)\times\mathcal{H}_m(\partial\Omega)}\left\lbrace J_{\ell_{\min}}(u,\ell),\ \|u\|_{L^2(\Omega)} = 1 \right\rbrace.
\end{equation}
This is equivalent to seeking a pair of minimizers $(u,\widehat{\ell}) \in H^1(\Omega) \times \mathcal{H}_{\widehat{m},\ell_{\min}}(\partial \Omega)$, with the admissible set of distribtions defined by \eqref{eq:dis_lb}, of the functional
\begin{equation} \label{eq:reinf_func_beta:v2}
J(u,\widehat{\ell}) = \int_{\Omega} \vert \nabla u \vert^2 \, \mathrm{d}x + \int_{\partial \Omega} \frac{u^2}{\widehat{\ell}} \, \,\mathrm{d}s.
\end{equation}

Using similar arguments as in \cite{della2021optimization}, it is straightforward to show that the eigenvalue problem is well defined. From \cite{BBN17} it is known that, if $\Omega$ is a ball, breaking of symmetry does not occur if $ m + \ell_{\min} \vert \partial \Omega \vert = \widehat{m} > m_0$, with $m_0$ the critical value of mass related to the Neumann eigenvalue \cite[Theorem 3.1]{BBN17}. Without a lower bound (that is if $\ell_{\min}=0$), symmetry breaking occurs if $m<m_0$, but if $\ell_{\min}>0$, this is no longer obvious.

\begin{proposition}\label{prop:sym_breaking_lb}
Let $\Omega = B_1(0) \subset \mathbb{R}^d$ be the unit ball, and $\widehat{m}<m_0$ the critical value of mass from Theorem \ref{thm:sym_breaking}.
Then there exists no radial symmetric solution to the eigenvalue problem 
\begin{equation}\label{rein:prop:eigval_2}
\lambda_{m, \ell_{\min}} (\Omega)= \min_{0 \neq u \in H^1(\Omega)} \frac{\int_{\Omega} \vert \nabla u \vert^2 \,\mathrm{d}x + \int_{\partial \Omega} u^2(\ell_{\min}+h_u)^{-1}\,\mathrm{d}s}{\int_{\Omega}u^2 \,\mathrm{d}x}
\end{equation}
for $\ell_{\min} \in [0,\widehat{m} / \vert \partial \Omega \vert)$ with the optimal distribution $h_u$ as given by Theorem \ref{thm:lower_bound_existence}.
\end{proposition}

\begin{proof}
Let us proceed step by step.
\begin{enumerate}
\item The case $\ell_{\min} = 0$ corresponds to the eigenvalue problem without a lower bound. The non-existence of a radial solution in the case $\widehat{m} < m_0$ is therefore covered by Theorem \ref{thm:sym_breaking}.
\item In the case $\ell_{\min} >0$, for a radial function $u\in H^1(\Omega)$ the optimal distribution, as given by Theorem \ref{thm:lower_bound_existence}, is constant. Without loss of generality we assume that $u$ is non-negative on $\partial \Omega$. The constant $c_u$, see \eqref{eq:c_u}, is then given by
\begin{equation}\label{eq:cu_radial}
\begin{split}
c_u&=\frac{\ell_{\min}}{\ell_{\min}\vert \partial \Omega \vert+m} \int_{\partial \Omega} u \, \mathrm{d}s = \frac{\ell_{\min}\vert \partial \Omega \vert}{\ell_{\min}\vert \partial \Omega \vert+m}\,\bar{u}\\
&=\frac{{\ell_{\min}\vert \partial \Omega \vert}}{\widehat{m}}\,\bar{u} < u\qquad\text{ a.e. on } \partial \Omega.
\end{split}
\end{equation}
where $m=\widehat{m}-\ell_{\min}|\partial\Omega|$ is the free mass that $h_u$ distributes, and $\bar{u}=(|\partial \Omega|)^{-1}\int_{\partial\Omega} u\,\mathrm{d}s$ is the integral mean of the trace of $u$. Since $u$ is assumed to be radially symmetric $u\equiv\bar{u}$ a.e. on $\partial \Omega$, and therefore $\{ \vert u \vert \ge c_u \} = \partial \Omega$.
The optimal insulation is then given by the constant distribution of insulating material
\begin{equation*}
h_u \equiv \widehat{m}/\vert \partial \Omega \vert  - \ell_{\min}.
\end{equation*}
This implies that for any radial symmetric minimizer to \eqref{rein:prop:eigval_2} the eigenvalue problem becomes equivalent to the Robin eigenvalue problem
\begin{equation}\label{robinev}
\lambda_1(\Omega,\alpha) = \min_{0 \neq u \in H^1(\Omega)} \frac{\int_{\Omega} \vert \nabla u \vert^2 \,\mathrm{d}x + \alpha \int_{\partial \Omega}u^2\,\mathrm{d}s}{\int_{\Omega}u^2 \,\mathrm{d}x}
\end{equation}
for $\alpha = (\widehat{m}/ \vert \partial \Omega \vert)^{-1}$, such that $u \in H^1(\Omega)$ is an eigenfunction to first Robin eigenvalue.
The first Robin eigenvalue is simple and the eigenfunction can be chosen to satisfy $u \ge \delta >0$ on $\overline{\Omega}$ see e.g. the Krein-Rutman theorem or \cite{Arendt}. Without loss of generality we assume that $\Vert u \Vert_{L^2(\Omega)} =1$.
\item Let $v \in H^1(\Omega)$ be a test function which is bounded from below on $\partial \Omega$. 
Due to \eqref{eq:cu_radial} we have $u + \varepsilon v > c_u>0$ for $\varepsilon > 0$ small enough, such that $\{u + \varepsilon v \ge c_u \} = \partial \Omega$. If $c_u < c_{u+\varepsilon v}$, we compute
\begin{align*}
0<c_{u+\varepsilon v} - c_u
&=(\ell_{\min} / \widehat{m}) \left(\int_{\{u + \varepsilon v \ge c_{u+\varepsilon v} \}} \hspace{-4em} u+\varepsilon v \, \mathrm{d}s - \int_{\partial \Omega} u \, \mathrm{d}s\right)\\
&=(\ell_{\min} / \widehat{m}) \left(\int_{\partial \Omega} \varepsilon v\,\mathrm{d}s - \int_{\{u + \varepsilon v \in (c_{u}, c_{u+\varepsilon v})\}} \hspace{-4em} u + \varepsilon v\,\mathrm{d}s\right)\\
& \le( \ell_{\min} / \widehat{m}) \int_{\partial \Omega} \varepsilon v\,\mathrm{d}s.
\end{align*}
Now, since
$$\lim_{\varepsilon\to0}c_{u+\varepsilon v}=c_u\qquad\text{and}\qquad u + \varepsilon v > c_u,$$
we have that for $\varepsilon$ small enough $u + \varepsilon v > c_{u +\varepsilon v}$. A similar computation holds for the case $ c_u \ge c_{u+\varepsilon v}$, for which $\{u + \varepsilon v \ge c_{u+\varepsilon v} \} = \partial \Omega$ by assumption. Then, the above computation yields
\begin{equation*}
c_{u+\varepsilon v} - c_u=\frac{\ell_{\min}}{\widehat{m}}\int_{\partial \Omega} \varepsilon v \, \mathrm{d}s
\end{equation*}
for any $v \in H^1(\Omega)$ such that the trace of $v$ is bounded from below on $\partial \Omega$ and $\epsilon > 0$ is small enough.
The optimal distribution for the perturbed function is given by
\begin{align*}
h_{u+\varepsilon v}&=(\ell_{\min}/ c_{u+\varepsilon v})( u + \varepsilon v) - \ell_{\min}\\
&=\frac{\ell_{\min}(u+\varepsilon v)}{c_u +( \ell_{\min} / \widehat{m}) \int_{\partial \Omega} \varepsilon v \, \mathrm{d}s}-\ell_{\min}
\end{align*}
 such that the optimal distribution is differentiable with respect to $\varepsilon$ for $\varepsilon >0$ small enough.
 We set $f(\varepsilon) = \ell_{\min} + h_{u+\varepsilon v}$ and compute the first derivative
\begin{align*}
\frac{d}{d\varepsilon}(f(\varepsilon) )\vert_{\varepsilon = 0} &= \frac{c_u\ell_{\min}v-\ell_{\min}^2 \bar{u} \widehat{m}^{-1} \int_{\partial \Omega} v \, \mathrm{d}s}{c_u^2} 
\\&= \frac{\ell^2_{\min} \bar{u}\vert \partial \Omega \vert}{\widehat{m}c_{u}^2}\Big(v - 1/\vert \partial \Omega \vert \int_{\partial \Omega} v \,\mathrm{d}s\Big).
\end{align*}
The second derivative yields
\begin{align*}
\frac{d^2}{d^2\varepsilon}(f(\varepsilon) )\vert_{\varepsilon = 0}
&= \frac{-2\ell_{\min}\widehat{m}^{-1} \int_{\partial \Omega} v \mathrm{d}s\Big(c_u\ell_{\min}v-\ell_{\min}^2 \bar{u} \widehat{m}^{-1} \int_{\partial \Omega} v \, \mathrm{d}s\Big)}{c_u^3}\\
&= \frac{-2\ell^3_{\min} \bar{u}\vert \partial \Omega \vert\int_{\partial \Omega} v }{\widehat{m}^{2}c_{u}^3}\Big(v-1/|\partial\Omega|\int_{\partial \Omega} v \,\mathrm{d}s\Big).
\end{align*}
We note, that for both derivatives, the integral mean vanishes.
Then, a expansion of the boundary term in $\varepsilon$ yields
\begin{align*}
\int_{\partial \Omega} \frac{(u+ \varepsilon v)^2}{\ell_{\min} + h_{u+\varepsilon v}} \, \mathrm{d}s =&
\int_{\partial \Omega} \frac{u^2}{f(0)}\mathrm{d}s + \varepsilon \int_{\partial \Omega} \frac{u(2vf(0) -u f^\prime(0))}{f(0)^2}\mathrm{d}s \\ &+ \varepsilon^2 \int_{\partial \Omega} \left(u^2 \frac{f^\prime(0)^2}{f(0)^3} - \frac{f(0)f^{\prime \prime}(0)}{2f(0)^3}-\frac{2uvf(0) f^\prime(0)}{f(0)^3} + \frac{v^2f(0)^2}{f(0)^3} \right)+ o(\varepsilon^3).
\end{align*}
The radial symmetry of $u$, such that both $u$ and $h_u$ are constant on $\partial \Omega$, imply that several terms vanish, since the integral mean of $f^\prime$ and $f^{\prime \prime}$ vanishes
\begin{align*}
\int_{\partial \Omega} \frac{(u+ \varepsilon v)^2}{\ell_{\min} + h_{u+\varepsilon v}} \, \mathrm{d}s =& \int_{\partial \Omega} \frac{u^2}{\ell_{\min}+h_u} \mathrm{d}s + \varepsilon \int_{\partial \Omega} \frac{2uv}{\ell_{\min}+h_{u}}\mathrm{d}s + \varepsilon^2 \int_{\partial \Omega} \frac{v^2}{\ell_{\min}+h_u} \mathrm{d}s \\ &+ \varepsilon^2\int_{\partial \Omega} \frac{uf^\prime(0)}{f(0)^2}\left( \frac{uf^\prime(0)}{f(0)}- 2v\right)+o(\varepsilon^3).
\end{align*}
For the last term we have that
\begin{align*}
\int_{\partial\Omega}\frac{uf^\prime(0)}{f(0)^2}\left(\frac{uf^\prime(0)}{f(0)}- 2v\right)\mathrm{d}s=
&\int_{\partial \Omega} \frac{u}{(\ell_{\min} + h_u)^2}\frac{\ell^2_{\min} u \vert \partial \Omega \vert}{\widehat{m}c_{u}^2}\left(v - 1/\vert \partial \Omega \vert \int_{\partial \Omega} v \,\mathrm{d}s\right)\\
&\left( \frac{\ell^2_{\min} u^2 \vert \partial \Omega \vert(v - 1/\vert \partial \Omega \vert \int_{\partial \Omega} v \,\mathrm{d}s)}{\widehat{m}c_{u}^2(\ell_{\min}+h_u)}- 2v\right)\mathrm{d}s\\
=&\int_{\partial \Omega} \frac{u}{(\widehat{m}/\vert \partial \Omega \vert)^2}\frac{\ell^2_{\min} u \vert \partial \Omega \vert}{\widehat{m}(\ell_{\min} \vert \partial \Omega \vert u \widehat{m}^{-1})^2}\left(v - 1/\vert \partial \Omega \vert \int_{\partial \Omega} v \,\mathrm{d}s\right)\\
&\left( \frac{\ell^2_{\min} u^2 \vert \partial \Omega \vert(v - 1/\vert \partial \Omega \vert \int_{\partial \Omega} v \,\mathrm{d}s)}{\widehat{m}(\ell_{\min} \vert \partial \Omega \vert u \widehat{m}^{-1})^2(\widehat{m}/\vert\partial \Omega \vert)}- 2v\right)\mathrm{d}s\\
=& \int_{\partial \Omega} \frac{\vert \partial \Omega \vert}{\widehat{m}}\left(v - 1/\vert \partial \Omega \vert \int_{\partial \Omega} v \,\mathrm{d}s\right)\left(- 1/\vert \partial \Omega \vert \int_{\partial \Omega} v \,\mathrm{d}s- v\right) \mathrm{d}s\\
=&\frac{\vert \partial \Omega \vert}{\widehat{m}} \int_{\partial \Omega} \left(\left( 1/\vert \partial \Omega \vert \int_{\partial \Omega} v \,\mathrm{d}s\right)^2 -v^2\right)\mathrm{d}s\\
=&\frac{\vert \partial \Omega \vert}{\widehat{m}} \left( 1/\vert \partial \Omega \vert \left(\int_{\partial \Omega} v \mathrm{d}s\right)^2 - \int_{\partial \Omega} v^2 \mathrm{d}s \right). 
\end{align*}
With $\ell_{\min} + h_u = \widehat{m} / \vert \partial \Omega \vert$, this yields that the boundary term expands to
\begin{align*}
\int_{\partial \Omega} \frac{(u+ \varepsilon v)^2}{\ell_{\min} + h_{u+\varepsilon v}} \, \mathrm{d}s =& \frac{\vert \partial \Omega \vert}{\widehat{m}} \left( \int_{\partial \Omega} u^2 \mathrm{d}s + \varepsilon \int_{\partial \Omega}2uv \mathrm{d}s \right) + \varepsilon^2 \frac{1}{\widehat{m}} \left(\int_{\partial \Omega} v \mathrm{d}s\right)^2
 +o(\varepsilon^3).
\end{align*}
\item We now consider the function
\begin{equation*}
F(\varepsilon) = \frac{\int_\Omega \vert \nabla (u+\varepsilon v) \vert^2 \, \mathrm{d}x+\int_{\partial \Omega} \frac{(u+\varepsilon v)^2}{\ell_{\min} + h_{u+\varepsilon v}} \, \mathrm{d}s}{\int_{\Omega} (u+\varepsilon v) ^2 \, \mathrm{d}s}.
\end{equation*}
Expanding with respect to $\varepsilon$ and using the assumption that $\int_{\Omega} u^2 \mathrm{d}x = 1$ yields
\begin{align*}
F(\varepsilon) =& \left(\int_\Omega \vert \nabla u +\varepsilon v\vert^2\, \mathrm{d}x+\int_{\partial \Omega} \frac{(u+\varepsilon v)^2}{\ell_{\min} + h_{u+\varepsilon v}} \, \mathrm{d}s\right)\\ &\left(1- \varepsilon \int_{\Omega} 2uv \,\mathrm{d}x + \varepsilon^2 \left(\left(\int_{\Omega} 2uv \mathrm{d}x\right)^2 - \int_{\Omega} v^2 \mathrm{d}x\right) + o(\varepsilon^3) \right).
\end{align*}
We combine this with the expanded boundary term and \eqref{robinev} and get
\begin{align*}
F(\varepsilon)&=F(0) + 2\varepsilon \left(\int_{\Omega} \nabla u \cdot \nabla v \mathrm{d}x + \frac{\vert \partial \Omega \vert}{\widehat{m}}\int_{\partial \Omega}uv \mathrm{d}s -F(0)\int_{\Omega} uv \mathrm{d}x \right)\\
&+\varepsilon^2 \left( \int_{\Omega} \vert \nabla v \vert^2 \mathrm{d}x + \frac{1}{\widehat{m}}\left(\int_{\partial \Omega} v\mathrm{d}x\right)^2-F(0)\int_{\Omega}v^2\mathrm{d}x\right)\\
&- \varepsilon^2 \left(\left(\int_{\Omega} 2 \nabla u \nabla v \mathrm{d}x + \frac{\vert \partial \Omega \vert}{\widehat{m}} \int_{\partial \Omega} 2uv \mathrm{d}s \right)\int_{\Omega}2uv\mathrm{d}s - F(0)\left(\int_{\Omega} 2uv\,\mathrm{d}x\right)^2 \right) + o(\varepsilon^3).
\end{align*}
\item Since $F(0)$ is the first Robin eigenvalue, and $u$ the corresponding eigenfunction, the $\varepsilon$-term vanishes for every $v \in H^1(\Omega)$. This implies that 
\begin{equation*}
\frac{d}{d\varepsilon}F(\varepsilon)\vert_{\varepsilon = 0} = 0
\end{equation*}
and therefore that the eigenfunction is a critical point for any suitable test function $v$.
\item For the second variation, we consider the $\varepsilon^2$-term. 

Using again the variational eigenvalue equation we have that
\begin{equation*}
\left(\int_\Omega 2\nabla u\nabla v\,\mathrm{d}x + \frac{|\partial\Omega|}{\widehat{m}} \int_{\partial\Omega} 2uv\,\mathrm{d}s \right)\int_\Omega 2uv\,\mathrm{d}s=F(0)\left(\int_\Omega 2uv\,\mathrm{d}x\right)^2.
\end{equation*}
We choose now $v \in H^1(\Omega)$ the eigenfunction to the eigenvalue problem in optimal insulation \eqref{eq:rayleigh_oi} such that $\Vert v \Vert_{L^2(\Omega)} = 1$ and which can be chosen to be non-negative and is therefore bounded from below and satisfies the assumption to ensure the differentiablity of $h_{u+\varepsilon v}$ with respect to $\varepsilon >0$ small enough. Since $v$ is a minimizer to
\begin{align*}
\lambda_{\widehat{m}}(\Omega)=\min_{0\ne u\in H^1(\Omega) }\frac{\int_\Omega |\nabla v|^2 \mathrm{d}x + \frac{1}{\widehat{m}}\left(\int_{\partial \Omega} |v|\,\mathrm{d}x\right)^2}{\int_\Omega v^2\mathrm{d}x}
\end{align*}
for which no radial minimizer exists by assumption on $\widehat{m}< m_0$, we have that
\begin{align*}
 \int_{\Omega} \vert \nabla v \vert^2 \mathrm{d}x + \frac{1}{\widehat{m}}\left(\int_{\partial \Omega} \vert v\vert \mathrm{d}x\right)^2 = \lambda_{\widehat{m}}(\Omega) < F(0).
\end{align*} 
This implies that for $v$ the eigenfunction to the classical eigenvalue problem \eqref{eq:rayleigh_oi}, the second derivative of the function $F(\varepsilon)$ satisfies
\begin{equation*}
\frac{d^2}{d\varepsilon^2}F(\varepsilon)\vert_{\varepsilon = 0}<0
\end{equation*}
such that the radial solution $u$ becomes a local maximizer when perturbed in the direction of the eigenfunction $v$ of the eigenvalue problem \eqref{eq:rayleigh_oi}. Since $v$ is not radial symmetric for $\widehat{m}<m_0$ this implies in particular that no radial symmetric minimizer can exist.
\end{enumerate}
\end{proof}

\subsection{Existence of optimal domains}

We now consider the optimization of the eigenvalue $\lambda_{m,\ell_{\min}}(\Omega)$, defined by \eqref{rein:prop:eigval}, within the class of convex bounded sets of prescribed volume in two or three dimensions, i.e. within
where
$$\mathcal{C}_V(Q)=\left\lbrace\Omega\subset Q\ :\ \Omega\text{ is convex and open in } \mathbb{R}^d \text{ and }|\Omega|=V\right\rbrace$$
being $V>0$ and $Q\subset\mathbb{R}^d$ a convex and compact domain for $d = 2,3$.\\
We are interested in the shape optimization of the eigenvalue where a lower bound $\ell_{\min}>0$ is imposed on distribution of insulating material, but the overall mass $\widehat{m}$ of insulating material is fixed:
$$\min\Big\{\lambda_\ell(\Omega)\ :\ \Omega\in\mathcal{C}_V(Q),\ \ell\in\mathcal{H}_{\widehat{m},\ell_{\min}}(\partial\Omega)\Big\}.$$
We consider instead an equivalent problem and minimize the eigenvalue $\widehat{\lambda}_{\widehat{m},\ell_{\min}}(\Omega)= \lambda_{m_\Omega,\ell_{\min}}(\Omega)$ where the remaining free mass $m_\Omega = \widehat{m}-\ell_{\min}|\partial \Omega|$ depends on the perimeter $|\partial\Omega|$:
\begin{align*} \label{rein:problem:m_ellmin}
\min\Big\{\widehat{\lambda}_{\widehat{m},\ell_{\min}}(\Omega)\ :\ \Omega\in\mathcal{C}_V(Q),\ |\partial\Omega|\le \widehat{m}/\ell_{\min}\Big\}.\tag{$\mathbf{\widehat{P}_{\widehat{m},\ell_{\min}}}$}
\end{align*}
The isoparamteric inequality implies an upper bound for $\ell_{\min}$ such that the perimeter constraint can be satisfied, i.e. $\ell_{\min}\le\widehat{m}/|\partial B_R|$, for a ball $B_R$ with radius $R>0$ such that the ball has volume $V$.
The case $\ell_{\min} = 0$ corresponds to the regular eigenvalue of optimal insulation \cite{keller}. Then, the existence of optimal domains in the cases where $\ell_{\min} = 0$ or $\ell_{\min} = \widehat{m} / \vert \partial B_R \vert$ follows from \cite{keller} and the isoperimeteric inequality. For $\ell_{\min} \in (0, \widehat{m} / \vert \partial B_R \vert)$ the following existence result can be deduced.

\begin{proposition} \label{rein:theorem:so_existence}
There exists an optimal domain $ \Omega$ for \eqref{rein:problem:m_ellmin} if $0 \le \ell_{\min} \le \widehat{m} / \vert \partial B_R \vert$ for a ball $B_R$ with radius $R>0$ such that the ball has volume $V$.
\end{proposition}

\begin{proof}
Let us proceed step by step.
\begin{enumerate}
\item If $\ell_{\min} = 0$, existence follows from \cite{keller}. If $\ell_{\min} = \widehat{m} / \vert \partial B_R \vert$, the perimeter constraint and isoperimetric inequality, \cite[Equation (1.3)]{veodf}, imply that the only admissible domains is the ball, and the eigenvalue corresponds to the Robin eigenvalue $\lambda_1(B_R, \ell_{\min}^{-1}).$
\item If $\ell_{\min} \in (0,\widehat{m} / \vert \partial B_R \vert)$, the admissible class of domains is non-empty, and with Theorem \ref{thm:lower_bound_existence}, we can choose a minimizing sequence $(\Omega_n)_{n \in \mathbb{N}}$ to \eqref{rein:problem:m_ellmin}, with $\Omega_n \in \mathcal{C}_V(Q)$ and corresponding eigenfunctions $u_n \in H^1(\Omega_n)$ and distributions of insulating material $\ell_n \in \mathcal{H}_{m_n}(\partial \Omega_n)$ for $m_n = \widehat{m} - \ell_{\min} \vert \partial \Omega_n\vert$. Without loss of generality we may assume that $u_n$ is non-negative, and that the distribution $\ell_n$ is optimal, i.e.
\begin{equation*}
\ell_n = h_{u_n}: = \frac{\ell_{\min}}{c_n}\max(u-c_n,0)
\end{equation*}
where $c_n$ is the unique constant satisfying 
\begin{equation*}
m_n c_n\ell_{\min}^{-1} = \int_{ \{|u_n|\ge c_n \} \cap \partial\Omega_n } (|u_n|-c_n) \,\mathrm{d}s
\end{equation*}
as defined in Proposition \ref{lemma:lower_bound_optimal_distribution}.
Using a compactness results for convex domains, \cite[Lemma 3.1]{BG97}, there exists a convex domain $\Omega \in \mathcal{C}_V(Q)$ such that after passing to a subsequence, the characteristic functions $\chi_{\Omega_n}$ converge in variation to $\chi_{\Omega}$, which corresponds in this setting to the usual intermediate convergence for functions of bounded variation, \cite{BG97} and \cite[Proposition 3.6]{Ambrosio}. In particular this implies that $\chi_{\Omega_n}\to\chi_{\Omega}$ strongly in $L^1(Q)$ and $\vert \partial \Omega \vert = \lim_{n\to\infty} \vert \partial \Omega_n\vert$, such that 
\begin{equation}\label{eq:conv_m}
\lim_{n\to\infty} m_n = \widehat{m} - \ell_{\min} \vert \partial \Omega \vert =:m,
\end{equation}
and $\Omega$ also satisfies the perimeter constraint $\vert \partial \Omega \vert \le \widehat{m}/\ell_{\min}.$ 
\item After trivially extending the functions $u_n$ to $ Q$, the sequence of extended functions $ \widetilde{u}_n$ is bounded in $SBV(Q)$, the special functions of bounded variation. The bounds on the objective functional imply
\begin{align*}
\Vert \widetilde{u}_n \Vert_{L^2(Q) }&=1, \quad
\Vert \nabla \widetilde{u}_n \Vert_{L^2(Q, \mathbb{R}^d)} \le C.
\end{align*}
Here, $\nabla \widetilde{u}_n$ refers to the piecewise weak gradient with
$$\nabla\widetilde{u}_n\vert_{\Omega_n} = \nabla u_n\qquad\text{and}\qquad\nabla \widetilde{u}_n \vert_{\widehat{Q} \backslash \Omega_n} = 0.$$
Using \cite[Theorem 2.6]{BG97}, the weak differentiabilty of $\widetilde{u}_n$ on $\Omega_n$ and $Q \backslash \overline{\Omega}_n$ and the uniform bound on the trace operator in the class $\mathcal{C}_V(Q)$, see \cite{keller,payne}, we derive 
\begin{equation} \label{eq:bound:bdy_term}
\int_{J_{\widetilde{u}_n}} u_n^2 d \mathcal{H}^{d-1} = \int_{\partial \Omega_n} u_n^2 \,\mathrm{d}s \le C \Vert u_n \Vert_{L^2(\partial \Omega_n)} \le C \Vert \widetilde{u}_n \Vert_{H^1(\Omega_n)} \le C
\end{equation}
from which we infer the boundedness of the sequence $(u_n)_{n \in \mathbb{N}}$ and in particular also of $(u_n^2)_{n \in \mathbb{N}}$ in $SBV(Q)$.

\item With the compactness results for special functions of bounded variation, \cite[Theorem 2.1]{BG} and \cite[Theorem 2]{BG10}, there exists a function $\widetilde{u} \in SBV(Q)$, such that after passing to a subsequence, $\widetilde{u}_n$ converges weakly in $SBV(Q)$ to $\widetilde{u}$, see, in particular 
\begin{alignat}{3}
D\widetilde{u}_n&\rightharpoonup^\star D\widetilde{u} \quad &&\text{ in the sense of measures}\label{eq:oi_conv_meas}\\
\widetilde{u}_n&\to\widetilde{u}\quad&&\text{ strongly in } L^2(Q)\label{eq:oi_l1_conv}\\
\nabla \widetilde{u}_n&\rightharpoonup \nabla \widetilde{u}\quad&&\text{ weakly in } L^2(Q, \mathbb{R}^d). \label{eq:oi_conv_grad}
\end{alignat}

Using these convergence results, we deduce as in \cite[Proposition 1]{keller} that $u\vert_{\Omega} \in H^1(\Omega)$ and $u\vert_{Q \backslash \Omega} = 0$.

With the lower semicontinuity of jump energies \cite[Theorem 2.12]{braides}, it follows that
\begin{equation} \label{rein:eq:lsc_ubdy}
\int_{\partial \Omega} \vert u \vert \,\mathrm{d}s \le \int_{\partial \Omega_n} \vert u_n \vert \,\mathrm{d}s.
\end{equation}
 
In summary, we have that $u \in H^1(\Omega)$ with
$$\|\nabla u\|^2_{L^2(\Omega)} \le \liminf_{n\to\infty}\|\nabla u_n\|^2_{L^2(\Omega_n)}$$
and $\Vert u \Vert_{L^2(\Omega)} = 1$. Left to show is the lower semicontinuity of the boundary term.
\item Using that $c_u$ satsfies \eqref{eq:c_u} yields
$$0\le c_n\le C\int_{\partial \Omega_n} u_n \,\mathrm{d}s \le C\Vert u_n \Vert_{H^1(\Omega_n)}$$
and implying that the sequence $(c_n)_{n\in\mathbb{N}}$ is bounded.\\
We first consider the case $\liminf_{n\to\infty} c_n=0$. After passing to a subsequence we have that 
\begin{align*}
0 & = \liminf_{n\to\infty} c_n=\lim_{k\to\infty} c_{n_k} \\ &= \lim_{k\to\infty} (\vert \{u_{n_k} \ge c_{n_k}\} \cap \partial \Omega_{n_k} \vert + m_{n_k}\ell_{\min}^{-1})^{-1}\int_{\partial \Omega_n \cap \{ u_{n_k} \ge c_{n_k}\} } |u_{n_k}|\,\mathrm{d}s.
\end{align*}
Using the estimate $|\partial\Omega_{n_k}|\le|\partial Q|$, since $\Omega_{n_k}\subset Q$, and \eqref{eq:conv_m} yields
\begin{align*}
0 &\ge (\vert \partial Q \vert + m \ell_{\min}^{-1} )^{-1} \lim_{k\to\infty} \int_{\partial \Omega_{n_k} \cap \{ u_{n_k}\ge c_{n_k}\} }|u_{n_k}|\,\mathrm{d}s \\
& \ge(\vert \partial Q \vert + m \ell_{\min}^{-1} )^{-1} \lim _{k\to\infty}\left( \int_{\partial \Omega_{n_k} }\vert u_{n_k} \vert \,\mathrm{d}s - \int_{\partial \Omega_{n_k}\cap \{ u_{n_k} < c_{n_k}\}}c_{n_k}\,\mathrm{d}s \right).
\end{align*}
The assumed convergence of the subsequence $\lim c_{n_k}= 0$ and the lower semicontinuity of the boundary term \eqref{rein:eq:lsc_ubdy} yields
\begin{equation*}
 \int_{\partial \Omega} \vert u \vert \,\mathrm{d}s \le 0.
\end{equation*}
 This implies that $u\vert_{\partial \Omega} = 0$ almost everywhere on $\partial \Omega$, such that for any $\ell \in \mathcal{H}_m(\partial \Omega)$
\begin{equation*}
\int_{\partial \Omega} \frac{u^2}{\ell_{\min} + \ell} \,\mathrm{d}s= 0 \le \liminf_{n\to\infty} \int_{\partial \Omega_n} \frac{u_n^2}{\ell_{\min} + \ell_n} \,\mathrm{d}s 
\end{equation*}
implying the optimality of the domain $\Omega$.
\item In the following we consider the case $\liminf c_n>0$. We pass to a convergent subsequence and set $c = \lim {c_n}$. We extend the function
$$h_{u_n}=\frac{\ell_{\min}}{c_n}\max(u_n -c_n,0) \in L^2(\partial \Omega)$$
to $Q$ by considering
\begin{equation*}
\widetilde{\ell}_n = \begin{cases}
(\ell_{\min}/c_n) \max(u_n -c_n,0)&\text{in }\Omega_n\\
0&\text{in } Q\backslash \Omega_n.
\end{cases}
\end{equation*}
Then $\widetilde{\ell}_n \vert_{\Omega_n} \in H^1(\Omega_n)$, \cite[5.20]{dobrowolski}, and $\widetilde{\ell}_n \in SBV(Q)$ with $\widetilde{\ell}_n \vert_{\partial \Omega_n} = h_{u_n}$ on $\partial \Omega_n$ using the chain rule for $BV$ functions \cite[Theorem 3.99]{Ambrosio}. The bound on $\widetilde{u}_n\vert_{\Omega_n}$ in $H^1(\Omega_n)$ and the fact that $\lim_n c_n=c>0$ yield that the sequence $(\widetilde{\ell}_n)_{n \in \mathbb{N}}$ is bounded in $SBV(Q)$, using again the uniform bound of the trace operator to bound the jump terms. \\
The compactness results in $SBV(Q)$ imply the existence of a function $\widetilde{\ell} \in SBV(Q)$, such that $\widetilde{\ell}_n$ converge weakly to $\widetilde{\ell}$ in $SBV(Q)$ and weakly in $L^2(Q)$ after passing to a subsequence. With the same arguments as before we conclude that $\widetilde{\ell}\vert_{\Omega}$ in $H^1(\Omega)$ and $\widetilde{\ell}\vert_{Q\backslash \overline{\Omega}} = 0$ and for the trace $\ell = \widetilde{\ell} \vert_{\partial \Omega} $ on $\partial \Omega$ that 
\begin{equation*}
\int_{\partial \Omega} \ell \,\mathrm{d}s \le \liminf_{n\to\infty} \int_{\partial \Omega_n} \ell_n \,\mathrm{d}s = m.
\end{equation*}
\item Next, we show that
\begin{equation}\label{eq:lsc_bdy_lboi}
\int_{\partial \Omega} \frac{u^2}{\ell_{\min} + \ell}\, \mathrm{d}s \le \int_{\partial \Omega_n} \frac{u_n^2}{\ell_{\min} + \ell_n}\, \mathrm{d}s.
\end{equation}
To achieve this, we consider the functions

\begin{equation*}
\widetilde{v}_n:= \begin{cases}
\displaystyle\frac{u_n^2}{\ell_{\min}+h_{u_n}} & \text{ in } \Omega_n \\ 0 & \text{ in } Q\backslash \Omega_n. \end{cases}
\end{equation*}
Using that $\widetilde{u}_n = 0 $ in $Q\backslash \overline{\Omega}$, and the explicit formula of $h_{u_n}$, we have $\widetilde{v}_n = G_{c_n}(\widetilde{u}_n)$ with
\begin{align*}
G_{c_n}(x) = \begin{cases} \ell_{\min}^{-1} x^2 & \text{ if }\vert x \vert < c_n\\
c_n \ell_{\min}^{-1} \vert x \vert & \text{ if } \vert x \vert \ge c_n \end{cases}
\end{align*}
 which are Lipschitz continuous functions. The chain rule for functions in $BV$ then implies that $\widetilde{v}_n$ in $SBV(Q)$, with
\begin{align*}
D\widetilde{v}_n = &\chi_{\Omega_n}\left(\chi_{\{\widetilde{u}_n < c_n\}} \ell_{\min} 2 \widetilde{u}_n \nabla \widetilde{u}_n + \chi_{\{\widetilde{u}_n \ge c_n\}} c_n\beta \nabla \widetilde{u}_n\right)\otimes \mathrm{d}x\\
&+\frac{\widetilde{u}_n^2}{\ell_{\min}+h_{u_n,c_n}}\nu \otimes \mathrm{d}\mathcal{H}^{d-1}\lfloor \, \partial \Omega_n.
\end{align*}
Combined with the bounds on $\Vert u_n \Vert_{H^1(\Omega_n)}$ and $c_n$, this implies that $\widetilde{v}_n\vert_\Omega$ is in $H^1(\Omega_n)$. The sequence $\widetilde{v}_n$ is bounded in $SBV(Q)$, and admits a weakly convergent subsequence $\widetilde{v}_n \rightharpoonup \widetilde{v} $ in $SBV(Q)$. With similar arguments as before, we conclude that $\widetilde{v} \vert_{\Omega} \in H^1(\Omega)$. \\
Next, we show that 
\begin{equation}\label{oi:eq:v}
\widetilde{v} = \frac{u^2}{\ell_{\min}+\ell}.
\end{equation}
Using $\ell_n \rightharpoonup \ell$ weakly in $L^2(Q)$, $\widetilde{u}_n\to\widetilde{u}$ strongly in $L^2(Q)$, as well as the Sobolev embedding $H^1(\Omega) \hookrightarrow L^4(\Omega)$, we infer that $\widetilde{v}_n \rightharpoonup \widetilde{v}$ in $L^2(\Omega)$.
The uniqueness of the limit, in particular \cite[Equation (3.11)]{Ambrosio}, implies \eqref{oi:eq:v} and the lower semicontinuity of jump terms implies the assertion \eqref{eq:lsc_bdy_lboi}.
\item For $0 < m_1 \le m_2$, and $\ell_1 \in \mathcal{H}_{m_1}(\partial \Omega)$, there exists $\ell_2 \in \mathcal{H}_{m_2}(\partial \Omega)$, such that for any $u \in H^1(\Omega)$
\begin{equation*}
\int_{\partial \Omega} \frac{u^2}{\ell_{\min} + \ell_2}\,\mathrm{d}s \le \int_{\partial \Omega} \frac{u^2}{\ell_{\min} + \ell_1}\,\mathrm{d}s.
\end{equation*}
Indeed, if we take $\ell_2 = (m_2/m_1) \ell_1$, then $\ell_2 - \ell_1 \ge 0$ on $\partial \Omega$, which implies the assertion. As a consequence, for any $u \in H^1(\Omega)$, the optimal distribution $h_u$ from Proposition \ref{lemma:lower_bound_optimal_distribution} is optimal also among all distributions $\ell \in L^2(\partial \Omega)$ with $\ell \ge 0$ and $\int_{\partial \Omega} \ell \, \mathrm{d}s \le m$. With \eqref{eq:lsc_bdy_lboi}, this implies 
\begin{equation*}
\int_{\partial \Omega} \frac{u^2}{\ell_{\min} + h_{u}} \,\mathrm{d}s\le \int_{\partial \Omega} \frac{u^2}{\ell_{\min}+ \ell} \,\mathrm{d}s\le \liminf_{n\to\infty} \int_{\partial \Omega_n} \frac{u^2_n}{\ell_{\min}+ \ell_n}\,\mathrm{d}s
\end{equation*}
and the optimality of the limit $\Omega$ to the problem \eqref{rein:problem:m_ellmin}.
\end{enumerate}
\end{proof}
The proof uses the embedding of $H^1(\Omega) \hookrightarrow L^4(\Omega)$ which holds only if $d \le 4$. In 
the absence of a convexity constraint non-existence of an optimal domain can be proven similar to the eigenvalue in optimal insulation \cite{BB19}.

\begin{remark}[Non-existence]\rm
We choose a domain $\Omega$, which is the disjoint union of two balls with radii $0<r_{1} < r_{2}$ chosen such that the volume constraint on the domains $\Omega$ is satisfied, for simplicity we take that $\vert \Omega\vert = \vert B_1(0) \vert$. One can verify that for any $\ell_{\min}>0$ the perimeter constraint is satisfied for $\Omega $ provided that $r_1$ is small enough. \\
Using the explicit formula \eqref{rein:optimal_h_u}, we compute 
\begin{align*}
\widehat{\lambda}_{\widehat{m},\ell_{\min}}(\Omega) &= \inf_{u \in H^1(\Omega)} \left\lbrace \int_{\Omega} \vert \nabla u \vert^2 \,\mathrm{d}x + \ell_{\min}^{-1}\int_{\partial \Omega \cap \{ \vert u \vert < c_u\} } u^2\,\mathrm{d}s + \frac{c_u}{\ell_{\min}} \int_{\partial \Omega \cap \{ \vert u \vert \ge c_u\} } \vert u \vert \,\mathrm{d}s \right\rbrace \\ &= \inf_{u \in H^1(\Omega)} \left\lbrace \int_{\Omega} \vert \nabla u \vert^2 \,\mathrm{d}x + \ell_{\min}^{-1}\int_{ \{ \vert u \vert < c_u\} } u^2\,\mathrm{d}s + \frac{1}{\ell_{\min} \vert \{ \vert u \vert \ge c_u \} \vert +m}\left(\int_{ \{ \vert u \vert \ge c_u\} } \vert u \vert \,\mathrm{d}s\right)^2 \right\rbrace.
\end{align*}
with the constant $c_u$ given by \eqref{rein:eq:c_u}.
Now, choosing 
$u \equiv 1$ on $B_{r_{1}}$ and $u\equiv 0$ on $B_{r_{2}}$, we compute
\begin{equation*}
c_{u} =\Big(1+ \frac{\widehat{m}}{\ell_{\min} \vert \partial B_{r_{1}} \vert }\Big)^{-1} .
\end{equation*}
The value $c_u \in (0,1]$ is such that
$$\{\vert u \vert \ge c_u \} = \partial B_{r_1}\qquad\text{and}\qquad\{\vert u \vert < c_u \} = \partial B_{r_2}.$$
Normalizing the function $u$ such that $\Vert u \Vert^2_{L^2(\Omega)} = 1$ yields the estimate
\begin{align*}
\widehat{\lambda}_{\widehat{m},\ell_{\min}}(\Omega)&\le \frac{\vert \partial B_{r_{1}}\vert^2 }{\vert B_{r_{1}}\vert( \ell_{\min} \vert \vert \partial B_{r_1}\vert +\widehat{m}) } \le \frac{d^2 \vert B_1(0) \vert}{\widehat{m}+\ell_{\min} d r_1^{d-1}|B_1|}r_1^{d-2}
\end{align*}
which can be arbitrary close to zero if $d \ge 3$. 
\end{remark}

The numerical approximation of the optimal domains can be improved by accounting for the effect the change of volume has on the eigenvalue. 
For this it is useful to know the scaling properties of the eigenvalues under consideration. For a domain $\Omega \subset \mathbb{R}^d$, we consider $t\Omega :=\{tx :x \in \Omega\}$ for $t > 0$. For the eigenvalues of the Dirichlet, Neumann and Robin Laplacian the following scaling properties are known for $n \in \mathbb{N}$, \cite[p. 84]{henrot}:
\begin{align*}
\lambda_n(t\Omega) = t^{-2} \lambda_n(\Omega), \qquad \mu_n(t\Omega) = t^{-2} \mu_n(\Omega), \qquad \lambda_n(t\Omega, \alpha t^{-1} ) = t^{-2} \lambda_n(\Omega,\alpha).
\end{align*} 
For the eigenvalue with an insulating film \eqref{rein:eigval_ell} for the scaled domain $t\Omega$, for $t \in (0,\infty)$ we can verify that 
\begin{align*}
\lambda_{\ell_{t}}(t\Omega)= t^{-2} \lambda_{\ell t^{-1}}(\Omega)
\end{align*}
where the function $\ell_t \in L^2(\partial (t\Omega))$ is defined by $\ell_t(x) = \ell(x/t)$ a.e. for $x \in \partial (t\Omega)$.
If we consider the problem with a lower bound, so that $\widehat{\ell} = \ell_{\min} + \ell$, then this implies the following scaling properties
$$\widehat{m}_t=t^d\widehat{m},\qquad m_t=t^d m,\qquad\ell_{\min,t}=t\ell_{\min},\qquad\ell_{t}=t\ell.$$
For the numerical approximation of the eigenvalue in optimal insulation, we always consider the eigenvalue scaled to the volume of the unit ball, and give the scaling parameters for reference. 
To achieve this, given a domain $\Omega \in \mathcal{C}_V(Q)$ and a scaled domain $t\Omega$, we compute $t = ( \vert t\Omega \vert / \vert \Omega \vert)^{1/d}$. After adjusting the parameters $\ell_{\min,t}$,$m_t$ and $\widehat{m}_t$ accordingly, the eigenvalues are iteratively minimized as described in Section \ref{sec:iterative_min} and the rescaled eigenvalue 
\begin{equation}
\widehat{\lambda}_{\widehat{m},\ell_{\min}}(\Omega) = t^2 \widehat{\lambda}_{\widehat{m}_t,\ell_{\min,t}}(t\Omega)
\end{equation}
is evaluated.

\section{Numerical approximation}

In this section, we consider the numerical approximation of the eigenvalue in optimal insulation with a lower bound on the thickness of the insulating film. We propose an iterative scheme to minimize the objective functional, give an error estimate for the approximation of the eigenvalue with discrete convex domains and consider the stability of the shape optimization among convex domains in $\mathbb{R}^3$. \\
To approximate the convex domains in $\mathbb{R}^3$, we use the \textit{discrete convex} domains previously defined in \cite{keller2_preprint}.

\subsection{Iterative minimization}\label{sec:iterative_min}

When minimizing over the product of two functions, decoupling arises naturally as a consequence of the discrete product rule for the difference quotient
\begin{align*}
d_t(a_kb_k) & = (d_ta_k)b_k + a_{k-1}(d_tb_k) \\ \text{ for } d_ta_k &= (1/\tau)(a_k -a_{k-1})
\end{align*}
for a suitable stepsize $\tau >0$.
This allows us to minimize $J_{\ell_{\min}}$ separately with respect to $u\in H^1(\Omega)$ and $\ell \in \mathcal{H}_{m}(\partial \Omega)$.
Minimizers $u$ of $J_{\ell_{\min}}(\cdot,\ell)$ with $\ell \in \mathcal{H}_m(\partial \Omega)$ fixed satisfy the Euler-Lagrange equation
\begin{equation} \label{eq:lb:variational}
 \int_\Omega \nabla u \cdot \nabla v \, \mathrm{d}x + \int_{\partial \Omega} \frac{uv}{\ell_{\min} + \ell}\, \mathrm{d}s = \lambda_{\ell,\ell_{\min}}(\Omega) \int_{\Omega} uv \,\mathrm{d}x
\end{equation}
for all $v \in H^1(\Omega)$ and 
\begin{equation*}
\lambda_{\ell,\ell_{\min}}(\Omega): = \inf_{u \in H^1(\Omega)} \left\lbrace J_{\ell_{min}}(u,\ell): \Vert u \Vert_{L^2} = 1\right\rbrace.
\end{equation*}
No additional regularization is needed here, since $\ell_{\min} >0$.

The constraint $\Vert u \Vert_{L^2(\Omega)} = 1 $ motivates to use test functions $v \in H^1(\Omega)$ for which $(u,v)_{L^2(\Omega)} = 0$, such that the right-hand side in \eqref{eq:lb:variational} with the unknown eigenvalue disappears.
An iterative scheme can then be derived with the corresponding evolution equation as in \cite{BB19}.
Using the optimal distribution \eqref{rein:optimal_h_u} yields the following algorithm.

\begin{algorithm_env} \label{rein:algorithm:lb}
Suppose $(\cdot,\cdot)_{\star,\Omega}$ is an appropriate inner product on $H^1(\Omega)$. For the initial data $ u_0 \in H^1(\Omega), \ell_0 \in \mathcal{H}_{m}(\partial \Omega)$ and $\Vert u_0\Vert^2 = 1$, set $ k = 1$, and repeat the following steps:
\begin{enumerate}
\item Find $d_t u_k \in H^1(\Omega)$ with $(u_{k-1},d_t u_k)_{\star,\Omega}=0$ s.t. for all $v \in H^1(\Omega)$ with $(u_{k-1},v)_{\star,\Omega}=0$:
\begin{equation*}
(d_t u_k,v)_{\star,\Omega} + \int_\Omega \nabla u_k \cdot \nabla v \, \mathrm{d}x + \int_{\partial \Omega} \frac{u_kv}{\ell_{\min} +\ell_{k-1}}\,\mathrm{d}s= 0.
\end{equation*}
\item Compute the optimal distribution $\ell_k = h_{u_k}\in \mathcal{H}_{m}(\partial \Omega)$ according to \eqref{rein:optimal_h_u}.
\item Stop if $\Vert d_tu_k\Vert_{\star,\Omega}\le \varepsilon_{\mathrm{stop}}$; otherwise increase $k$ to $k+1$ and continue with $(1)$.
\end{enumerate}
\end{algorithm_env}

The iterates are energy decreasing and satisfy an approximate energy estimate on finite intervals $[0, T ]$.
\begin{proposition}
Algorithm \ref{rein:algorithm:lb} is energy decreasing in the sense that for all $K = 0,1, \dots, \lfloor T/\tau\rfloor$ we have
\begin{equation*}
J_{\ell_{\min}}(u_K,\ell_K) + \tau \sum_{k = 1}^K \Vert d_t u_k \Vert_\star^2 \le J_{\ell_{\min}}(u_0,\ell_0).
\end{equation*}
If $\Vert u_0 \Vert_{L^2(\Omega)} = 1$, then $\Vert u^K \Vert_{L^2(\Omega)} \ge 1$, in particular
\begin{equation*}
\Vert u_K \Vert^2_{L^2(\Omega)} = 1+ \tau^2 \sum_{k = 1}^k \Vert d_t u_k \Vert_{L^2(\Omega)}^2.
\end{equation*}
\end{proposition}

\begin{proof}
Testing with $v = d_t u_k$ yields
\begin{align*}
0&=\|d_t u_k\|^2_\star + \int_\Omega \nabla u_k \cdot \nabla d_t u_k \, \mathrm{d}x + \int_{\partial \Omega} \frac{(d_t u_k) u_k}{\ell_{\min} + \ell_{k-1}} \, \mathrm{d}s \\
&=\Vert d_t u_k \Vert^2_\star+\frac{1}{2} d_t \Vert \nabla u_k \Vert^2 +\frac{\tau}{2} \Vert \nabla d_t u_k \Vert^2 + \frac{1}{2} \int_{\partial \Omega} \frac{d_t (u_k^2)}{\ell_{\min} +\ell_{k-1}} + \tau \frac{(d_t u_k)^2}{\ell_{\min} +\ell_{k-1}}\, \mathrm{d}s
\end{align*}
using the equality
\begin{equation*}
u_k(d_t u_k) = \frac{1}{2}d_t (u_k^2) + \frac{\tau}{2}(d_t u_k)^2.
\end{equation*}
This implies 
\begin{align*}
0 &\ge \Vert d_t u_k \Vert^2_\star+\frac{1}{2} d_t \Vert \nabla u_k \Vert^2 + \frac{1}{2} \int_{\partial \Omega} \frac{d_t (u_k^2)}{\ell_{\min} +\ell_{k-1}} \, \mathrm{d}s.
\end{align*}
We use the minimizing property of $\ell_k = h_{u_k}$, i.e. 
\begin{equation*}
 \int_{\partial \Omega} \frac{u_k^2}{\ell_{\min} +\ell_{k}} \, \mathrm{d}s\le\int_{\partial \Omega} \frac{u_k^2}{\ell_{\min} +\ell_{k-1}}\, \mathrm{d}s
\end{equation*}
to deduce that
\begin{align*} 
0 &\ge\Vert d_t u_k \Vert^2_\star+ \frac{1}{2} d_t \Vert \nabla u_k \Vert^2 + \frac{1}{2\tau} \int_{\partial \Omega} \frac{u_k^2}{\ell_{\min} +\ell_{k}}-\frac{u_{k-1}^2}{\ell_{\min} +\ell_{k-1}}\,\mathrm{d}s \\ 
& = \Vert d_t u_k \Vert^2_\star+ \frac{1}{2} d_t \Vert \nabla u_k \Vert^2 +\frac{1}{2} \int_{\partial \Omega} d_t\left(\frac{u_k^2}{\ell_{\min} +\ell_{k}} \right)\, \mathrm{d}s.
\end{align*}
Further, the orthogonality $(d_tu_k,u_k)= 0$ implies $\Vert u_k \Vert^2_{L^2} = \Vert u_{k-1} \Vert_{L^2}^2 + \tau^2 \Vert d_t u_k \Vert^2_{L^2}$. Summing over $k = 1,\dots , K$ leads to the assertion.
\end{proof}
Since no regularization is necessary the energy convergence is better than for the version without a positive lower bound. However, if $\ell_{\min}$ is small, we observe a higher error in the numerical approximation, both in a slower convergence behaviour of the iterations and error rates in the discretization, see Proposition \ref{rein:prop:consistency_lb_ldc} and Figure \ref{rein:ex:eigenvalues_l_min_3d}.

\subsection{Convergence analysis} \label{sec:oi:conv}

We consider here only the approximation of the eigenvalue in convex domains $\Omega \subset \mathbb{R}^3$, where the domain $\Omega$ is approximated with polyhedral domains with a family of regular triangulations $(\mathcal{T}_h)_{h>0}$
interpolating the convex domains, and discretize the space $H^1(\Omega)$ with the space of continuous piecewise affine linear functions
\begin{equation*}
\mathcal{S}^1(\mathcal{T}_h) = \left\lbrace v_h \in C(\overline{\Omega}_h) : v_h \vert _T \in P_1(T) \text{ for all } T \in \mathcal{T}_h \right\rbrace.
\end{equation*}

\begin{proposition}\label{rein:prop:consistency_lb_ldc}
For a convex domain $\Omega \subset \mathbb{R}^3$ with a $C^{2,1}$-boundary assume that there exists a pair of minimizers $(u,\ell)$ with $u \in H^2(\Omega)\cap W^{1,\infty}(\Omega)$, $\Vert u \Vert_{L^2(\Omega)} = 1$ and $\ell\in H^2(\partial \Omega)\cap \mathcal{H}_{m}(\partial \Omega)$. There exists a polyhedral domain $\Omega_h\subset \Omega$ with a triangulation $\mathcal{T}_h$ with maximal mesh size $h$ and
\begin{equation*}
\Omega \triangle \Omega_h \subset \left\lbrace x \in Q : \mathrm{dist}(x,\partial \Omega) \le ch^2 \right\rbrace
\end{equation*} 
 so that
\begin{equation*}
 \min_{u_h \in \mathcal{S}^1(\mathcal{T}_h), \ell_h \in \mathcal{S}^1(\partial \mathcal{T}_h) }\big\vert J_{\ell_{\min}}(u,\ell) -J_{\ell_{\min}}(u_h,\ell_h)\big\vert \le C(\ell_{\min}^{-2})h.
\end{equation*}
\end{proposition}

\begin{proof}
\begin{enumerate} 
\item We can find a sequence of discrete domains satisfying the requirement as described using standard approximation results of the triangulation \cite[Section 3.6]{dziuk} and interpolated surfaces \cite{dziukelliot}.
\item Set
\begin{equation*}
 u_h = \frac{I_h(u)}{\Vert I_h(u)\Vert_{L^2(\Omega_h)}},\qquad\ell_h = m \frac{\widetilde{I}_h(\ell)}{\Vert \widetilde{I}_h(\ell)\Vert_{L^1(\partial \Omega_h)}}
\end{equation*} 
with the nodal interpolation operator $I_h: C(\overline{\Omega})\to\mathcal{S}^1(\mathcal{T}_h)$ and $\widetilde{I}_h: C(\partial \Omega)\to\mathcal{S}^1(\partial \mathcal{T}_h)$, the nodal interpolation operator on the boundary $\partial \Omega$ as defined in \cite[Section 4]{dziukelliot}. \\For $h>0$ small enough, the functions $u_h$ and $\ell_h$ are well defined. Standard interpolation results, see \cite[Section 3.3.1]{dziuk} and \cite[Lemma 4.3]{dziukelliot} for element sides on the discrete boundary, imply
\begin{align*}
\Vert u - u_h \Vert +h \Vert \nabla [u-u_h] \Vert &\le C h^2 \Vert D^2 u \Vert\\
\Vert \ell-\ell_h^L \Vert_{L^2(\partial \Omega)} &\le Ch^2 \Vert \ell\Vert_{H^2(\partial \Omega)} 
\end{align*}
where $\ell_h^L$ is the lift of $\ell_h$ from $\partial \Omega_h$ to $\partial \Omega$, as described in \cite[Section 4.1]{dziukelliot}.
\item We have that
$$\|u\|_{H^1(\Omega\backslash\Omega_h) }\le C|\Omega\backslash\Omega_h|^{1/2}\|u\|_{W^{1,\infty}(\Omega)}\le Ch.$$
With the continuity of the trace operator we get 
\begin{align*} 
\vert J_{\ell_{\min}}&(u_h,\ell_h) -J_{\ell_{\min}}(u,\ell )\vert \le(\nabla [u_h + u], \nabla[u_h-u])_{L^2(\Omega_h)} + \Vert \nabla u \Vert^2_{L^2(\Omega\backslash \Omega_h)}\\
&\, + C \int_{\partial \Omega} \bigg\vert \left( \frac{u_h^2}{\ell_{\min}+\ell_h}\right)^L-\frac{u^2}{\ell_{\min}+\ell }\bigg\vert \,\mathrm{d}s\\
&\le C(h \Vert u \Vert^2_{H^2(\Omega)} +h^2\Vert u \Vert^2_{W^{1,\infty}(\Omega)})+ \int_{\partial \Omega} \frac{\vert (u_h^{L})^2 -u^2\vert }{\ell_{\min}+\ell^L_h} \,\mathrm{d}s+ C\int_{\partial \Omega}\frac{u^2 \vert\ell-\ell^L_h\vert }{(\ell_{\min}+\ell^L_h )(\ell_{\min}+\ell)} \,\mathrm{d}s\\
&\le Ch \|u\|^2_{H^2(\Omega)}+Ch^2\Vert u \Vert^2_{W^{1,\infty}(\Omega)} + C\ell_{\min}^{-1}h \Vert u \Vert^2_{H^2(\Omega)}+\Vert u^2 \Vert_{L^2(\partial \Omega)} \ell_{\min}^{-2} h^2 \Vert\ell \Vert_{H^2(\partial \Omega)}\\
&\le Ch \Vert u \Vert^2_{H^2(\Omega)} (1+\ell_{\min}^{-1}) + \ell_{\min}^{-2} h^2\Vert u \Vert^2_{L^4(\partial \Omega)}\Vert \ell \Vert_{H^2(\partial \Omega)} +Ch^2\Vert u \Vert^2_{W^{1,\infty}(\Omega)}\\
&\le Ch\left( (1+\ell_{\min}^{-1})\Vert u \Vert_{H^2(\Omega)}+h\Vert u \Vert^2_{W^{1,\infty}(\Omega)} + \ell_{\min}^{-2}h\Vert \ell \Vert_{H^2(\partial \Omega)} \right)
\end{align*}
with the continuous embedding of $W^{1,2}\hookrightarrow L^4(\partial \Omega)$ for $d \le 3$, \cite[Theorem 6.15]{dobrowolski}. This proves the asserted estimates.
\end{enumerate}
\end{proof}

\subsection{Stability of the shape optimization}

Since we consider the shape optimization among convex domains in $\mathbb{R}^3$, a conformal approximation of the convexity constraint among polyhedral domains is difficult.
To circumvent this, we approximate convex domains with discrete convex domains as defined in \cite{keller2_preprint}.

\begin{definition}[{\cite[Definition 2.1]{keller2_preprint}}]\label{ldc:def:discrete_convex_domain}
Let $\Omega_h \subset \mathbb{R}^3$ be a polyhedral bounded domain with a triangulation $\mathcal{T}_h$ with maximal mesh size $h$. We say that $\Omega_h$ is \textup{discrete convex} if $\partial \Omega_h$ is a piecewise linear Lagrange interpolant of the boundary $\partial \Omega$ of a convex domain $\Omega$.
\end{definition}

We discretize the class of admissible domains $\mathcal{C}_V(Q)$ with
\begin{align*}
\mathcal{C}^{h,c_{\mathrm{usr}}}_V(Q) =\lbrace \Omega _h \subset Q: &\text{ with } \mathcal{T}_h \in \mathbb{T}_{h,c_{\mathrm{usr}}} \text{ and } \Omega _h \text{ interpolates } \Omega\in \mathcal{C}_V(Q) \rbrace
\end{align*}
with $(\mathbb{T}_{h,c_{\mathrm{usr}}})_{h>0}$ the class of uniform shape regular triangulations, such that for a triangulation $\mathcal{T}_h \in \mathbb{T}_{h,c_{\mathrm{usr}}}$ and $T \in \mathcal{T}_h$ we have $h_T = \mathrm{diam}(T)\le h$ and $h_T \le c_{\mathrm{usr}} \rho_T$ for $\rho_T = \sup \{r: r>0, x \in T,B_r(x) \subset T\} $ the inner radius of $T$. \\

With the results in \cite{keller2_preprint} it is possible to prove the stability of shape optimization for simple problems, e.g. problems constrained by a Poisson problem, despite the non-conformal approximation of the convexity constraint. Of particular importance is the following compactness result.
\begin{corollary}[{\cite[Corollary 2.6]{keller2_preprint}}]
Let $ (\Omega_h)_{h>0}$ be a sequence of discrete convex domains with $\Omega_h \subset \mathcal{C}^{h,c_{\mathrm{usr}}}_V(Q)$ for $h>0$. Then there exists a convex domain $\Omega \in \mathcal{C}_V(Q)$, such that after passing to a subsequence $ \chi_{\Omega_h}\to\chi_{\Omega}$ in $L^1(Q)$.
\end{corollary}
This is a weaker result than the compactness with respect to convergence in variation, which is available for convex domains \cite{BG97}, but sufficient for simple optimization problems such as the ones considered in \cite{keller2_preprint}.
When adressing the approximation of problems with boundary terms or eigenvalues, better convergence results are needed, mainly a uniform bound on the trace operator as used in \cite{keller} and Proposition \ref{rein:theorem:so_existence}. The uniform bound can be proven with a constructive proof, cf. \cite{keller,payne}, by using a cone condition as described in \cite{adams}. With some technical assumptions on the admissible triangulations it is then possible to construct a uniform bound on the trace operator for the discrete domains. For now, we make the assumptions, that the class of discrete domains is suitable, such that the trace operator can be uniformly bounded.

\begin{align*}\label{rein:problem:lb_ldc}
\begin{cases}
\textup{Minimize } &J_{\ell_{\min}} (\Omega_h,u_h,\ell_h) \\
\textup{w.r.t. }&\Omega_h \in\mathcal{C}^{h,c_{\mathrm{usr}}}_V(Q)\text{ with }\vert \partial \Omega_h \vert \le \widehat{m}/\ell_{\min}\\ \text{with } & \mathcal{T}_h \in \mathbb{T}_{c_{\mathrm{usr}},h} \textup{ triangulation of } \Omega_h \tag{$\mathbf{P}_{\widehat{m},\ell_{\min}}^{h}$}\\
\textup{s.t. }& (u_h,\ell_h)\in \mathcal{S}^1(\mathcal{T}_h) \times \mathcal{S}^1(\partial \mathcal{T}_h) \textup{ minimizes } J_{\ell_{\min}}(\Omega_h,\cdot,\cdot) \\ \textup{and } &\Vert u_h \Vert_{L^2(\Omega_h)} = 1, \ell_h \in \mathcal{H}_{m}(\partial \Omega_h) \text{ with } m_h = \widehat{m}-\vert \partial \Omega_h \vert \ell_{\min}
\end{cases}
\end{align*} 
 Convergence as $ h \rightarrow 0$ holds in the sense of the following theorem.
\begin{theorem}
Assume that the class of discrete convex domains $(\mathcal{C}^{h,c_{\mathrm{usr}}}_V(Q))_{h>0}$ are such that a uniform bound on the trace inequality independent on $h$ applies. Let $(\Omega_h,u_h,\ell_h, \mathcal{T}_h)_{h>0}$ be solutions to the discrete Problems \eqref{rein:problem:lb_ldc}. Then for every accumulation point $(\Omega,u)$, the domain $\Omega$ solves Problem \eqref{rein:problem:m_ellmin} as $h\to 0$ with $u \in H^1(\Omega)$ is the corresponding eigenfunction.
\end{theorem}

\begin{proof}
Let us proceed step by step.
\begin{enumerate}
\item Using the compactness for discrete convex domains in \cite[Corollary 2.6]{keller2_preprint} there exists a convex domain $\Omega \in \mathcal{C}_V(Q)$, such that the characteristic functions converge in $L^1(\Omega)$. Without the convergence in variation of the characteristic functions, the perimeter is only lower semicontinuous and in particular 
\begin{equation*}
m = \widehat{m} - \ell_{\min}\vert \partial \Omega \vert \ge \widehat{m} - \ell_{\min}\liminf_{h\to0} \vert \partial \Omega_h \vert = \limsup_{h\to0} m_h.
\end{equation*}
\item Using the uniform trace operator for the class $(\mathcal{C}^{h,c_{\mathrm{usr}}}_V(Q))_{h >0}$, we use the same arguments as in the proof of Proposition \ref{rein:theorem:so_existence} to show that the trivially extended functions $\widetilde{u}_h$ admit a convergent subsequence such that 
\begin{align*}
\widetilde{u}_h \rightharpoonup \widetilde{u} \in SBV(Q) \quad &\text{ with } \widetilde{u} \vert_{\Omega} \in H^1(\Omega), \Vert u \Vert_{L^2(\Omega)} = 1.
\end{align*}
Likewise, we can find a function $\ell \in L^2(\partial \Omega)$, such that
\begin{equation}
\int_{\partial \Omega} \ell\, \mathrm{d}s \le \liminf_{h\to0} \int_{\partial \Omega_h} \ell_h\, \mathrm{d}s \le \liminf_{h\to0} m_h \le m
\end{equation}
and
\begin{equation}\label{rein:eq:stability_ldc_lsc}
 J_{\ell_{\min}} (u,h_u) \le J_{\ell_{\min}}(u,\ell) \le \liminf_{h\to0} J_{\ell_{\min}}(u_h,\ell_h)
\end{equation}
for the optimal distribution $h_u \in \mathcal{H}_m(\partial \Omega)$ given by Proposition \ref{lemma:lower_bound_optimal_distribution}.
\item Suppose $\Omega^\star$ is an optimal domain to Problem \eqref{rein:problem:m_ellmin} with eigenfunction $u^\star \in H^1(\Omega)$ and optimal distribution of insulating material $\ell^\star \in\mathcal{H}_m(\partial \Omega)$.
 Using Proposition \ref{rein:prop:consistency_lb_ldc}, density arguments and Reshetnyak's theorem \cite[Theorem 2.5]{BG} we can find a sequence of admissible triples $(\Omega_h^\star, u^\star_h, \ell^\star_h)$, such that
\begin{equation*}
\lim_{h\to0} J_{\ell_{\min}}(\Omega_h^\star,u^\star_h,\ell^\star_h) = J_{\ell_{\min}}(\Omega^\star,u^\star,\ell^\star) .
\end{equation*}
With \eqref{rein:eq:stability_ldc_lsc}, we then have
\begin{equation*}
 J_{\ell_{\min}}(\Omega^\star,u^\star,\ell^\star)= \lim_{h\to0} J_{\ell_{\min}}(\Omega_h^\star,u^\star_h,\ell^\star_h) \ge \liminf_{h\to0} J_{\ell_{\min}}(\Omega_h,u_h,\ell_h) \ge J_{\ell_{\min}}(\Omega,u,\ell_{u})
 \end{equation*}
which implies the optimality of the limit $\Omega$ with eigenfunction $u \in H^1(\Omega)$ and optimal distribution $\ell_u\in\mathcal{H}_m(\partial \Omega)$.
\end{enumerate}
\end{proof}

\section{Numerical experiments} \label{sec:num_exp}

In this section we first consider the numerical approximation of the eigenfunction on $B_1(0) \subset \mathbb{R}^3$, for several values $m < m_0$ where $m_0$ is the critical value, such that, for the eigenvalue without a lower bound, symmetry breaking occurs. For $m > m_0$, the optimal distribution is constant, such that the lower bound becomes redundant. \\
Then we approximate the optimal domains for the eigenvalues arising in optimal insulation both with and without a lower bound. For ease of notation when $\ell_{\min} =0 $ we set $\widehat{\lambda}_{\widehat{m},0}: = \lambda_{\widehat{m}}(\Omega)$, approximated as described in \cite{BB19} with regularization parameter $\varepsilon = N^{-1/d}/10$ for the boundary term, with $d = 3$ and $N$ the number of nodes in the triangulation.

\subsection{Approximation of the eigenvalue with a lower bound on the ball}

The critical value of mass $m_0$ relates to the first non-trivial Neumann eigenvalue $\mu_2$ and satsfies $\lambda_{m_0}(B_1(0)) = \mu_2(B_1(0))$, see \cite{BBN17}.
For the approximation in $\mathbb{R}^3$, we can estimate that $m_0\approx 5.7963$. Precise values for $m_0$ can be found in \cite[Theorem 1.6]{huang}. We consider $\widehat{m} \in \{2,\dots , 5\}$ and $\ell_{\min}=(\widehat{m}/|\partial\Omega|)q$, $q\in[0,1]$, such that the mass fraction $q$ of the total mass $\widehat{m}$ is fixed through a constant distribution. In Figure \ref{rein:ex:eigenvalues_l_min_3d} the eigenvalues for $q = i/20, i = 0, \dots, 20 $ are shown. If $i = 0,$ i.e. $\ell_{\min} = 0$, this corresponds to the eigenvalue arising from the Dirichlet boundary conditions \eqref{eq:rayleigh_oi}.
For $i = 20$, the insulating material is distributed evenly, resulting in Robin boundary conditions.
Following the results of Proposition \ref{prop:sym_breaking_lb} we observe symmetry breaking in the optimal insulation occurs for all $q < 1$ and $\widehat{m} < m_0$.
For $\ell_{\min} $ small, especially if $\widehat{m}$ is small as well, the increase of the eigenvalue with respect to $\ell_{\min}$ is less distinct and showing a clear spike for $\widehat{m} = 2$ at around $\ell_{\min} \approx 0.03$. This is most likely caused by numerical inaccuracies as indicated by the error estimates in Section \ref{sec:oi:conv} which depend on $\ell_{\min}^{-2}$. In two-dimensional experiments, where a higher resolution of the domain is possible, this is no longer observed. The optimal asymmetric distributions of insulating material are shown in Figure \ref{rein:ex:ball_3d} for $\widehat{m} \in \{2,3,5\}$, and $q = i/N$, for $N=5$, $i=0,\dots,N$.

\begin{figure}[h] 
\begin{center}
     \includegraphics[width = 6cm]{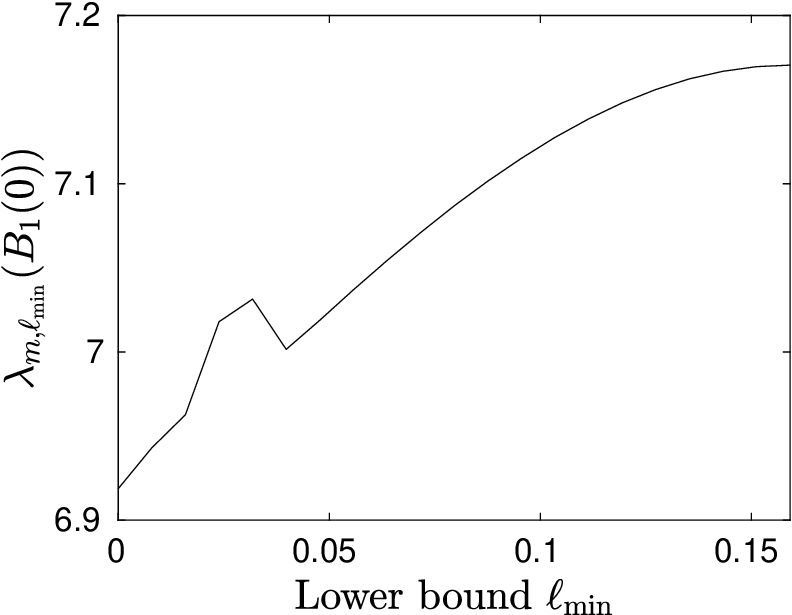}\quad
     \includegraphics[width = 6cm]{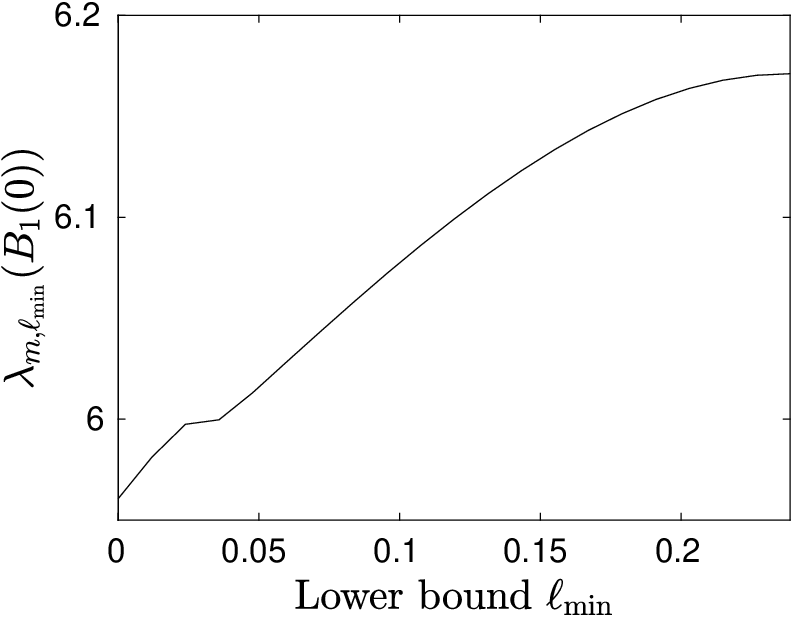}\\ \ \\
     \includegraphics[width = 6cm]{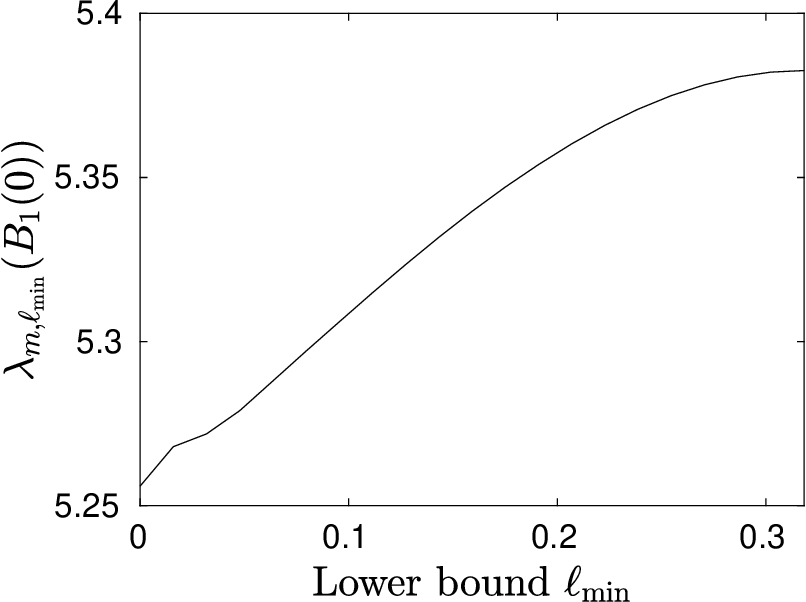}\quad
     \includegraphics[width = 6cm]{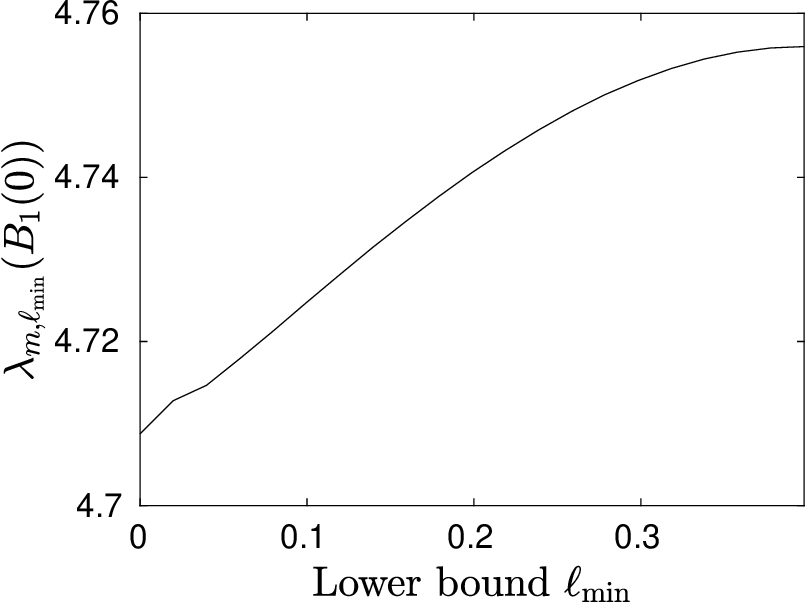}
  \caption{Eigenvalues $\lambda_{m,\ell_{\min}}(B_1(0))$ with $B_1(0) \subset \mathbb{R}^3$ for $\widehat{m} = \{2, \dots, 5\}$ and $m = \widehat{m} - \ell_{\min} \vert B_1(0)\vert $ (left to right, top to bottom) as functions of $\ell_{\min} = (\widehat{m}/ \vert \partial B_1(0) \vert) q$, for $ q \in [0,1]$, evaluated for $q = i/20, i = 0, \dots, 20 $, approximated on a triangulation with mesh size $h = 2^{-3}$. }\label{rein:ex:eigenvalues_l_min_3d}
\end{center}
\end{figure}

\begin{figure}[h]
\begin{center}
\begin{minipage}{0.3\textwidth}
\begin{center}
 \includegraphics[width = 0.95\textwidth]{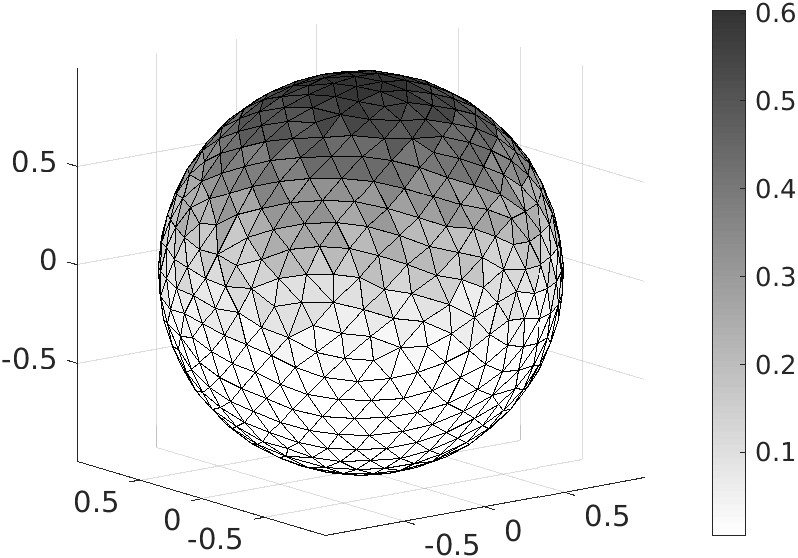}
     \includegraphics[width = 0.95\textwidth]{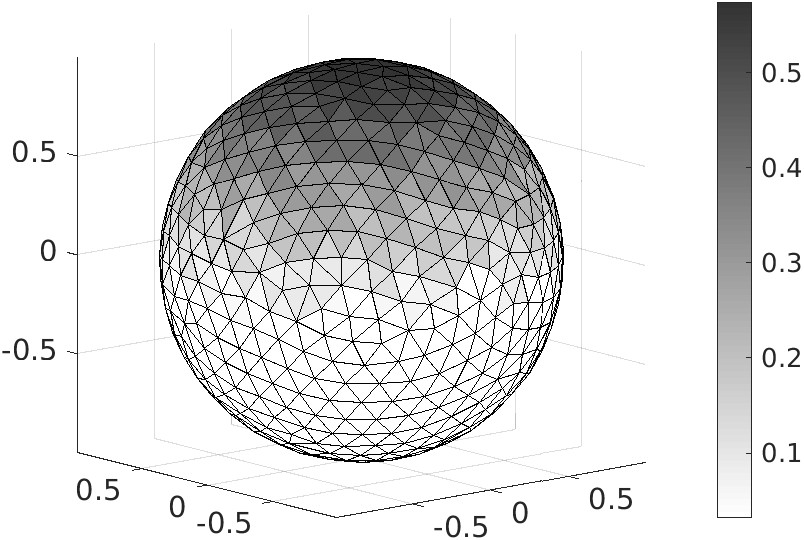}
     \includegraphics[width = 0.95\textwidth]{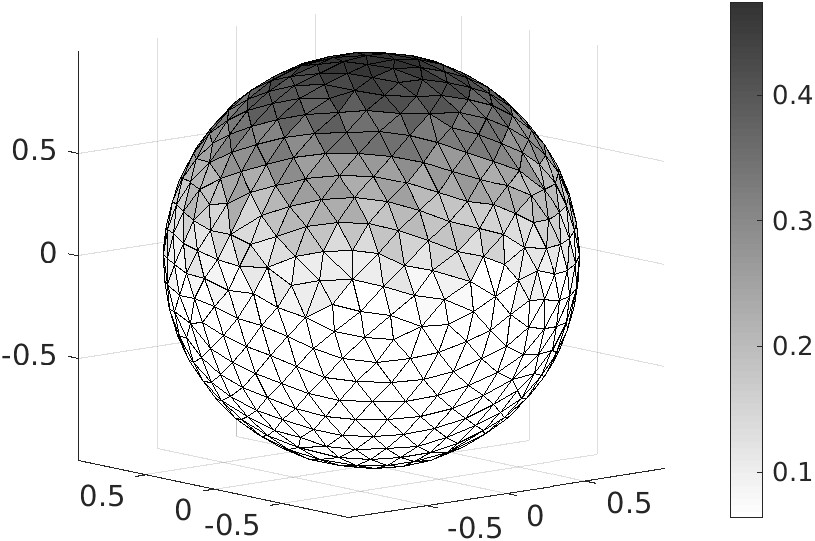}
     \includegraphics[width = 0.95\textwidth]{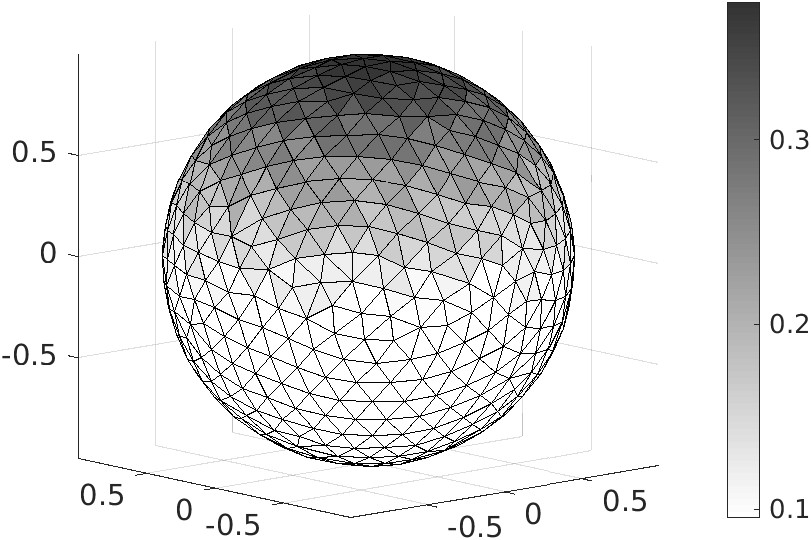}
     \includegraphics[width = 0.95\textwidth]{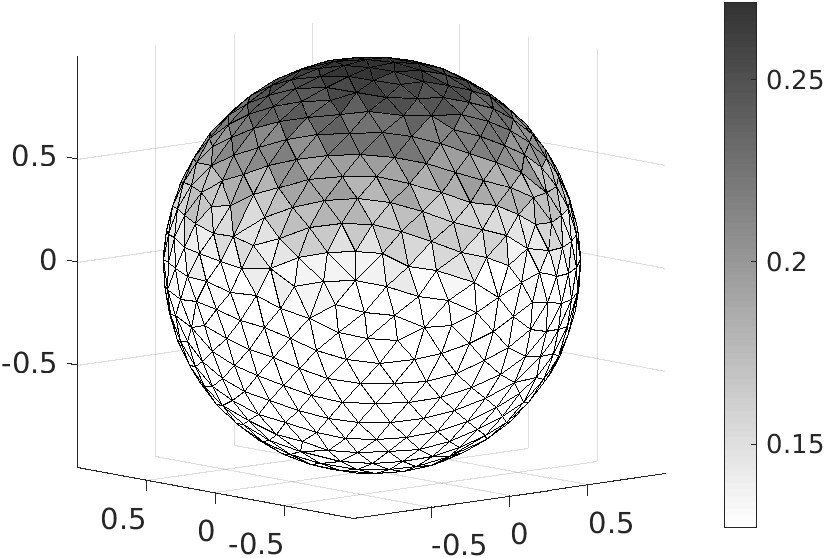}
\end{center}
\end{minipage}
\vline
\begin{minipage}{0.3\textwidth}
\begin{center}
 \includegraphics[width = 0.95\textwidth]{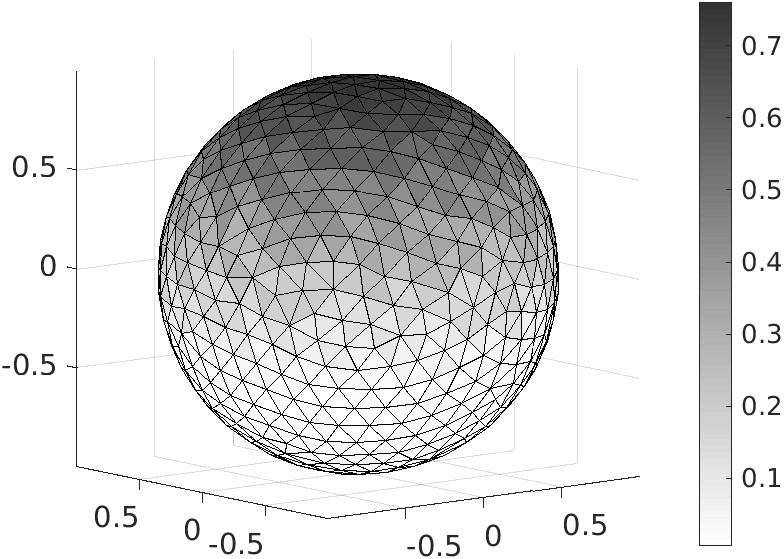}
     \includegraphics[width = 0.95\textwidth]{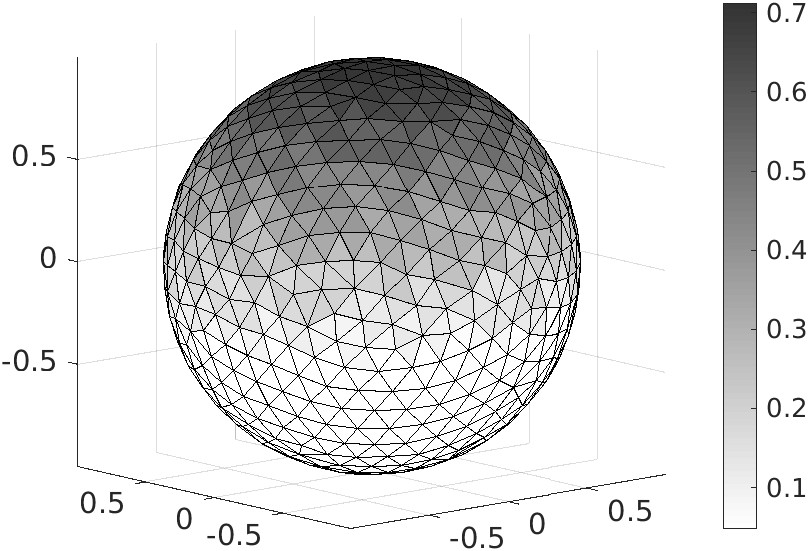}
     \includegraphics[width = 0.95\textwidth]{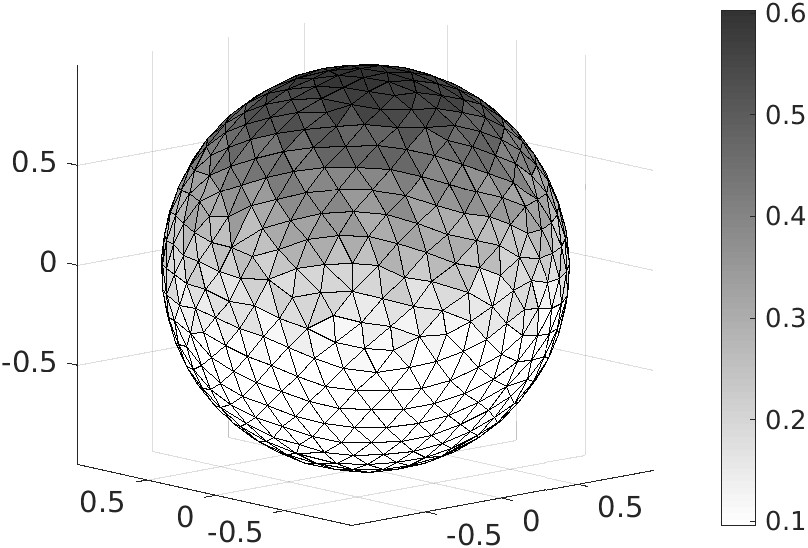}
     \includegraphics[width = 0.95\textwidth]{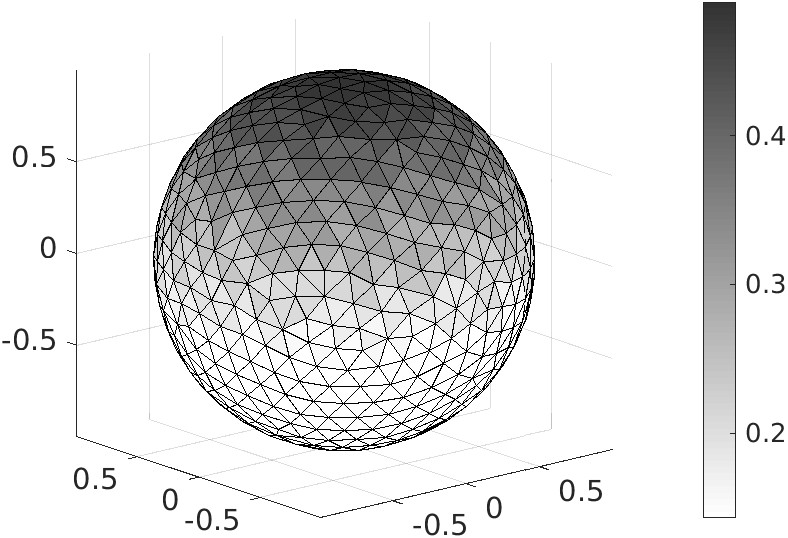}
     \includegraphics[width = 0.95\textwidth]{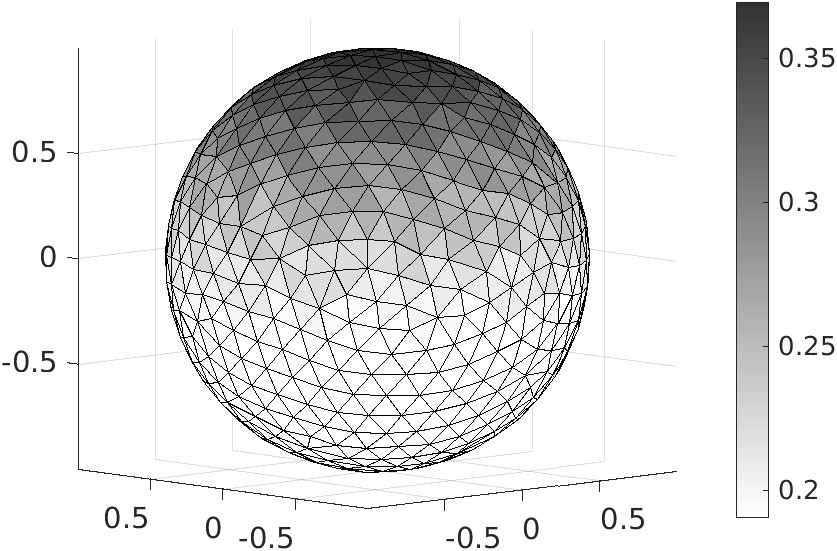}
\end{center}
\end{minipage}
\vline
\begin{minipage}{0.3\textwidth}
\begin{center}
 	\includegraphics[width = 0.95\textwidth]{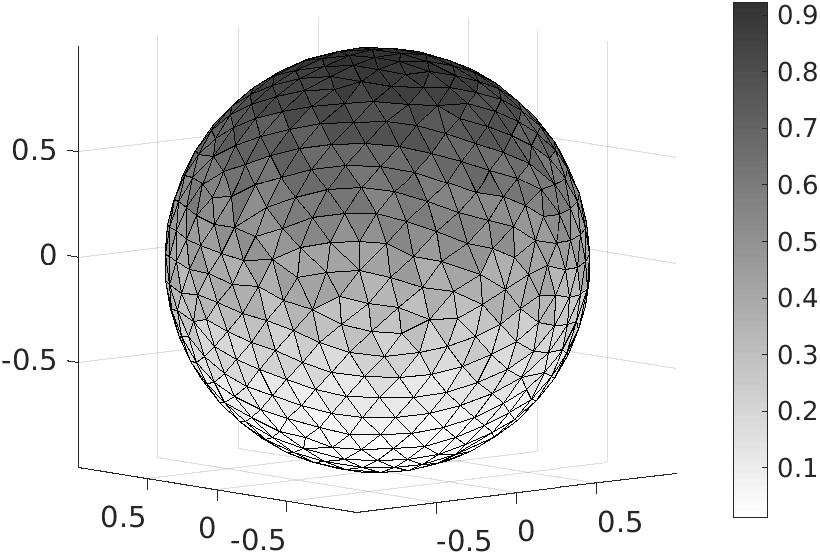}
     \includegraphics[width = 0.95\textwidth]{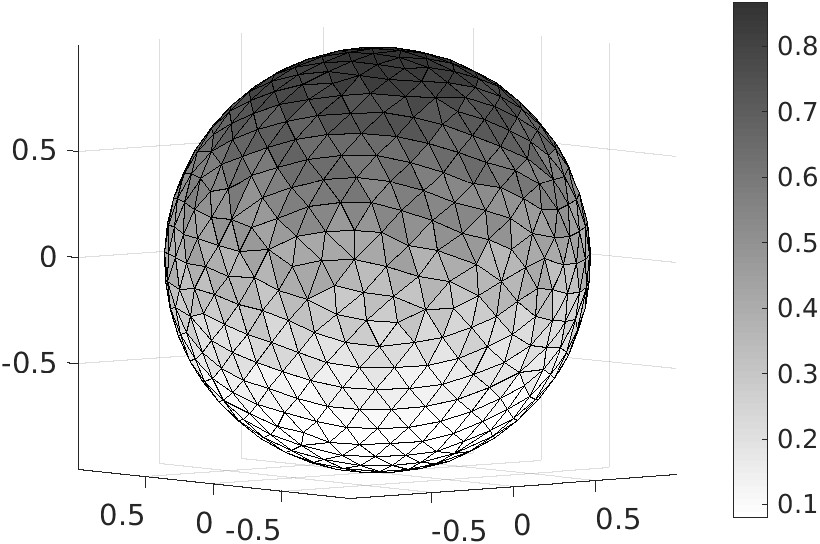}
     \includegraphics[width = 0.95\textwidth]{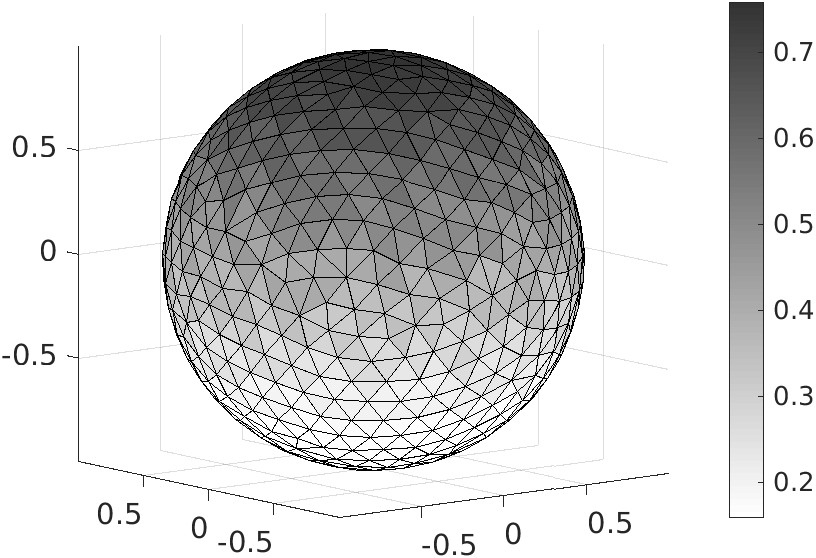}
     \includegraphics[width = 0.95\textwidth]{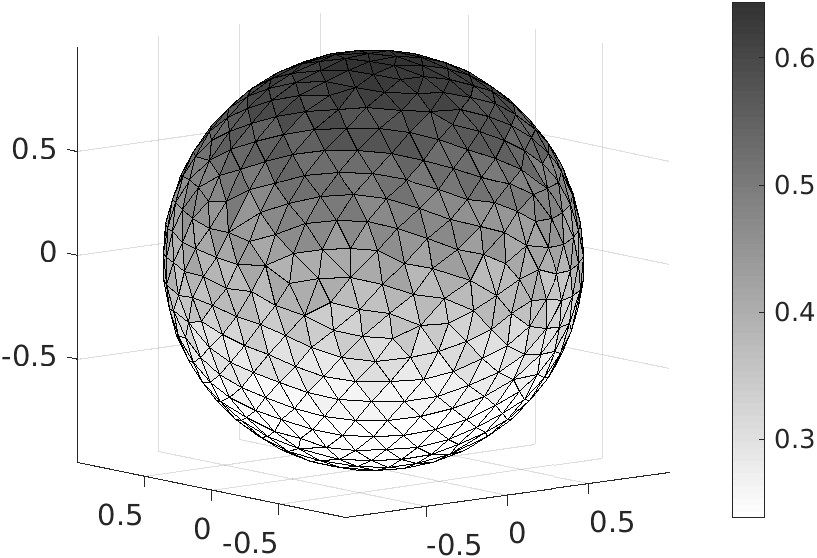}
     \includegraphics[width = 0.95\textwidth]{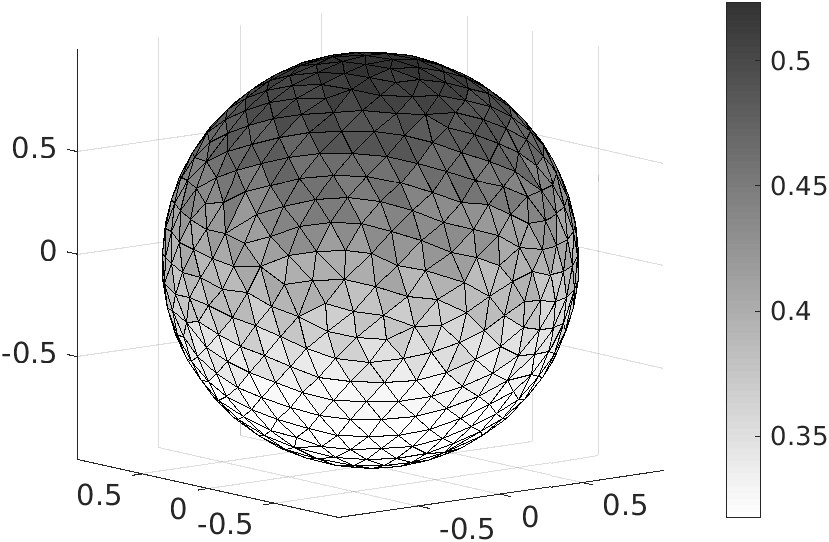}
\end{center}
\end{minipage}

\caption{Optimal distribution to the eigenvalue $\lambda_{m,\ell_{\min}}(B_1(0))$ with $B_1(0) \subset \mathbb{R}^3$ for $\widehat{m}\in\{2,3,5\}$ and $m=\widehat{m}-\ell_{\min}|B_1(0)|$ (left to right), and $\ell_{\min}=(\widehat{m}/ \vert \partial B_1(0) \vert) q$, for $ q = i/5$, for $ i = 0, \dots ,4$ (top to bottom), the optimal distribution $\widehat{\ell} = \ell_{\min} + h_{u}$ is shown on the boundary; approximated on a triangulation with mesh size $h = 2^{-3}$. The optimal distributions appear to be non-radial but rotationally symmetric for all $q<5$, i.e. where the free mass $m = \widehat{m} - \vert \partial B_1(0) \vert \ell_{\min} >0$.}\label{rein:ex:ball_3d}
\end{center}
\end{figure}

\subsection{Shape optimization}

We optimize the eigenvalue $\widehat{\lambda}_{\widehat{m},\ell_{\min}}$ within the class $\mathcal{C}_V(Q)$. To reduce the number of constraints we drop volume constraint and optimize the scaled eigenvalues instead. We then use the algorithm as described in \cite{keller2_preprint}, and compute the deformation via a problem of linear elasticity with elasticity parameters $E = 0.5,\nu =0.2$ for Young's module, and the damping parameter $\rho = 0.5$ to ensure the coercivity of the elasticity bilinear form.\\
From \cite{BBN17, keller, BB19} we expect that a breaking of symmetry occurs for the eigenfunctions of the ball and the optimal domains if $\widehat{m}<m_0\approx 5.7963$.

The experiments below where repeated for different initial domains with a mesh size $h = 2^{-3}$ and gave consistent results. For $\ell_{\min} = 0,$ the results where consistent with the optimization among rotationally symmetric domains from \cite{keller}. In particular, it can be observed, that there exist stationary asymmetric domains for $\widehat{m}\in \{2, \dots, 10\}$, and that the ball is stationary for $\widehat{m} \ge 6$.
The optimal domains were approximated for $\widehat{m} \in \{2,\dots ,10\}$, and $\ell_{\min} = ( \widehat{m}/ \vert \partial B_1(0) \vert) q$ for $ q \in [0,1)$, evaluated for $q = i/5, q = 0, \dots, 4 $. For $q=1$, the shape optimization is trivial, see Section \ref{sec:lb_ev}, and is solved by a ball with a constant layer of insulating material. The eigenvalue problem is then equivalent to the Robin eigenvalue problem. \\
For $q = 0$, the shape optimization corresponds to the optimization of the eigenvalue in optimal insulation \eqref{eq:rayleigh_oi}, approximated as in \cite{BB19} with regularization parameter $\varepsilon = N^{-1/d}/10$, where $N$ is the number of nodes in the mesh. 
The optimized eigenvalues are shown in Figure \ref{rein:ex_solb_3d_val_1} and Figure \ref{rein:ex_solb_3d_val_2} for $\widehat{m} <m_0$ and $\widehat{m}>m_0$ respectively, where we include the values for $q = 0$ and $q = 1$ from the values for the regular eigenvalue in optimal insulation $\lambda_{\widehat{m}}(\Omega)$ and for the first Robin eigenvalue on the ball $\lambda_1(B_1(0),\ell_{\min}^{-1})$ with $\ell_{\min} = \widehat{m}/\vert \partial B_1(0) \vert$.\\
In Figures \ref{rein:ex:so_3dm1} and \ref{rein:ex:so_3dm2} we show the optimized asymmetric domains for $\widehat{\lambda}_{\widehat{m},\ell_{\min}}$ for $\widehat{m} \in \{2,\dots 9\}$ and different values of $\ell_{\min} \in [0, \widehat{m}/ \vert \partial B_1(0) \vert )$. 
 For $\widehat{m} < m_0$, they are compared to the respective eigenvalues in the unit ball. For $m > m_0$, Theorem \ref{thm:sym_breaking} implies that the optimal distribution on the ball is always constant for any $\ell_{\min}$, such that the lower bound becomes redundant. \\
For $\ell_{\min} = 0$, i.e. the classical eigenvalue in optimal insulation, the results are consistent with our expectations and the previous numerical experiments \cite{keller}. In particular, all approximated domains and the corresponding distribution of insulating material are approximately rotationally symmetric, even if the rotational symmetry is not enforced in the approximation with local discrete convex domains. To ensure that no bias occurs due to only using rotationally symmetric initial domains, we also optimized initial domains which are not rotationally symmetric. \\
For $\ell_{\min}>0$ similar patterns regarding the breaking of symmetry can be observed.
The experiments suggest, that as $q\to1$ and vanishing free mass $m\to0$, the optimal domains transform from the asymmetric domain to the ball, which is optimal if $ m= 0$. This corresponds to a constant distribution of insulating material. For $\widehat{m} < m_0$ the domains appear to be asymmetric for all values of $q$, whereas if $\widehat{m} >m_0$ for larger $q$, the approximated asymmetric domain is not optimal compared to the ball with a constant distribution or no asymmetric domain is approximated and the ball appears to be optimal. \\ For $q = 0$, the optimal domains often tend to be non-smooth, see also \cite{keller}. With increasing $q$ the kink in the boundary seems to smoothen and in general for $q >0$, a singular behaviour of the boundary cannot be observed.
Another interesting observation lies in the thickness of the optimal distributions.
 For $\ell_{\min}=0$ it was observed that for $\widehat{m} <m_0$, the insulating material is thickest around the pointed end of the domains, while for larger $\widehat{m}$, this behaviour reverses, such that the blunt end is insulated better. For $\ell_{\min}>0$ , this pattern can be observed as well and the location, where the insulating material is thickest, is observed to depend on whether $\widehat{m} < m_0$ or $\widehat{m} >m_0$.

\begin{figure}[bh!] 
\begin{center}
\begin{minipage}{0.49\textwidth}
\begin{center}
	\includegraphics[height = 4cm]{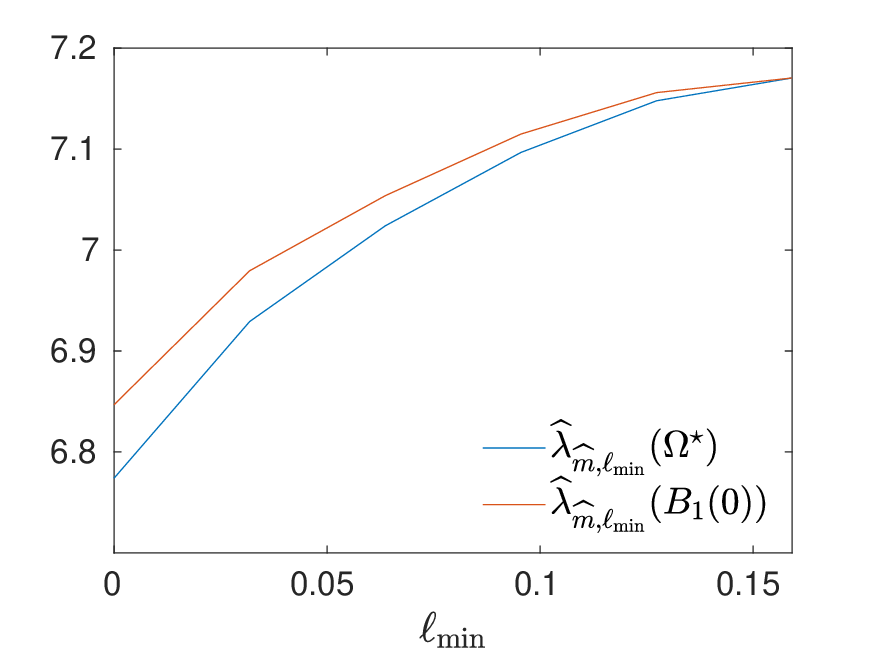}
     \includegraphics[height = 4cm]{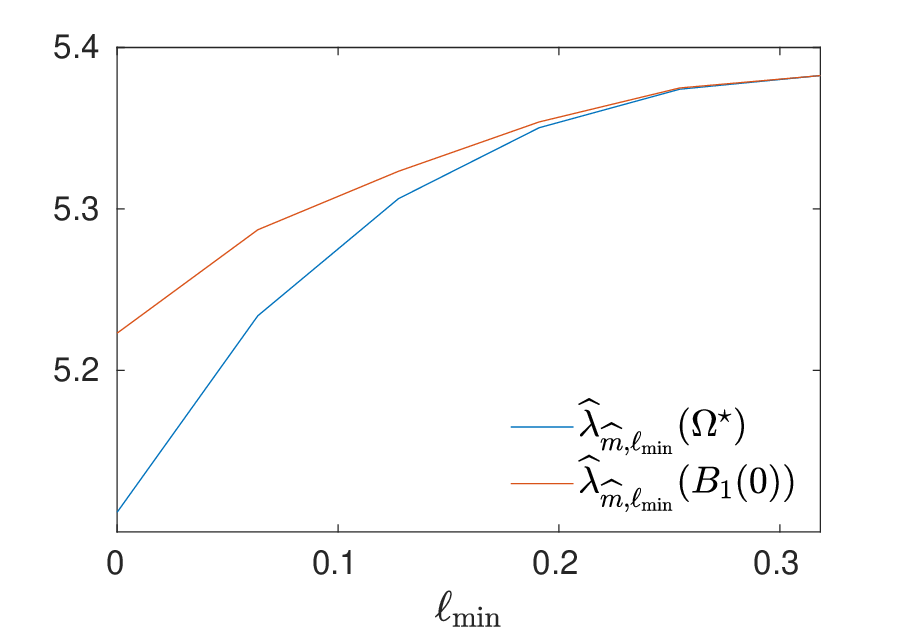}
     \end{center}
     \end{minipage}
     \begin{minipage}{0.49\textwidth}
\begin{center}
     \includegraphics[height = 4cm]{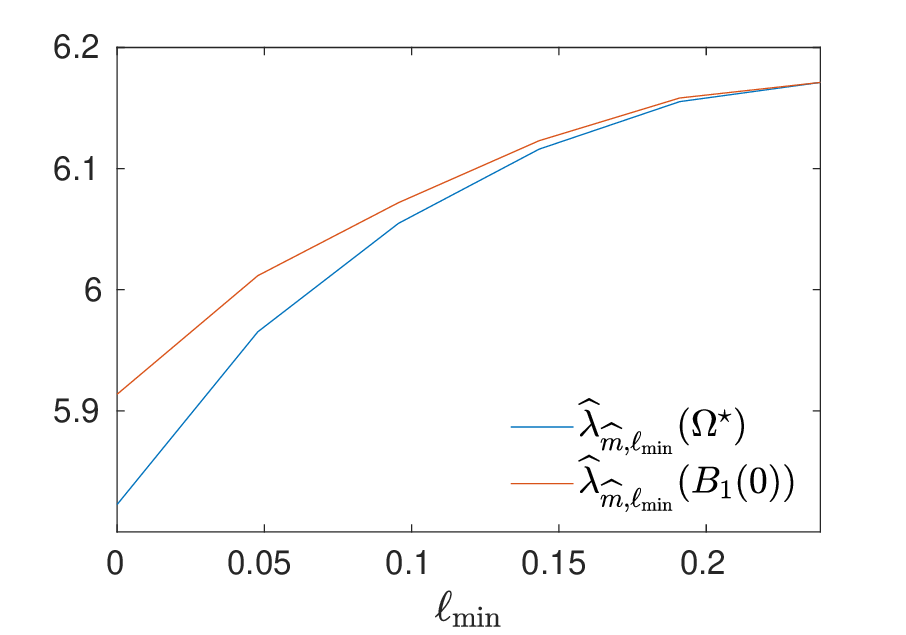}
     \includegraphics[height = 3.8cm]{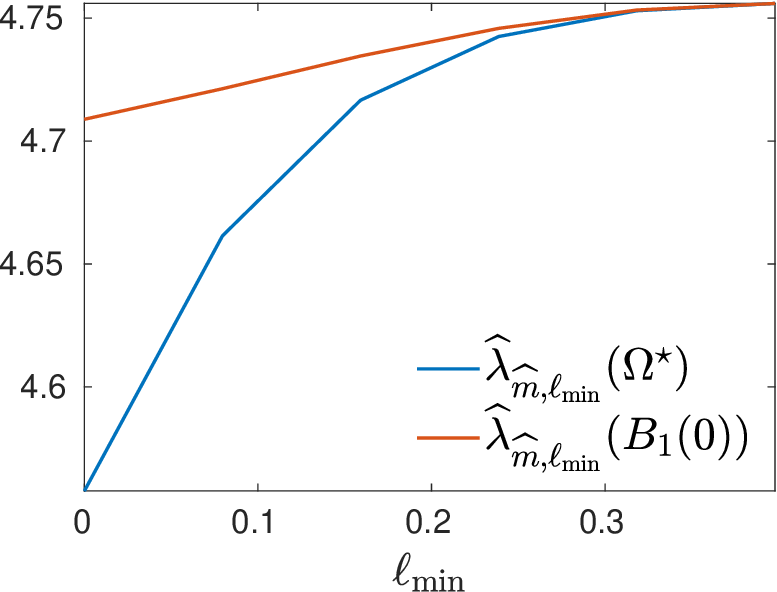}
\end{center}
\end{minipage}
\caption{Eigenvalues $\widehat{\lambda}_{\widehat{m},\ell_{\min}}$ of the initial domain $B_1(0) \subset \mathbb{R}^3$ and the approximated optimal domains $\Omega^\star$ with $\ell_{\min}=(\widehat{m}/|\partial B_1(0)|)q$, for $q\in[0,1]$, evaluated for $q=i/5$, $i=0,\dots,5$, for $\widehat{m}\in\{2,3,4,5\}$ (left to right, top to bottom).}\label{rein:ex_solb_3d_val_1}
\end{center}
\end{figure}

\begin{figure}[h] 
\begin{center}
	 \includegraphics[width = 6cm]{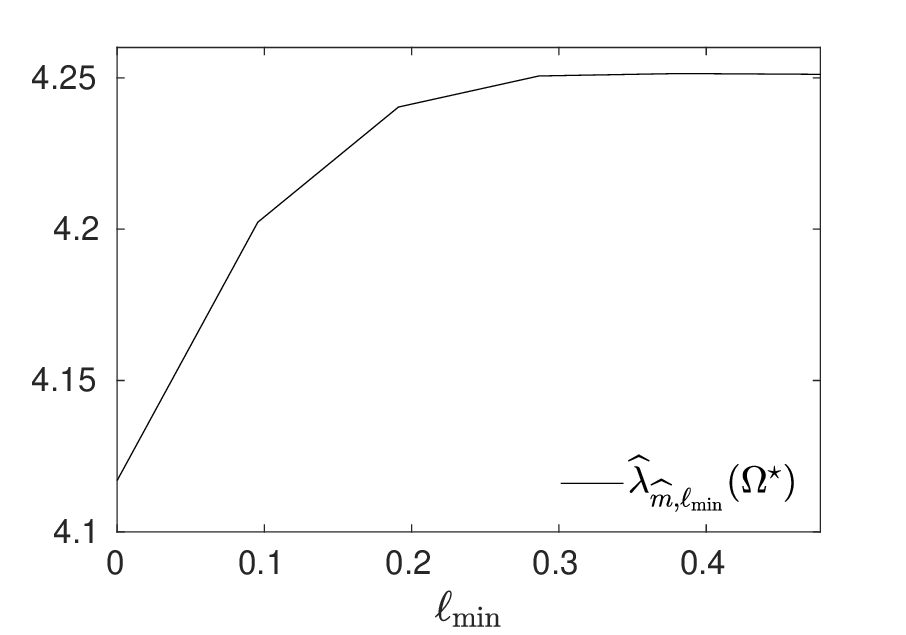}
     \includegraphics[width = 6cm]{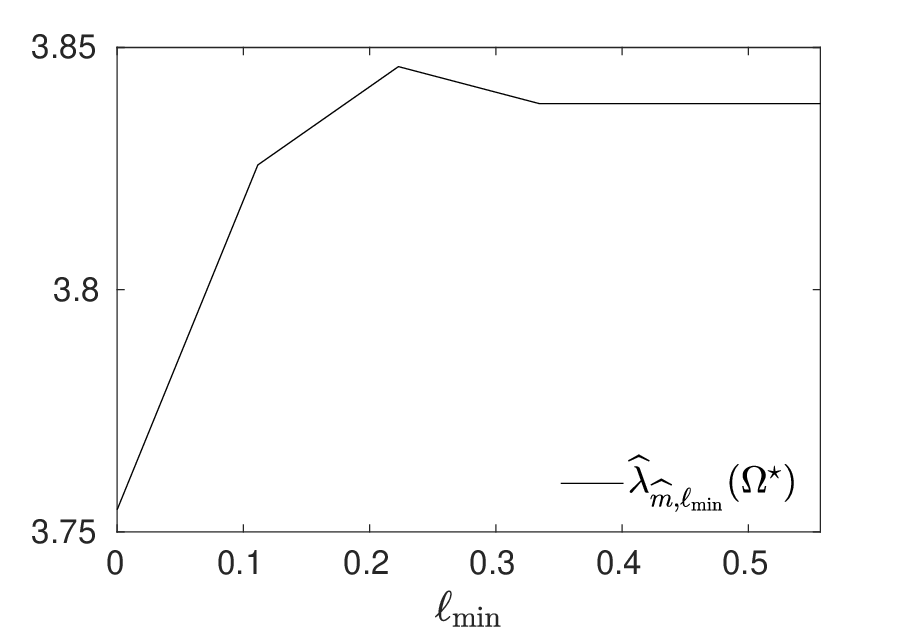}
     \includegraphics[width = 6cm]{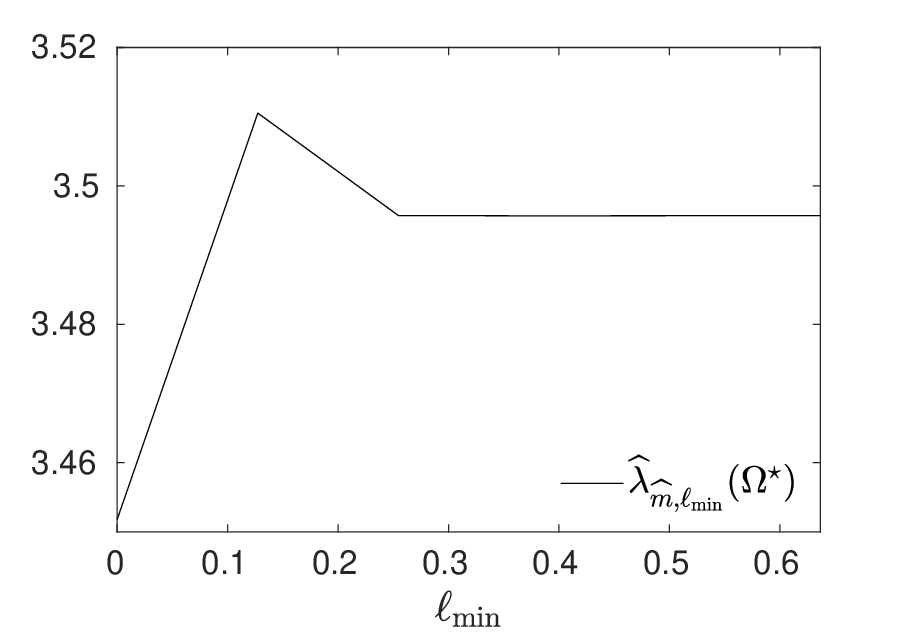}
     \includegraphics[width = 6cm]{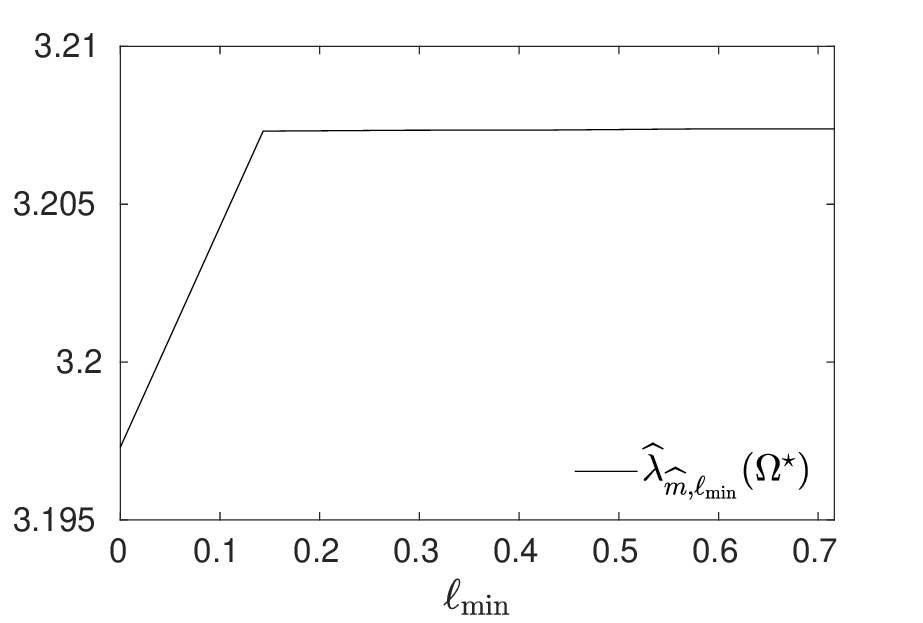}
  \caption{Eigenvalues $\widehat{\lambda}_{\widehat{m},\ell_{\min}}$ of the approximated optimal domains $\Omega^\star$ with $\ell_{\min}=(\widehat{m}/|\partial B_1(0)|)q$, for $q\in[0,1]$, evaluated for $q=i/5$, $i=0,\dots,5$, for $\widehat{m}\in\{6,\dots,9\}$ (left to right, top to bottom).} \label{rein:ex_solb_3d_val_2} \end{center}
\end{figure}

\begin{figure}[h] 
\begin{center}
\begin{minipage}{0.24\textwidth}
\begin{center}
	\includegraphics[ width = 0.98\textwidth]{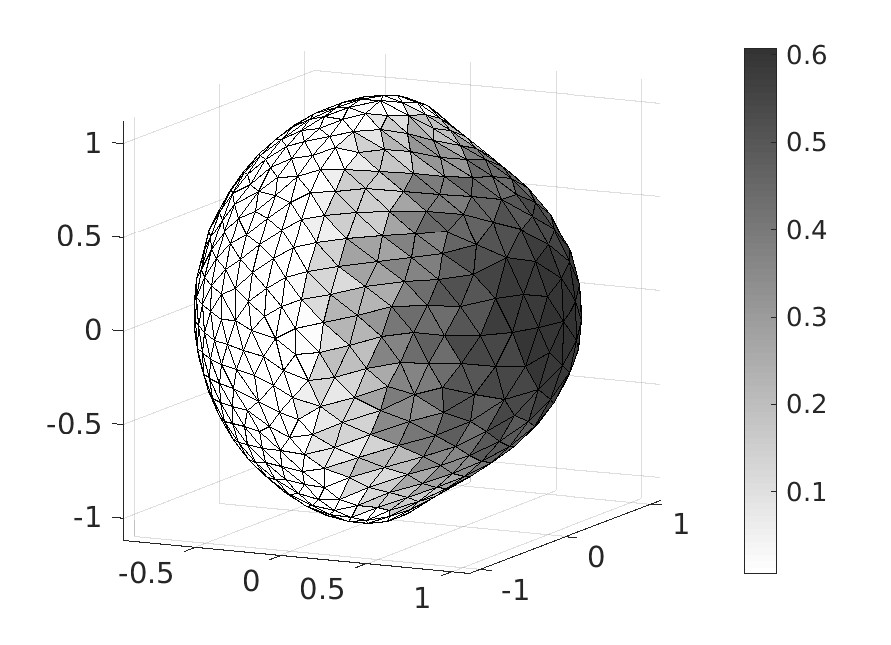}
     \includegraphics[width = 0.98\textwidth]{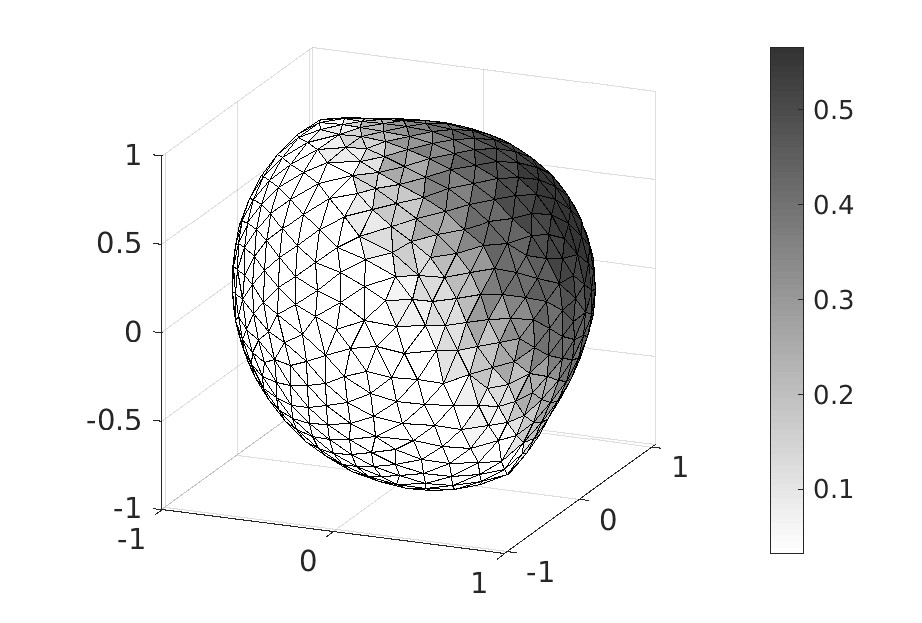}
     \includegraphics[width = 0.98\textwidth]{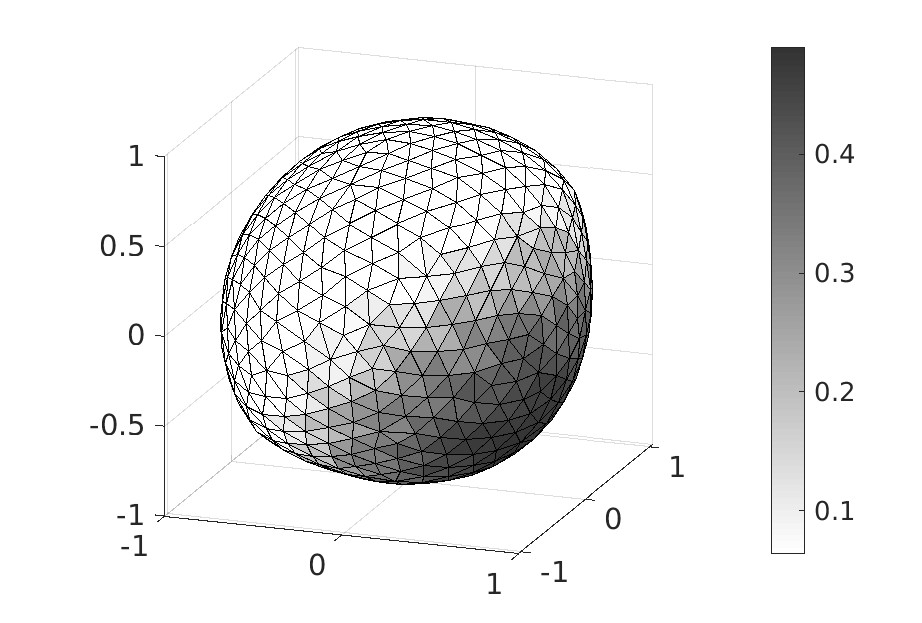}
     \includegraphics[width = 0.98\textwidth]{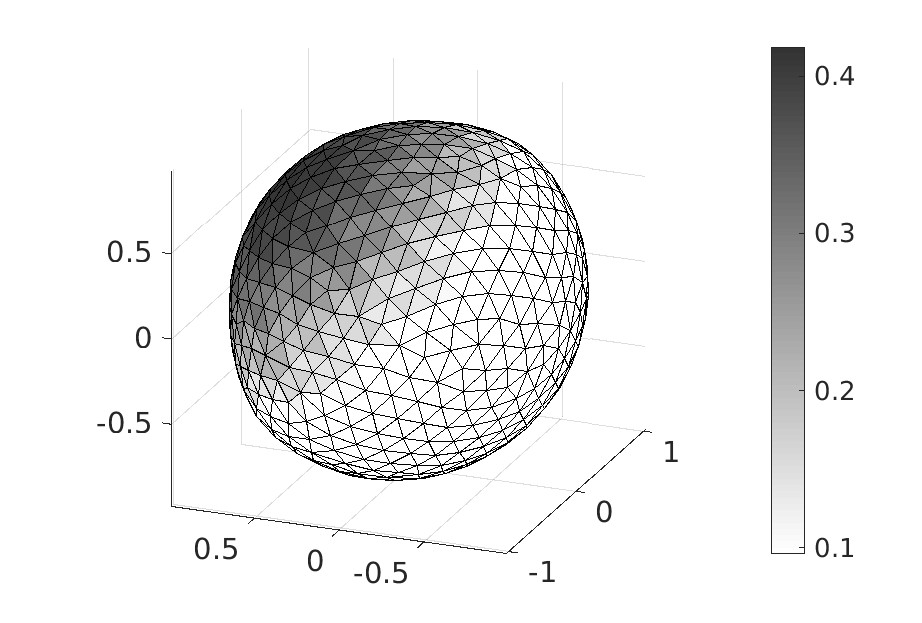}
     \includegraphics[width = 0.98\textwidth]{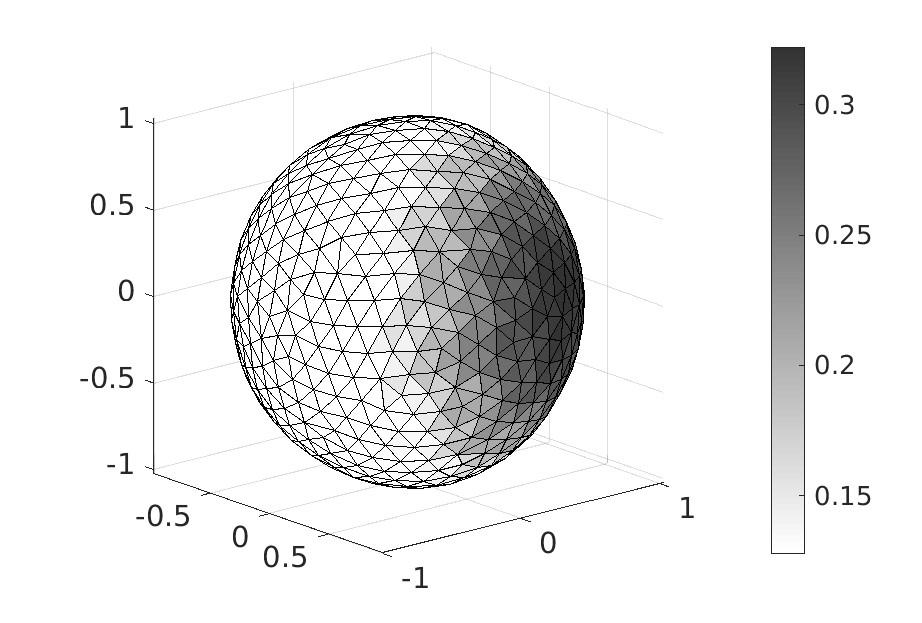}
\end{center}
\end{minipage}
\vline
\begin{minipage}{0.24\textwidth}
\begin{center}
\includegraphics[width = 0.98\textwidth]{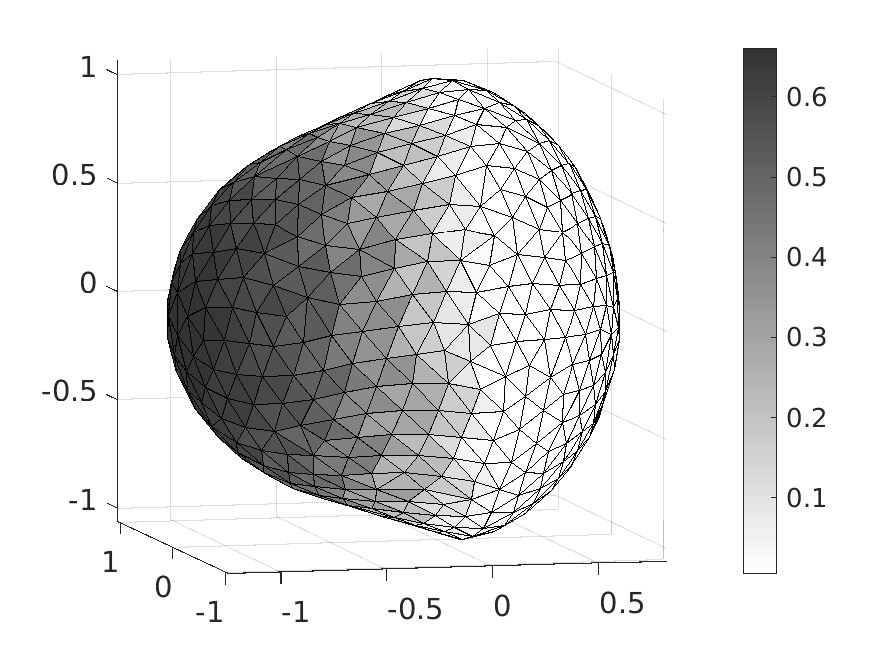}
     \includegraphics[width = 0.98\textwidth]{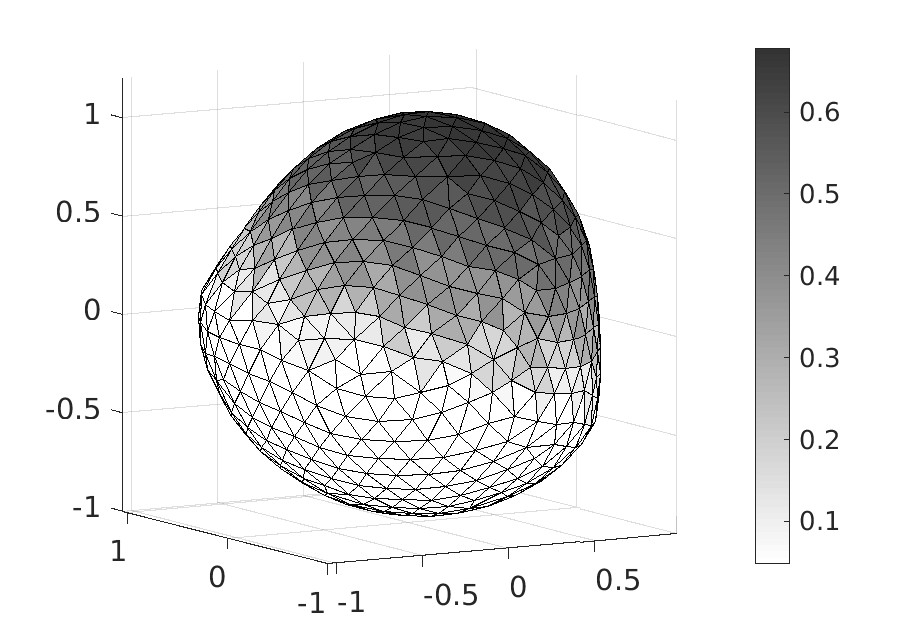}
     \includegraphics[width = 0.98\textwidth]{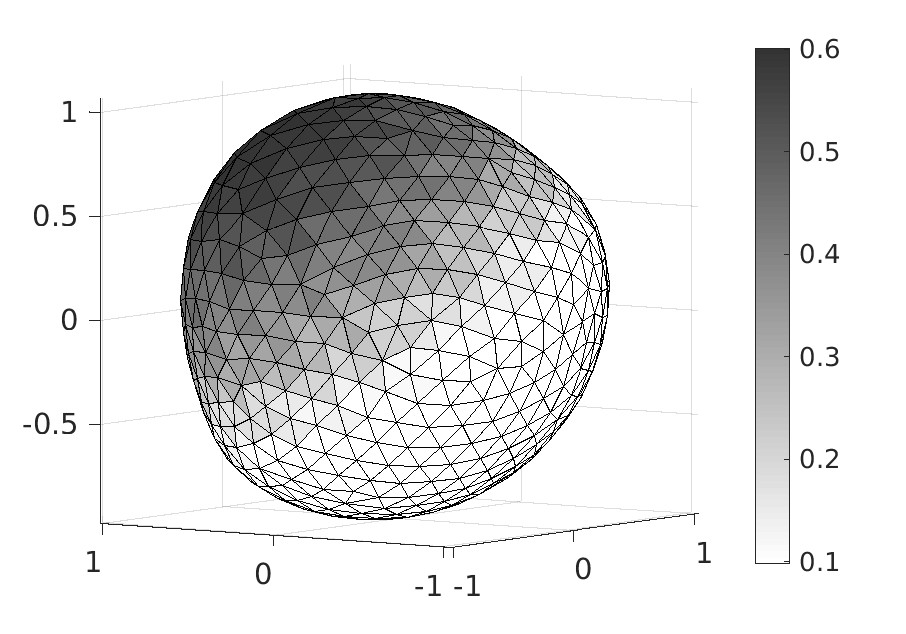}
     \includegraphics[width = 0.98\textwidth]{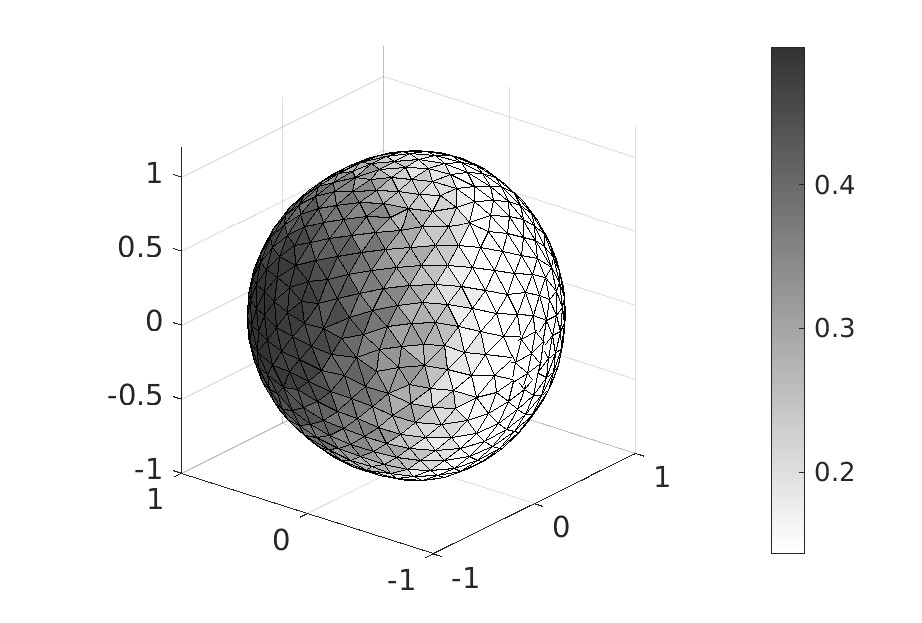}
     \includegraphics[width = 0.98\textwidth]{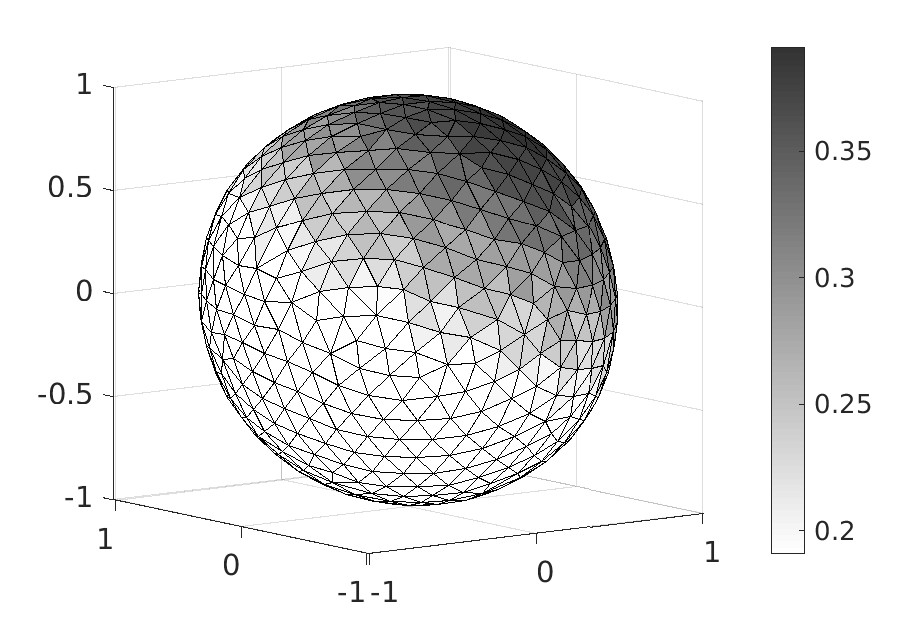}
\end{center}
\end{minipage}
\vline
\begin{minipage}{0.24\textwidth}
\begin{center}
\includegraphics[width = 0.98\textwidth]{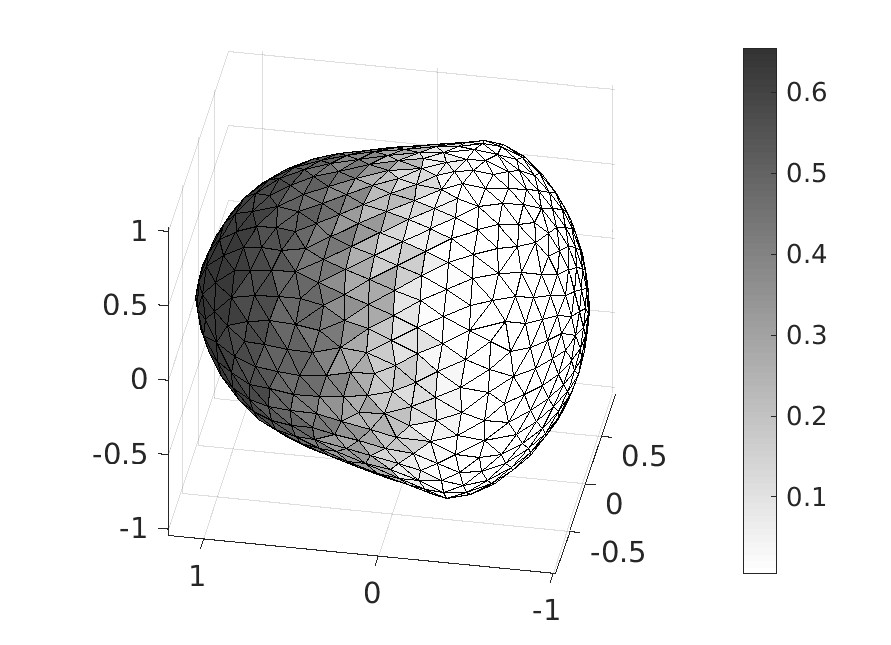}
	\includegraphics[width = 0.98\textwidth]{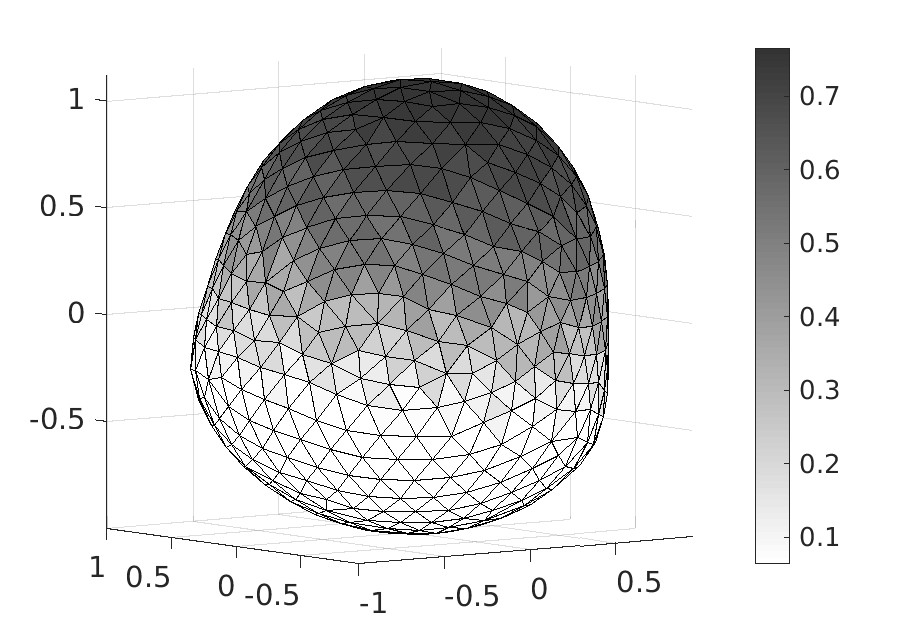}
    \includegraphics[width = 0.98\textwidth]{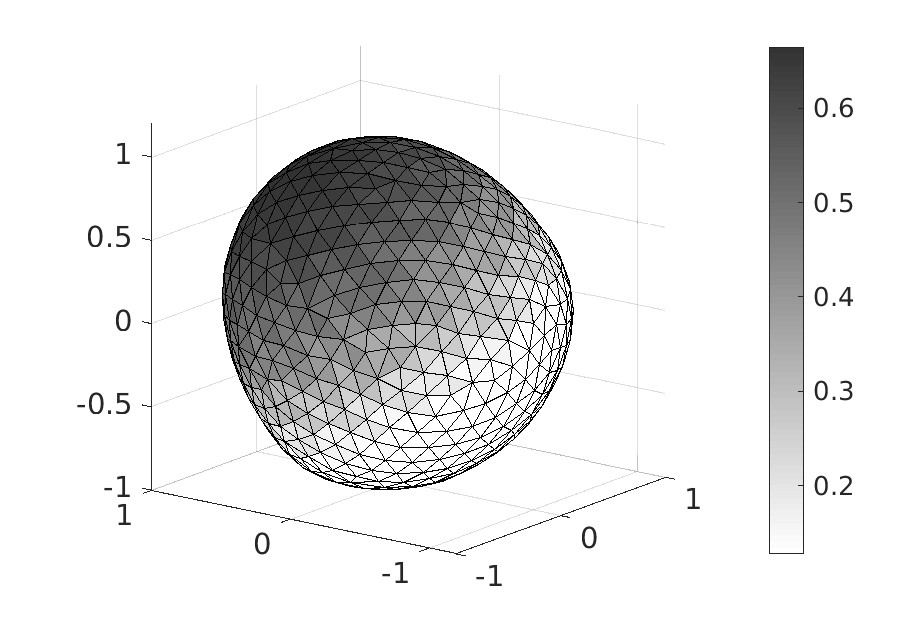}
    \includegraphics[width = 0.98\textwidth]{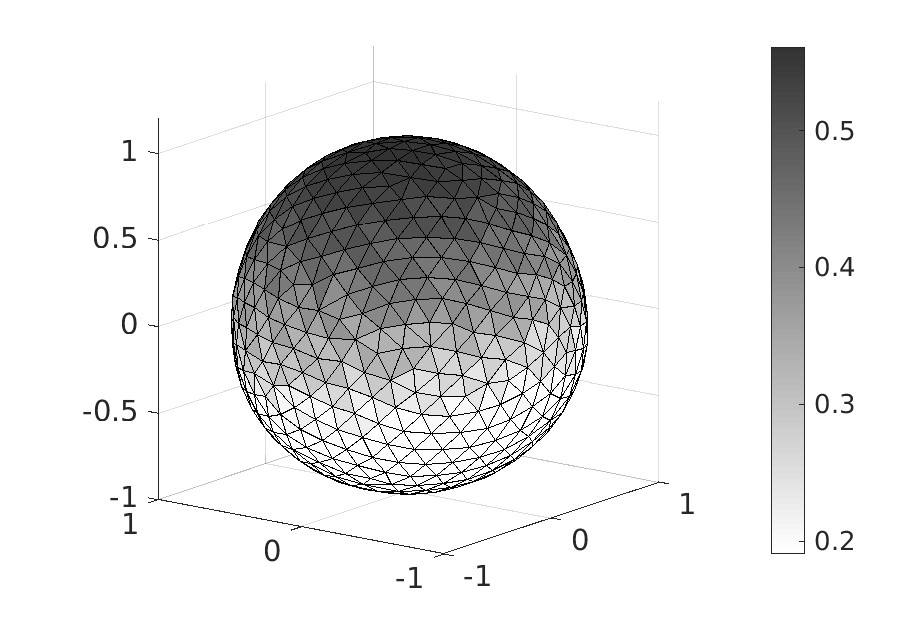}
    \includegraphics[width = 0.98\textwidth]{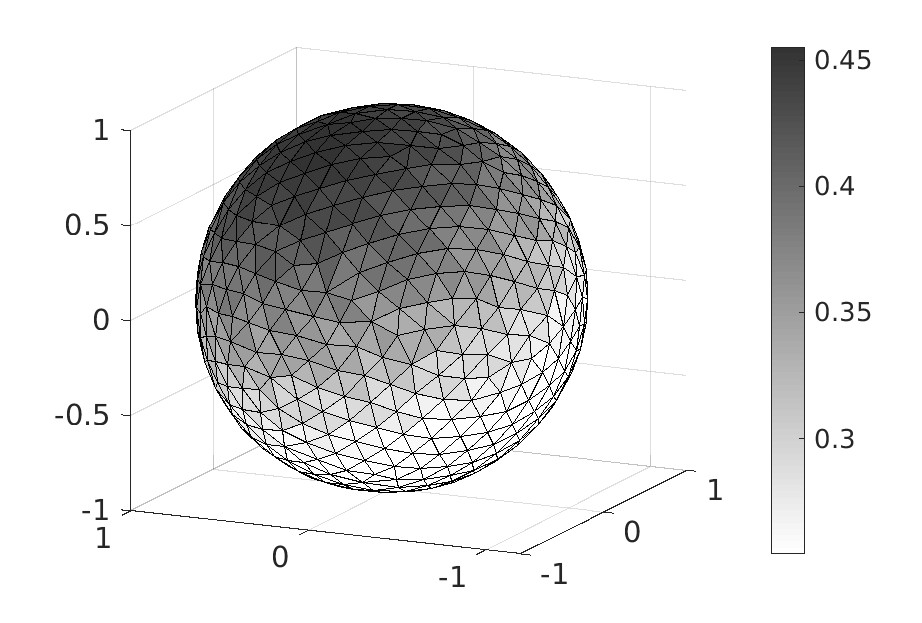}
\end{center}
\end{minipage}
\vline
\begin{minipage}{0.24\textwidth}
\begin{center}
\includegraphics[width = 0.98\textwidth]{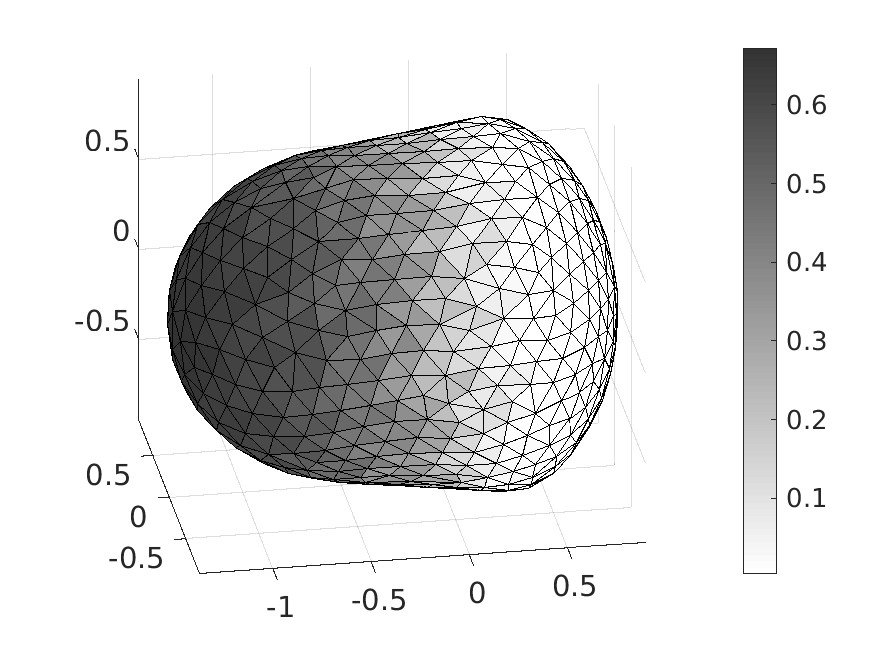}
	\includegraphics[width = 0.98\textwidth]{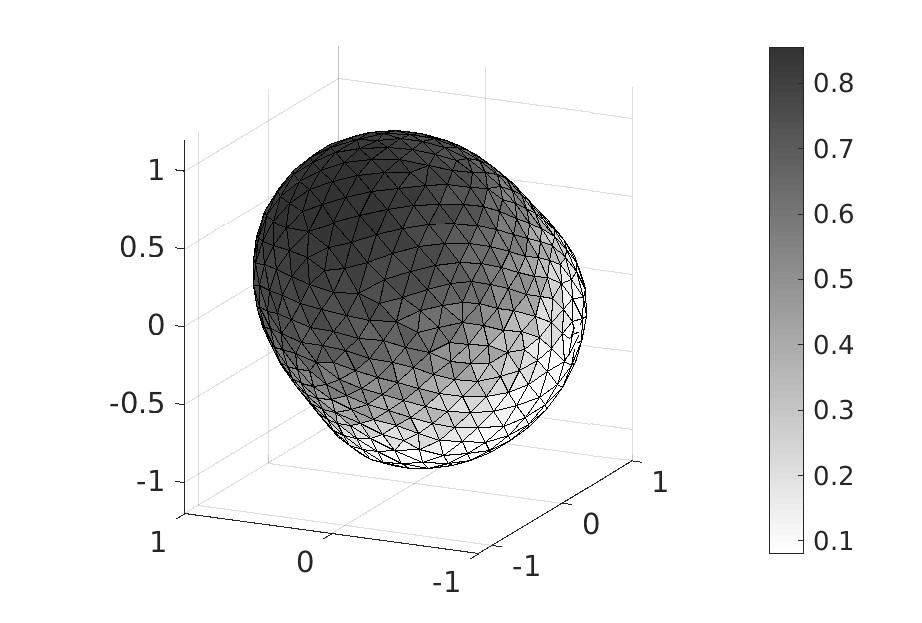}
    \includegraphics[width = 0.98\textwidth]{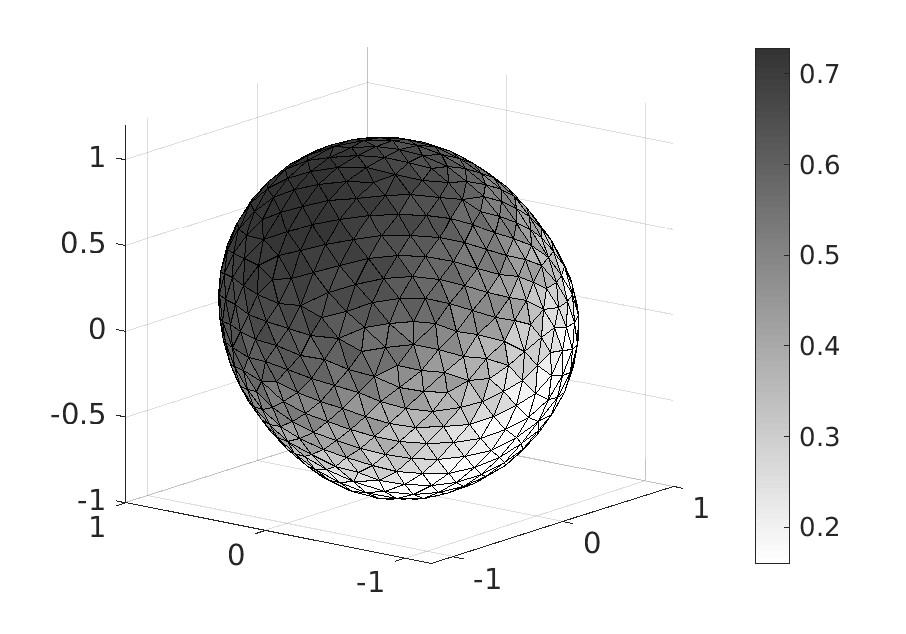}
    \includegraphics[width = 0.98\textwidth]{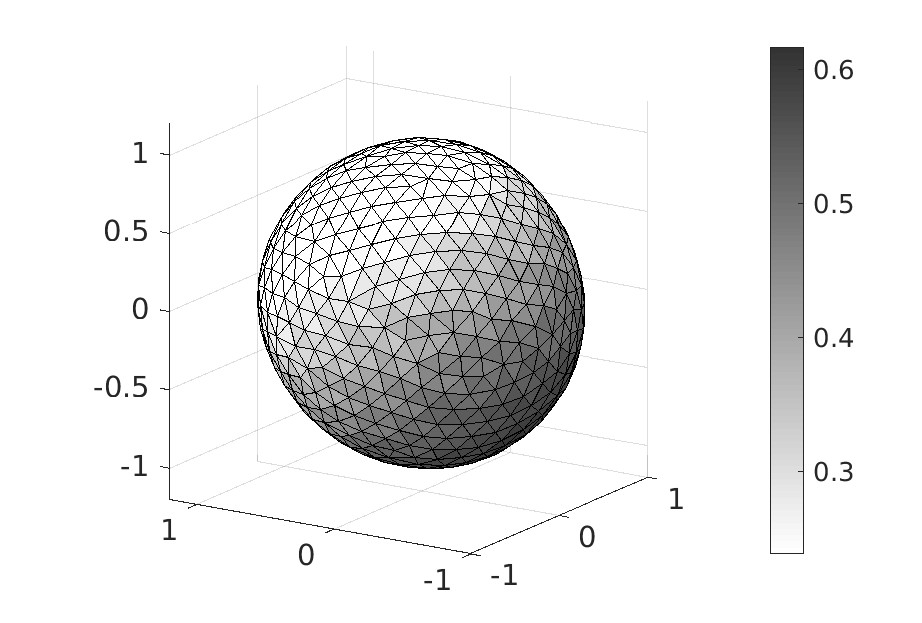}
    \includegraphics[width = 0.98\textwidth]{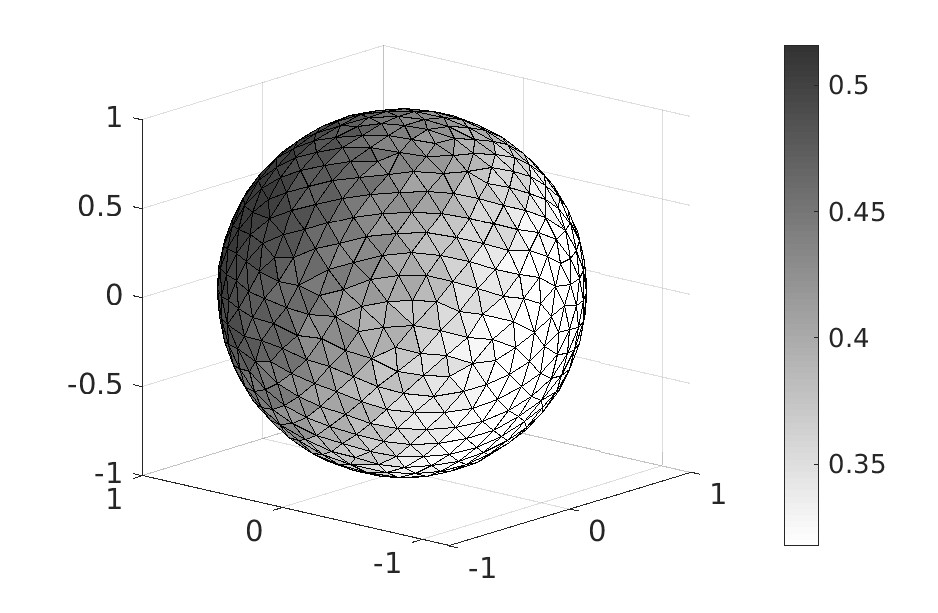}
\end{center}
\end{minipage} 
\caption{Approximated optimal domains for $\widehat{\lambda}_{\widehat{m},\ell_{\min}}$ for $\widehat{m}\in\{2,\dots,5\}$ (left to right) and $\ell_{\min} = (\widehat{m}/|\partial B_1(0)|)q$, for $q\in[0,1)$, evaluated for $q = i/5$, $i=0,\dots,4$ (top to bottom), with initial domain approximated on a mesh with maximal mesh size $h = 2^{-3}$. With increasing $\ell_{\min}$, the optimal domains transform from the asymmetric domain to the ball, and the observed kink in the boundary smoothens. The optimal domains appear to be rotationally symmetric, and symmetry breaking in the optimal insulation can be observed for all values of $\widehat{m}$ and $\ell_{\min} <  (\widehat{m}/|\partial B_1(0)|)$. For the value $ i = 5$ the optimal domains are balls with constant distributions and are not shown here.}\label{rein:ex:so_3dm1}
\end{center}
\end{figure}

\begin{figure}[h]
\begin{minipage}{0.3\textwidth}
\begin{center}
\includegraphics[trim = {1cm 0 1cm 0},width = 0.95\textwidth]{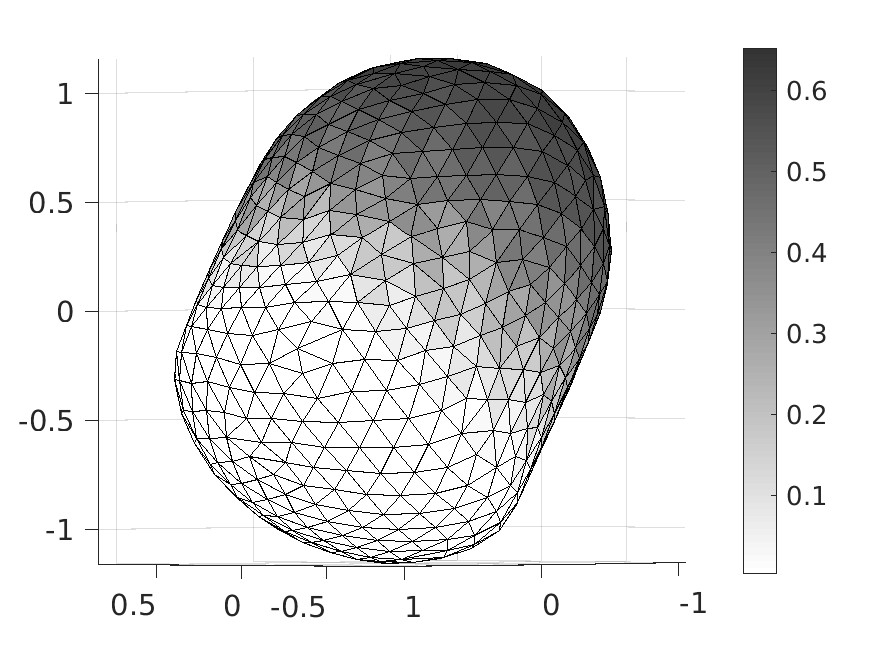}
\includegraphics[trim = {1cm 0 1cm 0},width = 0.95\textwidth]{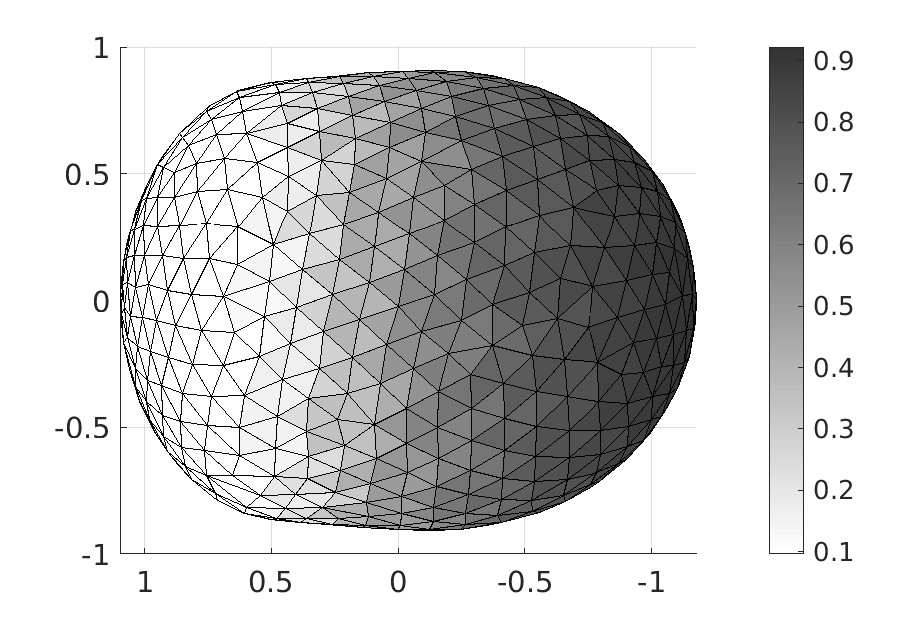}
\includegraphics[trim = {1cm 0 1cm 0},width = 0.95\textwidth]{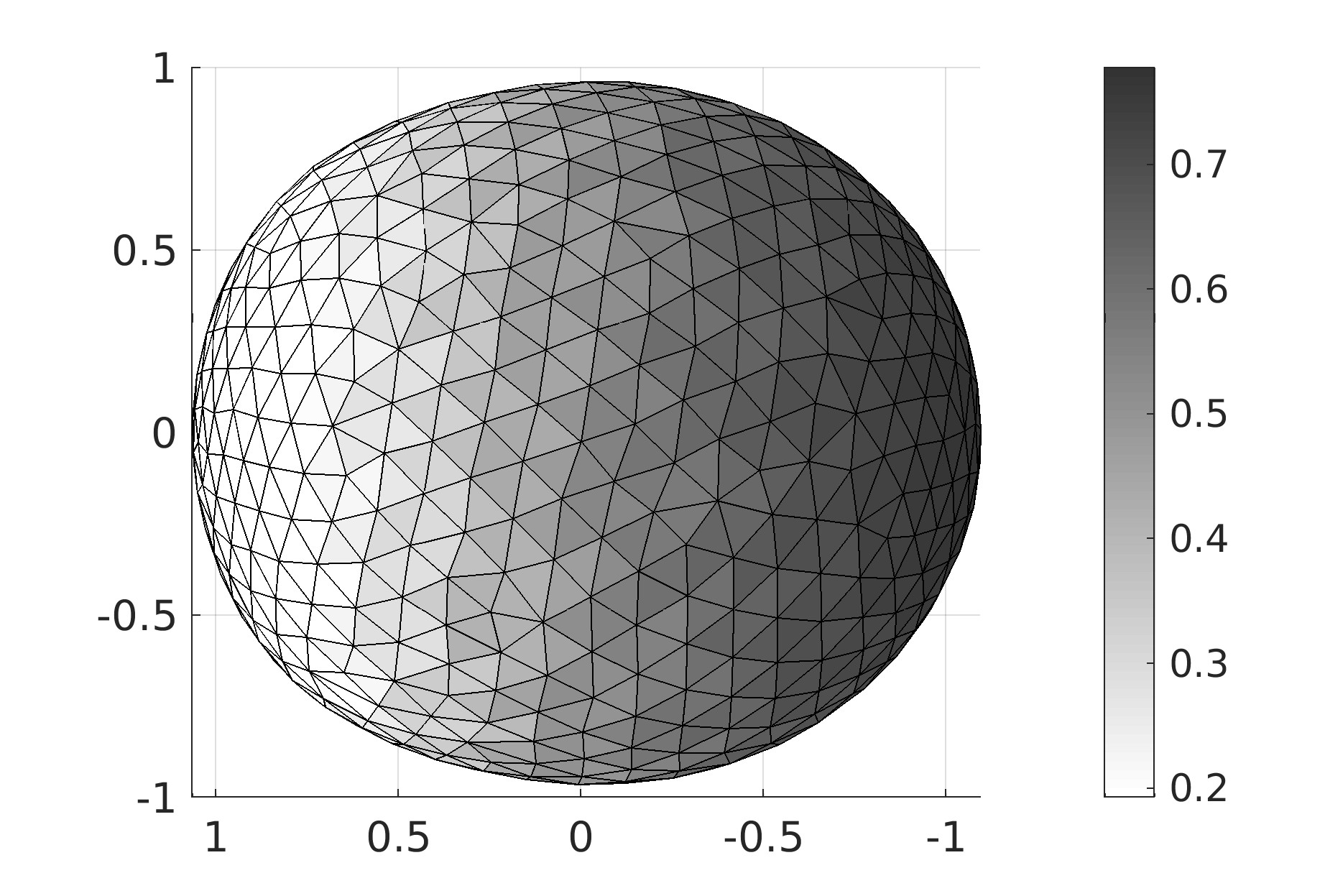}
\includegraphics[trim = {1cm 0 1cm 0},width = 0.95\textwidth]{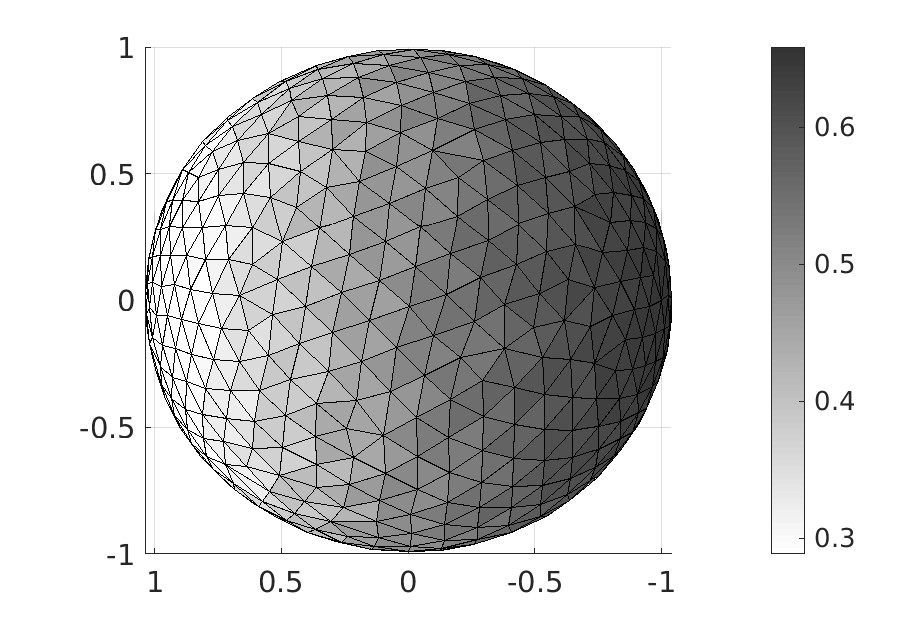}
\includegraphics[trim = {1cm 0 1cm 0},width = 0.95\textwidth]{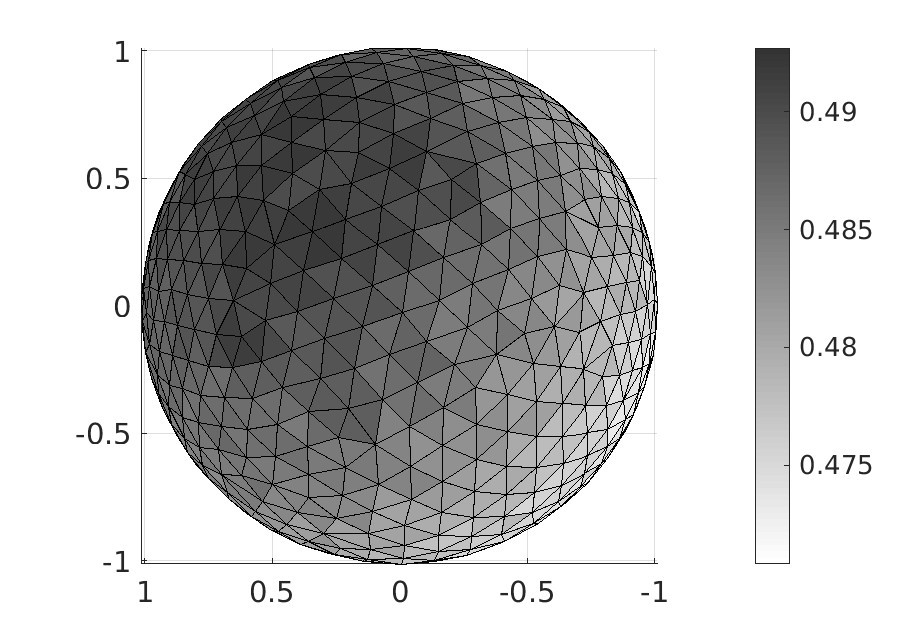}
\end{center}
\end{minipage}
\hfill\vline\hfill
\begin{minipage}{0.3\textwidth}
\begin{center}
\includegraphics[trim = {1cm 0 1cm 0},width = 0.95\textwidth]{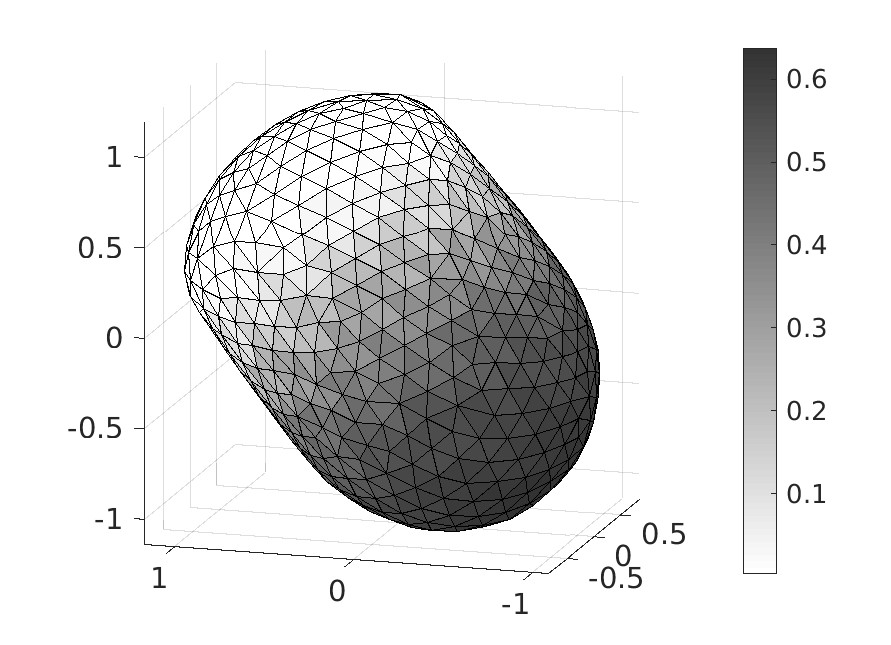}
\includegraphics[trim = {1cm 0 1cm 0},width = 0.95\textwidth]{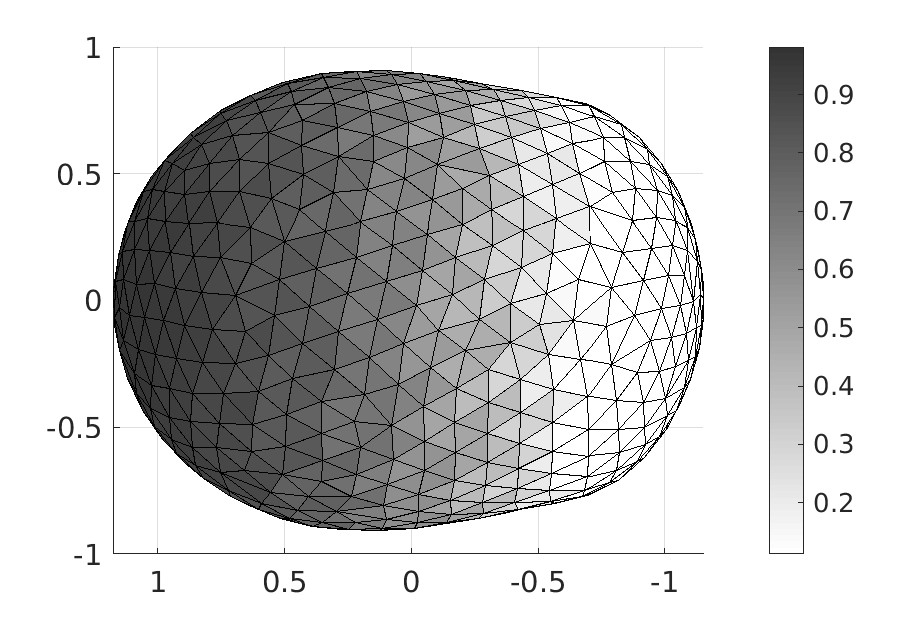}
\includegraphics[trim = {1cm 0 1cm 0},width = 0.95\textwidth]{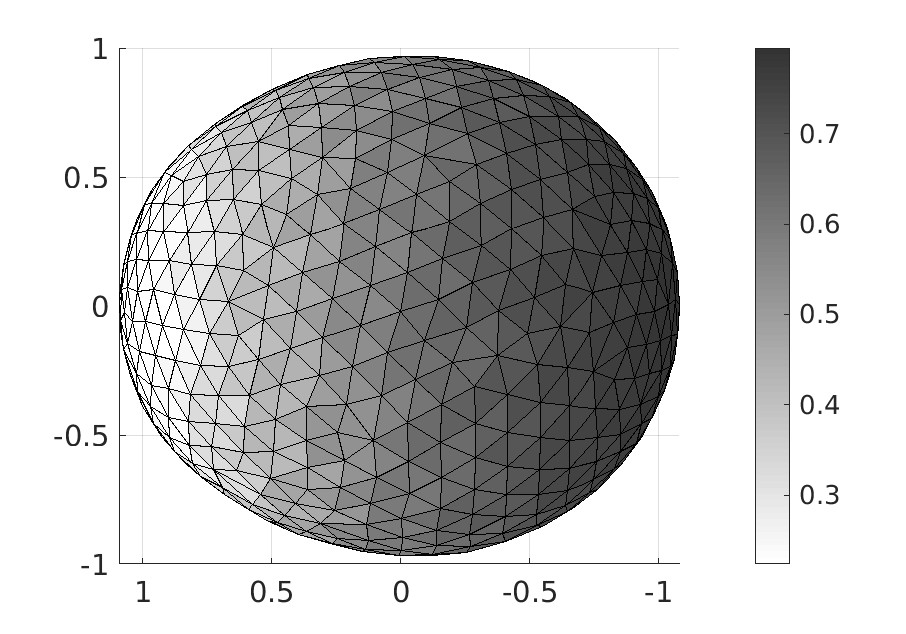}
\includegraphics[trim = {1cm 0 1cm 0},width = 0.95\textwidth]{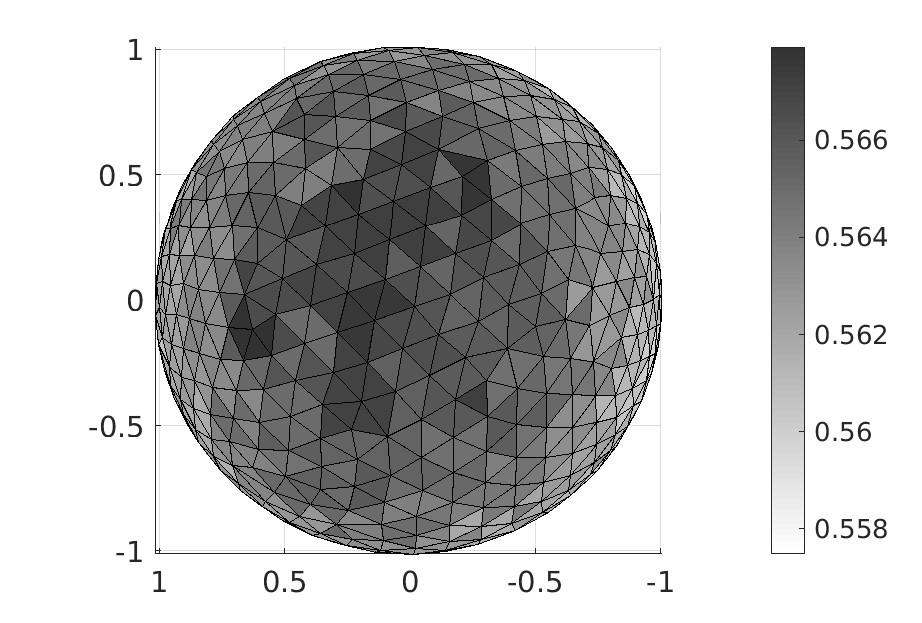}
\includegraphics[trim = {1cm 0 1cm 0},width =0.95\textwidth]{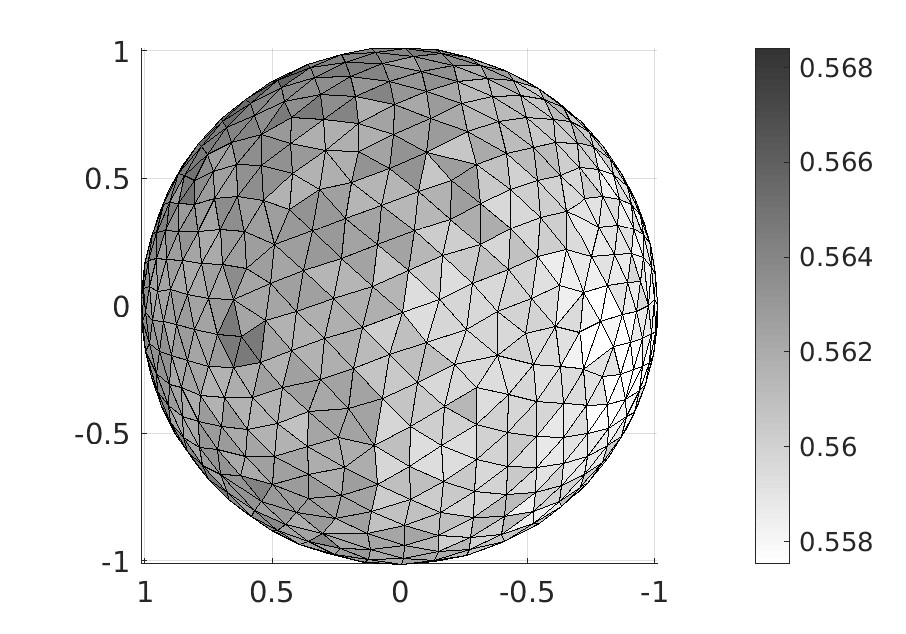}
\end{center}
\end{minipage}
\hfill\vline\hfill
\begin{minipage}{0.3\textwidth}
\begin{center}
\includegraphics[trim = {1cm 0 1cm 0},width = 0.95\textwidth]{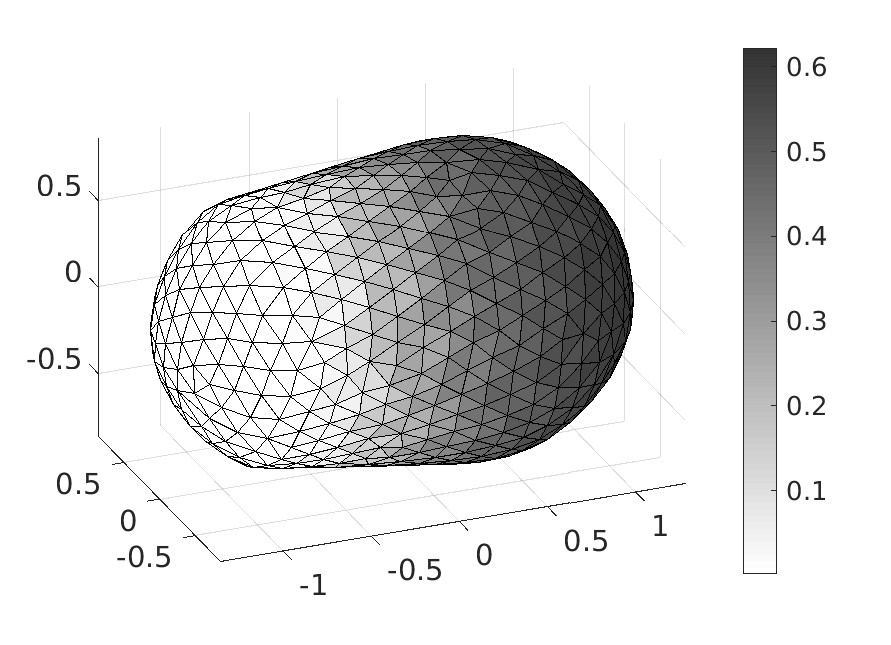}
\includegraphics[trim = {1cm 0 1cm 0},width = 0.95\textwidth]{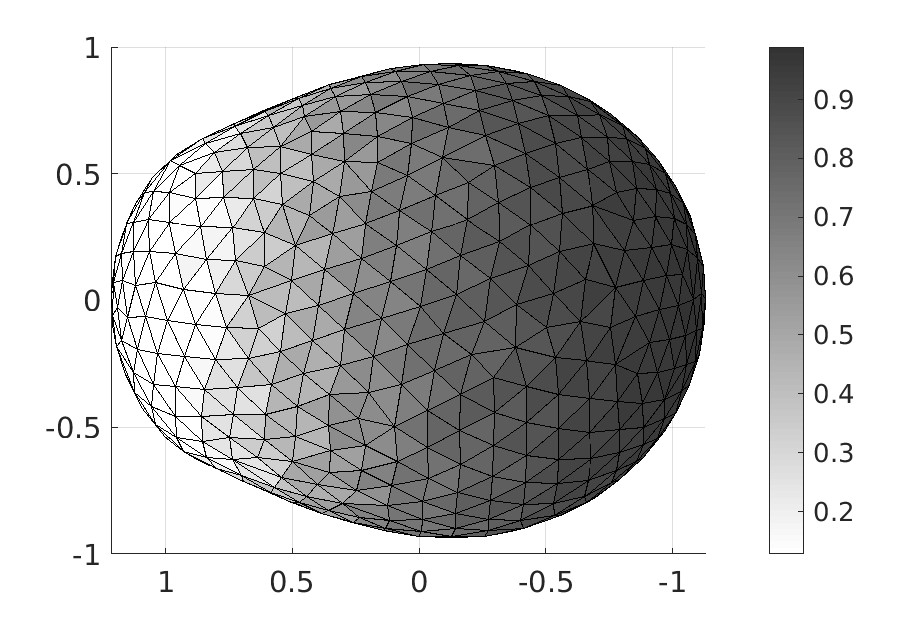}
\includegraphics[trim = {1cm 0 1cm 0},width = 0.95\textwidth]{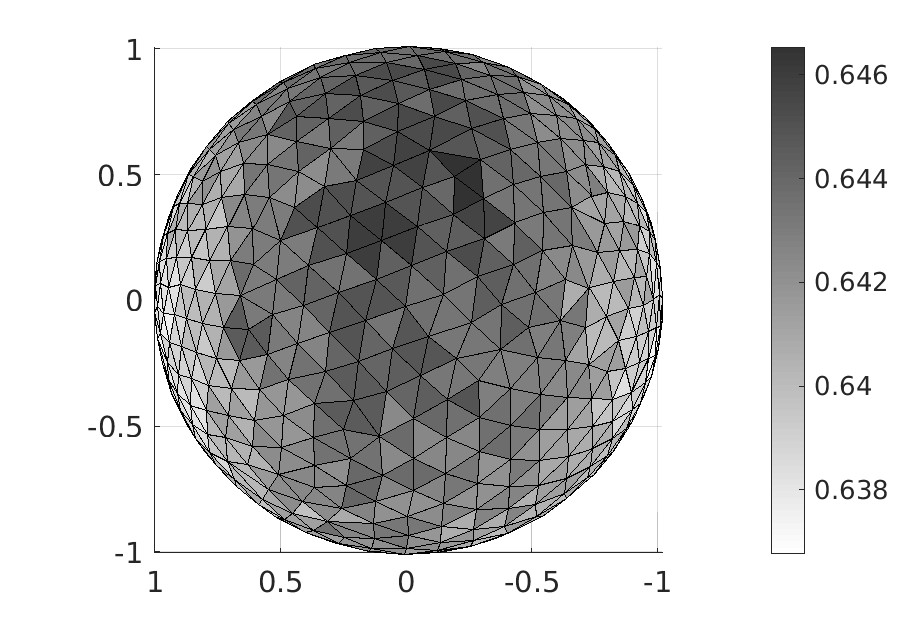}
\includegraphics[trim = {1cm 0 1cm 0},width =0.95\textwidth]{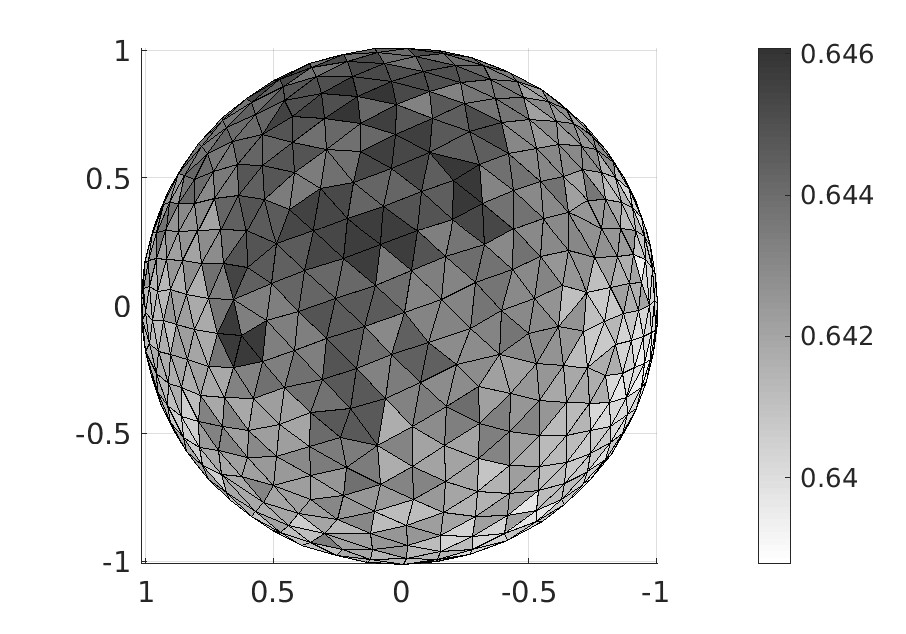}
 \includegraphics[trim = {1cm 0 1cm 0},width =0.95\textwidth]{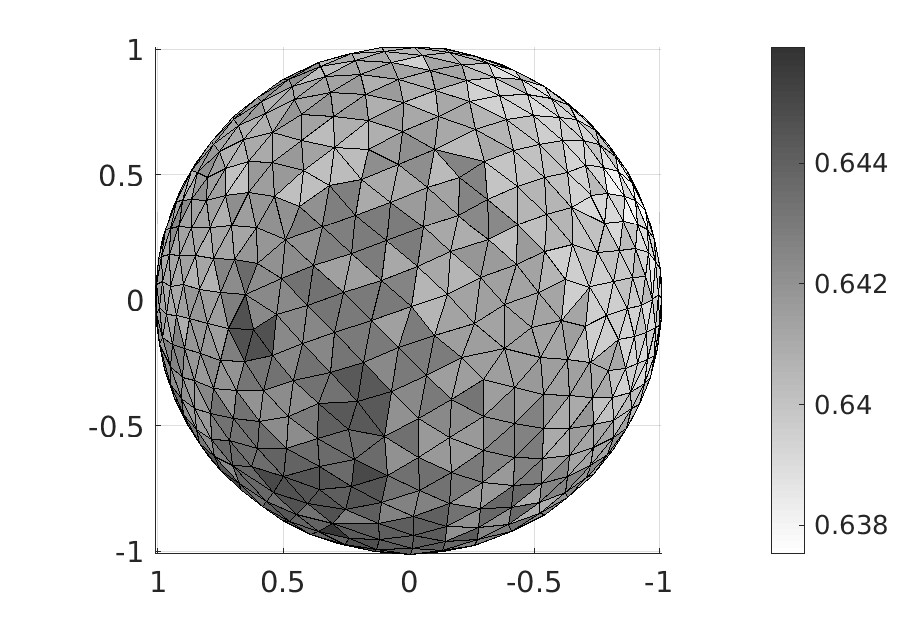}
\end{center}
\end{minipage}
\caption{Approximated optimal domains for $\widehat{\lambda}_{\widehat{m},\ell_{\min}}$ for $\widehat{m}\in\{6,7,8\}$ (left to right) and $\ell_{\min}=(\widehat{m}/|\partial B_1(0)|)q$ for $q\in[0,1)$, evaluated for $q=i/5$, $i=0,\dots,4$ (top to bottom), with initial domain approximated on a mesh with maximal mesh size $h=2^{-3}$. With increasing $\ell_{\min}$, the optimal domains transform from the asymmetric domain to the ball, and the observed kink in the boundary smoothens. The optimal domains appear to be rotationally symmetric. Symmetry breaking in the optimal domains and optimal insulation can be observed for smaller values of $\widehat{m}$ and $\ell_{\min}$. For the value $i = 5$ the optimal domains are balls with constant distribution and are not shown here.}\label{rein:ex:so_3dm2}
\end{figure}

\color{black}

\clearpage
\small

\clearpage

\printbibliography[title = References]

S\"oren Bartels:\quad{\tt bartels@mathematik.uni-freiburg.de}\\
Giuseppe Buttazzo:\quad{\tt giuseppe.buttazzo@unipi.it}\\
Hedwig Keller:\quad{\tt hedwig.keller@mathematik.uni-freiburg.de}

\end{document}